\documentclass[12pt,reqno]{amsart}

\usepackage{amsmath}
\usepackage{amsthm,amscd,amsfonts,amssymb,epic,eepic,rotating}
\usepackage[PostScript=dvips,balance,midshaft]{diagrams}

\swapnumbers  

\newtheorem{thm}[subsection]{Theorem}
\newtheorem{cor}[subsection]{Corollary}
\newtheorem{lem}[subsection]{Lemma}
\newtheorem{prop}[subsection]{Proposition}
\newcounter{remarkscounter}
\newenvironment{remarks}
{\begin{list}{{\rm(\arabic{remarkscounter})} }{\usecounter{remarkscounter}

\setlength{\labelsep}{\fill} \setlength{\leftmargin}{0pt}
\setlength{\itemindent}{\fill}
\setlength{\labelwidth}{\fill} \setlength{\listparindent}{0pt}}}
{\end{list}}

\makeatletter 
\renewcommand\section{\@startsection {section}{1}{\z@}%
                                   {-3.25ex \@plus -1ex \@minus -.2ex}%
                                   {1.8ex \@plus.2ex}%
                                   {\centering\normalfont\normalsize\bfseries}}%
\renewcommand\subsection{\@startsection {subsection}{1}{\z@}%
                                   {-2.0ex \@plus -1ex \@minus -.2ex}%
                                   {-1.8ex \@plus.2ex}%
                                   {\normalfont\normalsize\bfseries}}%
\makeatother

\renewcommand\mod[1]{\ (\mathop{\rm mod}#1)}
\theoremstyle{definition}

\newcommand{\quash}[1]{}

\numberwithin{equation}{subsection}

\newcommand{\image}{\mathop{\rm Image}}

\newcommand{\Cech}{\v{C}ech\ }
\renewcommand{\H}{\check{H}}
\newcommand{\Cr}{{\sf cr}}
\newcommand{\LL}{\left\{ \L,\L \right\}}
\newcommand{\RZ}{\mathbb R/2\mathbb Z}

\newcommand{\bs}{\boldsymbol{\sim}}
\newcommand{\lp}{\left(} \newcommand{\rp}{\right)}
\newcommand{\grad}{\mathop{grad}}

\newcommand{\A}{\mathcal A}
\newcommand{\F}{\L \times_M \L}
\newcommand{\ev}{{\mathbf{ev}}}
\renewcommand{\L}{\Lambda}

\newcommand{\half}{\frac{1}{2}}
\newcommand{\M}{\mathcal M}
\newcommand{\Z}{\mathbb Z}

\newcommand{\rh}{\theta_{\half\to s}}
\newcommand{\hr}{\theta_{s\to\half}}
\renewcommand{\r}{\frac{r}{r+1}}

\newcommand{\ov}[1]{\overline{#1}}
\newcommand{\Ty}{\Omega}

\newcommand{\scs}{{\tiny{\heartsuit}}}

\begin{document}

\newarrow{Equals} =====
\newarrow{Implies} ===={=>}
\newarrow{Onto} ----{>>}
\newarrow{Into} C--->
\newarrow{Dotsto} ....>

\setcounter{tocdepth}{1}
\title{Loop products and closed geodesics}
\author{Mark Goresky${}^1$}\thanks{1.  School of Mathematics, Institute for Advanced Study,
Princeton N.J.  Research partially supported by DARPA grant \# HR0011-04-1-0031.}
\author{Nancy Hingston${}^2$}\thanks{2.  Dept. of Mathematics, College of New Jersey,
Ewing N.J.}

\begin{abstract}
{\tiny The critical points of the length function on the free loop space
$\L(M)$ of a compact Riemannian manifold $M$ are the closed geodesics on $M.$  The
length function gives a filtration of the homology of $\L(M)$ and we show that
the Chas-Sullivan product
\begin{diagram}[size=2em]
H_i(\L) \times H_j(\L) &\rTo^{*}& H_{i+j-n}(\L)
\end{diagram}
is compatible with this filtration.  We obtain a very simple expression for the
associated graded homology ring $Gr H_*(\L(M))$ when all geodesics are closed, or
when all geodesics are nondegenerate.  We also construct a new
but related cohomology product
\begin{diagram}[size=2em]
H^i(\L,\L_0) \times H^j(\L,\L_0) &\rTo^{\circledast}& H^{i+j+n-1}(\L,\L_0)
\end{diagram}
(where $\L_0 = M$ is the constant loops), also compatible with the length
filtration, and we obtain a similar expression for the ring $Gr H^*(\L,\L_0)$
in these two cases.  The non-vanishing of products $\sigma^{*n} \in H_*(\L)$
and $\tau^{\circledast n} \in H^*(\L,\L_0)$ is shown to be related R. Bott's
analysis of the rate at which the Morse index grows when a geodesic is iterated.}
\end{abstract}
\maketitle
\vspace{-.6in}
{\footnotesize\tableofcontents}   
\section{Introduction}
\subsection{}
Let $M$ be a smooth compact manifold without boundary.
In \cite{CS}, M. Chas and D. Sullivan constructed a new product structure
\begin{equation}\label{eqn-product-intro}
\begin{diagram}[size=2em]
H_i(\L) \times H_j(\L) &\rTo^{*}& H_{i+j-n}(\L)
\end{diagram}\end{equation}
on the homology $H_*(\L)$ of the free loop space $\L$ of $M.$  In
\cite{Cohen} it was shown that this product is a homotopy invariant of the
underlying manifold $M.$  In contrast, the closed
geodesics on $M$ depend on the choice of a Riemannian
metric, which we now fix.  In this paper we investigate the interaction
beween the Chas-Sullivan product on $\L$ and the energy function, or 
rather, its square root,
\[ F(\alpha) = \sqrt{E(\alpha)} = \left( \int_0^1|\alpha'(t)|^2dt \right)^{1/2},\]
whose critical points are exactly the closed geodesics.
For any $a$, $0 \le a \le \infty$ we denote by
\begin{equation}\label{eqn-notation}
 \Lambda^{\le a},\ \Lambda^{>a}, \Lambda^{=a}, \Lambda^{(a,b]}\end{equation}
those loops $\alpha \in \Lambda$ such that $F(\alpha) \le a,$
$F(\alpha) > a$, $F(\alpha) = a$, $a < F(\alpha) \le b,$ etc.  (When
$a = \infty$ we set $\L^{<a} = \L^{\le a} = \L$.) 
In this paper we will use homology $H_*(\L^{\le a};G)$
with coefficients in the ring $G = \Z$ if $M$ is orientable and $G = \Z/(2)$
otherwise.   In \S \ref{sec-CS-product} we prove the following.
\begin{thm}  The Chas-Sullivan product extends to a family of products\footnote{
Here, $\H_i(\L^{\le a})$ denotes \Cech homology.  In Lemma \ref{lem-cech} we
show that the singular and \Cech homology agree if $0 \le a \le \infty$ is a
regular value or if it is a nondegenerate critical value in the sense of Bott.}
\begin{align}
\H_i(\L^{\le a}) \times \H_j(\L^{\le b}) &\overset{*}{\longrightarrow}
\H_{i+j-n}(\L^{\le a+b})\notag \\
\H_i(\L^{\le a},\L^{\le a'}) \times \H_j(\L^{\le b}, \L^{\le b'})
&\overset{*}{\longrightarrow}
\H_{i+j-n}(\L^{\le a+b}, \L^{\le \max(a+b',a'+b)})\notag \\
\H_i(\L^{\le a},\L^{<a}) \times \H_j(\L^{\le b},\L^{<b})
&\overset{*}{\longrightarrow} \H_{i+j-n}(\L^{\le a+b}, \L^{<a+b})
\label{eqn-level-homology-product}
\end{align}
whenever $ 0 \le a' < a \le \infty$ and $0 \le b' < b \le \infty.$
These products are compatible with respect to the natural inclusions
$\L^{\le c'} \to \L^{\le c}$ whenever $c' \le c.$\end{thm}
\noindent
We refer to $\H_i(\L^{\le a}, \L^{<a})$ as the {\em level homology} group, or
the {\em homology at level $a$}, with its associated {\em level homology product} 
(\ref{eqn-level-homology-product}).
It is zero unless $a$ is a critical value of $F.$

\subsection{}
In \S \ref{sec-cohomology-products} we consider analogous products in the
cohomology of the free loop space.  (We discuss the cup product briefly in 
\S \ref{subsec-cup}.)  It is possible to mimic the construction of
the Chas-Sullivan product, word for word, in cohomology, but this results in a
trivial product, cf. \S \ref{subsec-first-cohomology-product}.  
However by utilizing a certain one parameter family of reparametrizations, 
it is possible to construct a nontrivial product in cohomology.
\begin{thm}   Let $0 \le a' < a \le \infty$ and $0 \le b' <b < \infty.$
There is a family of products
\begin{align}
H^i(\L,\L_0) \times H^j(\L,\L_0) &\overset{\circledast}{\longrightarrow}
H^{i+j+n-1}(\L,\L_0)\label{eqn-level-cohomology1} \\
\H^i(\L^{\le a}, \L^{\le a'}) \times \H^j(\L^{\le b}, \L^{\le b'})
&\overset{\circledast}{\longrightarrow} \H^{i+j+n-1}(\L^{\le \min(a+b',a'+b)},
\L^{\le a'+b'})\notag \\
\H^i(\L^{\le a}, \L^{<a}) \times \H^j(\L^{\le b}, \L^{<b})
&\overset{\circledast}{\longrightarrow} \H^{i+j+n-1}(\L^{\le a+b}, \L^{<a+b})
\label{eqn-level-cohomology-product}
\end{align}
which are associative and (sign-)commutative, and are compatible with the 
homomorphisms induced by the inclusions $\L^{\le c'}
\to \L^{\le c}$ whenever $c'<c.$ The product (\ref{eqn-level-cohomology1}) is
independent of the Riemannian metric.\end{thm}
The same construction gives (cf. \S \ref{subsec-based}) a (possibly noncommutative)
product $\circledast$ on the cohomology of the based loop space $\Omega$ such
that $h^*(a\circledast b) = h^*(a) \circledast h^*(b)$ where $a,b\in H^*(\L)$
and $h:\Omega \to \L$ denotes the inclusion.  In \S 
\ref{sec-based-loop-coproduct-nontrivial} we calculate some
non-zero examples of this product.

\subsection{}
If the ring $\left(H^*(\L,\L_0),\circledast\right)$ is finitely generated then the 
existence of the product $\circledast$ is already enough
to answer a question of Y. Eliashberg, cf. \S \ref{subsec-Eliashberg}:  the maximal
degree of an ``essential'' homology class of level $\le t$ can grow at most linearly
with $t.$

\subsection{}  There is a well-known isomorphism between the Floer homology
of the cotangent bundle of $M$ and the homology of the free loop space of $M$,
which transforms the pair-of-pants product into the Chas-Sullivan product
on homology, see \cite{AS1, AS2, SW, Viterbo,CHV}.  The cohomology product described
above should therefore correspond to some geometrically defined product on the
Floer {\em cohomology}; it would be interesting to see an explicit construction of
this product. (The obvious candidate would be some 1-parameter variation of the 
coproduct on homology given by the upside-down pair of pants.)  It would also be
interesting to compare the cohomology product described above with the coproduct
in homology that is outlined in \cite{Sullivan}. (For odd dimensional spheres the
product in \cite{Sullivan} is zero while the $\circledast$ product is non-zero.) 

\subsection{}
The {\em critical value} (see \S \ref{sec-cr-def}) of a homology class 
$0\ne\eta \in H_i(\L)$ is defined to be
\begin{equation}\label{eqn-intro-cr}
 \Cr(\eta) = \inf \left\{ a \in \mathbb R:\ \eta \text{ is supported on }
\L^{\le a} \right\}.\end{equation}
The critical value of a cohomology class
$0\ne \alpha \in H^*(\L,\L_0)$ is defined to be
\[ \Cr(\alpha) = \sup \left\{ a \in \mathbb R:\ \alpha \text{ is supported on }
\L^{\ge a} \right\}.\]
(These are necessarily critical values of $F$.) 
In Proposition \ref{prop-product-extended} and Proposition 
\ref{prop-coprod-critical} we show 
that the products  $*$ and $\circledast$ satisfy the following relations:
\begin{align*}
\Cr(\alpha*\beta) &\le \Cr(\alpha) + \Cr(\beta)\ \text{ for all } \alpha,\beta
\in H_*(\L)\\
\Cr(\alpha \circledast \beta) &\ge \Cr(\alpha) + \Cr(\beta)\ \text{ for all }
\alpha,\beta \in H^*(\L,\L_0).
\end{align*}

\subsection{}
A homology class $\eta \in H_*(\L)$ is said to be {\em level nilpotent}
if $\Cr(\eta^{*N}) < N\Cr(\eta)$ for some $N>1.$  A cohomology class
$\alpha\in H^*(\L,\L_0)$ is level nilpotent if 
$\Cr(\alpha^{\circledast N}) > N\Cr(\alpha)$
for some $N>1.$  There are analogous notions in level homology and cohomology:
A homology (resp. cohomology) class $\eta$ in $\H(\L^{\le a},\L^{<a})$ (where
$\H$ denotes homology, resp. cohomology)
is said to be level-nilpotent if some power vanishes: 
$\eta^{*N}=0$ (resp. $\eta^{\circledast N}=0$) in $\H(\L^{\le Na}, \L^{<Na}).$ 
In \S \ref{sec-nilpotence} and \S \ref{sec-nilpotent-cohomology} we prove:
\begin{thm} If all closed geodesics on $M$ are nondegenerate then every 
homology class in $H_*(\L),$ every cohomology class in $H^*(\L,\L_0),$ 
every level homology class and every level cohomology 
class\footnote{see previous footnote} in $H(\L^{\le a}, \L^{<a})$ is 
level-nilpotent (for all $a \in \mathbb R$).   \end{thm}

\subsection{}
On the other hand, non-nilpotent classes exist when all geodesics are closed.
Suppose $E$ is the energy function of a metric in which all
geodesics on $M$ are closed, simply periodic, and have the same prime
length $\ell$, as defined in \S \ref{subsec-def-all-closed}.  The
critical values of $F = \sqrt{E}$ are the (non-negative) integer 
multiples of $\ell.$  The set of critical points  with
critical value $r\ell$ ($r \ge 1$) form a (Morse-Bott) nondegenerate 
critical submanifold $\Sigma_r\subset \L$ that is diffeomorphic to 
the unit sphere bundle $SM$ by the mapping $\alpha \mapsto \alpha'(0)/r\ell.$

Let $\lambda_r$ be the Morse index of any geodesic of 
length $r\ell.$  Let $h = \lambda_1 + 2n-1$ where $n = \dim(M).$  
Then $H_i(\L^{\le \ell})=0$ for $i >h.$ Let
\[ \Theta \in H_h(\L^{\le \ell};G) \cong G\]
be a generator of the top degree homology group.  
In \S \ref{sec-all-closed} and Corollary \ref{cor-level-ring} we prove:

\begin{thm}  The $r$-fold Chas-Sullivan product
\[ \Theta^{*r} \in H_{\lambda_r +2n-1}(\L^{\le r\ell}, \L^{<r\ell};G) \cong G\]
generates the top degree homology at the level $r\ell,$ and more generally,
the Chas-Sullivan product with $\Theta$ induces an isomorphism
\[ H_i(\L^{\le a}, \L^{<a}) \to H_{i+h-n}(\L^{\le a+\ell}, \L^{<a+\ell})\]
for all degrees $i$ and for all level values $a.$  The energy $E$ determines a
filtration $0=I_0\subset I_1 \subset \cdots \subset H_*(\L,\L_0)$ such that
$I_j*I_k \subset I_{j+k}.$  The associated graded ring is isomorphic (with
degree shifts) to the ring
\begin{equation}\label{eqn-ring}
 \text{Gr}_IH_*(\L,\L_0) \cong H_*(SM)[T]_{\ge 1}\end{equation}
of polynomials of degree $\ge 1,$ where $H_*(SM)$ denotes the homology 
(intersection) ring of $SM.$
\end{thm}
\noindent
The full Chas-Sullivan ring $H_*(\L)$ was computed by R. Cohen, J. Jones, and
J. Yan \cite{CJY} for spheres and projective spaces.
The relatively simple formula (\ref{eqn-ring}) is compatible with their computation.
It seems likely that there may be other results along these lines when the
Riemannian metric has large sets of closed geodesics.

\subsection{}
In \S \ref{sec-coH-closed} and Corollary \ref{cor-cohomology-graded}
we prove the analogous result for cohomology.
Suppose that all geodesics on $M$ are closed, simply periodic,
and have the same prime length $\ell.$  Then $H^i(\L^{\le \ell}, \L^{=0})
=0$ for $i < \lambda_1.$  Let
\[ \Omega \in H^{\lambda_1}(\L^{\le \ell}, \L^{=0};G) \cong G\]
be a generator of the lowest degree cohomology group $G = \Z
\text{ or } \Z/(2).$

\begin{thm} The $r$-fold product
\[ \Omega^{\circledast r} \in H^{\lambda_r}(\L^{\le r\ell}, \L^{<r\ell})
\cong G\]
generates the lowest degree cohomology class at level $r\ell$ and more generally,
the product with $\Omega$ induces an isomorphism
\[ H^i(\L^{\le a},\L^{<a}) \to H^{i+h-n}(\L^{\le a+\ell}, \L^{<a+\ell})\]
for all degrees $i$ and all level values $a.$  Moreover, the energy
induces a filtration
\[ H^i(\L,\L_0) = I^0 \supset I^1 \supset I^2 \supset \cdots \]
by ideals such that $I^j \circledast I^k \subset I^{j+k}.$  The associated
graded ring $Gr^IH^*(\L,\L_0)$ is isomorphic (with degree shifts) to the ring,
\[ H^*(SM)[T]_{\ge 1}\]
where $H^*(SM)$ denotes the cohomology ring of $SM.$
\end{thm}

\subsection{Counting closed geodesics}\label{intro-counting}
By \cite{Vigue} if $M$ is a compact, simply connected
Riemannian manifold whose cohomology algebra $H^*(M;\mathbb Q)$ cannot be
generated by a single element then the Betti numbers of $\L$ form an
unbounded sequence, whence by \cite{Gromoll}, the manifold $M$ admits
infinitely many prime closed geodesics.  This result leaves open the case
of spheres and projective spaces (among others).

It is known \cite{Bangert, Franks, Nancy2} that
any Riemannian metric on $S^2$ has infinitely many prime closed geodesics,
and it is conjectured that the same holds for any Riemannian sphere or
projective space of dimension $n>2.$ (But see \cite{Ziller} for
examples of Finsler metrics on $S^2$ with finitely many prime closed
geodesics, all of which are nondegenerate.)  It should, in principle, be possible
to count the number of closed geodesics using Morse theory on the
free loop space $\L,$ but each prime geodesic $\gamma$ is associated with
infinitely many critical points, corresponding to the iterates $\gamma^m.$
So it would be useful to have an operation on $H_*(\L)$ that corresponds
to the iteration of closed geodesics.  

If $\lambda_1$ is the Morse index
of a prime closed geodesic $\gamma$ of length $\ell$, then by \cite{Bott1}
(cf. Proposition \ref{prop-index-iterates}), the Morse
index $\lambda_m$ of the iterate $\gamma^m$ can be anywhere between
$m\lambda_1 -(m-1)(n-1)$ and $m\lambda_1 + (m-1)(n-1).$   For
nondegenerate critical points, the
Chas-Sullivan product $[\ov\gamma]*\cdots * [\ov\gamma]$ is non-zero
exactly when (cf. Theorem \ref{prop-level-products}) the index growth
is minimal (i.e., when $\lambda_m =  m\lambda_1 - (m-1)(n-1)$).  Here,
$[\ov\gamma] \in H_{\lambda_1+1} (\L^{\le \ell}, \L^{<\ell})$ is the level
homology class represented by the $S^1$ saturation of $\gamma$.  The
Pontrjagin product (on the level homology of the based loop space) is
zero unless $\lambda_m = m\lambda_1.$  The level cohomology product 
$\circledast$ is non-zero when the index growth
is maximal (cf.~Proposition \ref{prop-index-iterates}).

\subsection{Related products}  In \cite{CS}, Chas and Sullivan also
defined a Lie algebra product $\left\{\alpha,\beta\right\}$ on the homology
$H_*(\L)$ of the free loop space.  In \S \ref{sec-bracket} we combine
their ideas with the construction of the cohomology product $\circledast$
to produce a Lie algebra product on the cohomology $H^*(\L,\L_0)$  In \S 
\ref{subsec-bracket-nondegenerate} we use the calculations described in
\S \ref{intro-counting} to show that these products are sometimes non-zero.
Also following \cite{CS}, we construct products 
on the $T=S^1$-equivariant cohomology $H^*_T(\L,\L_0).$
\subsection{} Several of the proofs in this paper require technical results
that are well-known to experts (in different fields) but are difficult to 
find in the literature.  These technical tools are described in the Appendices,
as are the (tedious) proofs of Proposition \ref{prop-cohomology-commutative} 
and Theorem \ref{thm-X}.  The collection of products and their
definitions can be rather confusing, so in each case we have created a ``boxed''
\boxed{diagram} which gives a concise way to think about the product.

\subsection{Acknowledgments} We wish to thank Fred Cohen and Dennis Sullivan
for a number of valuable conversations.   We are pleased to thank Matthias
Schwarz, who long ago pointed out to the second author
the concept and importance of nilpotence of products, 
in the context of the pair-of-pants product.  The authors would 
also like to thank the Institute for Advanced Study in Princeton N.J. for its 
hospitality and support during the academic year 2005-2006.  The second 
author was supported in 2005-06
by a grant to the Institute for Advanced Study by the Bell Companies
Fellowship.  The first author is grateful to the Defense Advanced Research
Projects Agency for its support from grant number HR0011-04-1-0031.

\section{The free loop space}\label{sec-freeloop}
\subsection{}\label{subsec-free-loop}
Throughout this paper, $M$ denotes an $n$ dimensional
smooth connected compact Riemannian manifold.  Let $\alpha:[a,b] \to M$
be a piecewise smooth curve.  Its {\em length} and {\em energy} are given
by
\[ L(\alpha) = \int_a^b|\alpha'(t)|dt\ \text{ and }\ E(\alpha) = \int_a^b
|\alpha'(t)|^2dt.\]
The Cauchy-Schwartz inequality says that $L(\alpha)^2 \le (b-a)E(\alpha).$
The formulas work out most simply if we use the Morse function $F(\alpha) =
\sqrt{E(\alpha)}.$

The {\em free loop space} $\Lambda$ consists of $H^1$ mappings $\alpha:[0,1] 
\to M$ such that $\alpha(0) = \alpha(1).$  It admits the structure of a
Hilbert manifold (\cite{Klingenberg}, \cite{Chataur}), so it is a complete
metric space, hence paracompact and Hausdorff.  The loop space $\Lambda$ is homotopy
equivalent to the Frechet manifold of smooth loops $\beta:S^1 \to M.$
Denote by $\L_0 =\L^{\le 0} \cong M$ the space of constant loops.

The energy of a loop depends on its parametrization; the length does not.
Thus, $L(\alpha) \le F(\alpha)$ for all $\alpha \in \L,$
with equality if and only if the loop is {\em parametrized proportionally to arc
length} (abbreviated PPAL), meaning that $|\alpha'(t)|$ is constant.  Every
geodesic is, by definition, parametrized proportionally to arclength. A 
loop $\alpha \in \L$ is a critical point of $F$ if and only if $\alpha$ 
is a closed geodesic.  Let $\Sigma \subset \L$ be the set of critical 
points of $F,$ and set $\Sigma^{=a} = \Sigma \cap \L^{=a}.$

The index and nullity of the
critical points of $F$ coincide with those of $E.$ Recall (for example, from
\cite{Klingenberg} p.~57) that the index of a closed geodesic $\gamma$ is
the dimension of a maximal subspace of $T_{\gamma}(\L)$ on which the Hessian
$d^2F(\gamma)$ is negative definite, and the nullity of $\gamma$ is
$\dim(T_{\gamma}^{0}\Lambda)-1$ where $T_{\gamma}^{0}\Lambda$ is the null
space of the Hessian $d^{2}F(\gamma).$ The $-1$ is
incorporated to account for the fact that every closed geodesic $\gamma$
occurs in an $O(2)$ orbit of closed geodesics. The critical point $\gamma$ is
\emph{nondegenerate} if this single orbit is a nondegenerate Morse-Bott critical
submanifold, or equivalently, if the nullity is zero. A number $a \in 
\mathbb R$ is a {\em nondegenerate critical value} if the critical set 
$\Sigma^{=a}$ consists of nondegenerate critical orbits.  In this case 
there are finitely many critical orbits in $\Sigma^{=a}$ and the number
$a\in \mathbb R$ is an isolated critical value. In \S \ref{sec-all-closed} 
we will encounter a critical set $\Sigma^{=a}$ of dimension $>1$ 
(consisting of geodesics with nullity $>0$), which is nondegenerate 
in the sense of Bott.  In this case we say the critical value $a\in 
\mathbb R$ is {\em nondegenerate in the sense of Bott}.  To
distinguish ``nondegenerate" from ``nondegenerate in the sense of Bott", we
will sometimes refer to the former case with the phrase ``isolated 
nondegenerate critical orbit".

Denote by $\A\subset \Lambda$ the subspace of loops parametrized proportionally
to arc length (PPAL).  Then $F(\alpha) = L(\alpha)$ for all $\alpha \in \A.$
We write $\A^{\le a}$ (etc.) for those $\alpha \in \A$ such that $F(\alpha)
\le a,$ cf. equation (\ref{eqn-notation}). The following result is due to 
Anosov \cite{Anosov}.
\begin{prop}\label{prop-PPAL-space}
For all $a \le \infty$ the inclusion $\A^{\le a} \to \L^{\le a}$ is a homotopy
equivalence.  A homotopy inverse is the mapping $A:\L^{\le a} \to \A^{\le a}$
which associates to any path $\alpha$ the same path parametrized proportionally
to arclength, with the same basepoint. It follows that the set of loops of 
{\em length} $\le a$ also has the homotopy type of $\L^{\le a}.$\qed
\end{prop}

\subsection{}\label{subsec-evaluation}
The {\em evaluation mapping}
${\mathbf{ev}}_s: \Lambda \to M$ is given by $\ev_s(\alpha) = \alpha(s).$
The {\em figure eight space} $\mathcal F= \L \times_M \L$ is the pullback
of the diagonal under the mapping
\begin{equation}
\label{eqn-evaluation-mapping}
 \ev_0 \times \ev_0: \Lambda \times \Lambda \to M \times M.
\end{equation}
It consists of composable pairs of loops.
Denote by $\phi_s:\mathcal F \to \L$ the mapping which joins the two loops at
time $s,$ that is,
\[ \phi_s(\alpha,\beta)(t) = \begin{cases}
\alpha(\frac{t}{s}) & \text{for }\ t \le s \\
\beta(\frac{t-s}{1-s}) & \text{for }\ s \le t \le 1 \end{cases}. \]
The mapping $\phi_s$ is one to one.
The energy of the composed loop $\phi_s(\alpha,\beta)$ is
\[ E(\phi_s(\alpha,\beta)) = \frac{E(\alpha)}{s} + \frac{E(\beta)}{1-s}\]
which is minimized when
\begin{equation}\label{eqn-minimizing-time}
s = \sqrt{E(\alpha)}/( \sqrt{E(\alpha)} + \sqrt{E(\beta)}).\end{equation}

\begin{lem}\label{lem-phi} Consider $M=\L_0 \times_M \L_0$ to be a subspace
of $\mathcal F=\L \times_M \L.$  Then the mapping $\phi_{\min}: \mathcal F-M 
\to \L$ defined by $\phi_{\min}(\alpha,\beta) = \phi_s(\alpha,\beta)$ for
\[s =\frac{F(\alpha)}{F(\alpha) + F(\beta)}
\]
extends continuously across $M$, giving a mapping
$\phi_{\min}:\F \to \L$ which is homotopic to the embedding 
$\phi_s:\mathcal F \to \L$ for any $s \in (0,1),$ and which satisfies
\begin{equation}\label{eqn-Fsum}
F(\phi_{\min}(\alpha,\beta)) = F(\alpha) + F(\beta).
\end{equation}
If $\alpha$ and $\beta$ are PPAL then so is $\phi_{\min}(\alpha,\beta).$
\qed
\end{lem}

If $A,B\subset \L$ write $A \times_M B = (A \times B) \cap (\F)$ and
define $A*B = \phi_{\min}(A \times_M B)$ to be the
subset consisting of all composed loops, glued together at the energy-minimizing
time.  Then $\L^{\le a} * \L^{\le b} \subset \L^{\le a+b}.$

\subsection{}\label{subsec-normal-bundle}
By \cite{Chataur} Prop. 2.2.3 or \cite{Baker} Prop. 1.17, the figure eight
space $\mathcal F = \L \times_M \L$ has an $n$ dimensional normal bundle $\nu$ and
tubular neighborhood $N$ in $\L \times \L$ (see \S \ref{subsec-tubes})
because the mapping (\ref{eqn-evaluation-mapping}) is a submersion whose
domain is a Hilbert manifold.  Similarly, for any $a,b \in \mathbb R$ the
space
\[ \mathcal F^{<a,<b} = \left\{ (\alpha,\beta) \in \L^{<a} \times \L^{<b}:\
\alpha(0) = \beta(0) \right\}\]
has a normal bundle and tubular neighborhood in $\L^{<a} \times \L^{<b}$
and the  image $\phi_s(\mathcal F)$ has a normal bundle
and tubular neighborhood in $\Lambda$  because it is the pre-image of the
diagonal $\Delta \subset M \times M$ under the submersion
\[\begin{CD}
(\ev_0,\ev_s):\Lambda  @>>> M \times M. \end{CD} \]

The normal bundle of $\Delta$ in $M\times M$ is non-canonically isomorphic
to the tangent bundle $TM,$ so the normal bundle $\nu$ is orientable
if $M$ is orientable.

\section{The finite dimensional approximation of Morse}\label{sec-FDA}
\subsection{}In this section we recall some standard facts concerning
the finite dimensional approximation $\mathcal M$ to the free loop
space $\Lambda$ of a smooth compact Riemannian manifold $M.$  This
finite dimensional approximation was described
by Morse \cite{Morse} but his description is rather difficult to interpret by
modern standards.  It was clarified by Bott \cite{Bott2}  and further described
by Milnor \cite{Milnor}.  Related finite dimensional models are discussed
in \cite{Bahri}.

Fix $\rho >0$ less than one half the injectivity radius of $M.$
For points $x,y \in M$ which lie at a distance less than $\rho,$
we will write $|x-y|$ for this distance.

\begin{lem}  \label{lem-points-are-close}  Fix $N \ge 1.$
Let $x=(x_0,x_1,\cdots,x_N) \in M^{N+1}.$  Let $\alpha:[0,1] \to M$ be any
piecewise smooth curve such that $\alpha(i/N) = x_i.$  If
$F(\alpha) \le \rho \sqrt{N}$ then $|x_i-x_{i-1}|
\le \rho$ for each $i = 1,2,\cdots, N$ and hence, for each $i$ there
is a unique geodesic segment from $x_{i-1}$ to $x_i.$  If $\gamma=\gamma(x)$
denotes the path obtained by patching these geodesic segments together,
with $\gamma(i/N) = x_i$ then
\[F(\gamma(x)) = \sqrt{N\textstyle{\sum_{i=1}^N} |x_i - x_{i-1}|^2}.\]
\end{lem}
\proof  Let $\alpha_i:\left[ \frac{i-1}{N}, \frac{i}{N}\right] \to M$ denote the
$i$-th segment of the path.  Then $L(\alpha_i)^2 \le E(\alpha_i)/N \le \rho^2.$
Therefore $|x_i-x_{i-1}| \le \rho.$  The energy of the resulting piecewise geodesic
path $\gamma$ is therefore
$E(\gamma)=\Sigma_{i=1}^N E(\gamma_i)= N\Sigma_{i=1}^N |x_i-x_{i-1}|^2.$
\qed

For $N \ge 1$ and $a \in \mathbb R$ let
\[ \M_N^{\le a} = \left\{ (x_0,x_1,\cdots,x_N)\in M^{N+1}:\ x_0 = x_N
\text{ and } F(\gamma(x)) \le a \right\}.\]
According to Lemma \ref{lem-points-are-close}, if $a \le \sqrt{N}\rho$
then we have a well defined mapping
\begin{equation}\label{eqn-iota}
\gamma:\mathcal M_N^{\le a} \hookrightarrow \L.\end{equation}

\begin{prop}\label{prop-fda} Suppose $a \le \sqrt{N}\rho.$  Then the mapping
$F\circ\gamma:\M_N^{\le a}\to \mathbb R$ is smooth and proper.
The restrictions $\gamma:\M^{\le a} \hookrightarrow \L^{\le a}$ and
$\gamma:\M^{<a} \hookrightarrow \L^{<a}$ are homotopy equivalences.
The mapping $\gamma$  identifies the critical points (with values
$\le a$) of $F\circ\gamma$ with the critical points (with values $\le a$)
of $F.$  The Morse index and nullity 
of each critical point are preserved under this identification.
If, in addition, $a$ is a regular value of $F$ or if $a$ is a nondegenerate
critical value of $F$ in the sense of Bott (cf. \S \ref{subsec-free-loop})
then the spaces $\M_N^{\le a}$ and $\L^{\le a}$ have the homotopy types of
finite simplicial complexes.
\end{prop}

\proof
There is a homotopy inverse $h: \L^{\le a} \to \M_N^{\le a}$ which assigns to
any loop $\alpha:[0,1]\to M$ the element $x=(x_0,\cdots,x_N)$ where
$x_i = \alpha(i/N)$ for $0 \le i \le N.$  Since $F(\alpha) \le a,$ Lemma
\ref{lem-points-are-close} implies that $F\circ\gamma(h(\alpha)) \le a.$
The composition $h\circ \gamma$ is the identity.  The composition
$\gamma \circ h:\L^{\le a} \to \L^{\le a}$ is homotopic to the identity:  we describe
a homotopy $H_T$ from $\alpha \in \L^{\le a}$ to $\gamma h(\alpha).$  Given
$T\in [0,1]$ there exists $i$ such that $(i-1)/N \le T \le i/N.$
The homotopy $H_T(\alpha)(t)$ coincides with $\alpha(t)$ for $t \le (i-1)/N.$
It coincides with the piecewise geodesic path $\gamma(\alpha)(t)$ for
$t \ge (i/N).$  For $t$ in the interval $[(i-1)/N,i/N]$ the path $H_T(\alpha)(t)$
agrees with $\alpha $ for $t \le T$ and it is geodesic on $[T,i/N].$
Replacing part of the curve $\alpha$
with a geodesic segment between the same two points does not increase its energy, so
$H_T:\L^{\le a} \times [0,1] \to \L^{\le a}$ is the desired homotopy.

If $a \in \mathbb R$ is a regular value then $\M_N^{\le a}$ is a smooth compact
manifold with boundary, so it can be triangulated, hence $\L^{\le a}$ is
homotopy equivalent to a simplicial complex.

If the critical value $a$ of $F$ is nondegenerate in the sense of Bott, then $a$
is also a (Bott-) nondegenerate critical value of $F\circ\gamma.$  It
is then possible to Whitney stratify $\M_{N}^{\le a}$ so that $\M_{N}^{=a}$ is a
closed union of strata.  The complete argument is standard but technically messy;
here is an outline.  Each connected component of the singular set $S$ of
$F\circ\gamma$ is a stratum. The set $\M_N^{=a} - S$ is another stratum; it is
a manifold because it contains no critical points of $F\circ\gamma.$  Finally,
$\M_N^{<a}$ is the open stratum.
According to the generalized Morse lemma, there exist local coordinate near each
point $x$ in the critical set, with respect to which the function $F\circ\gamma$
has the form $F(\gamma(x))+ \Sigma_{i=1}^r x_i^2 - \Sigma_{i=r+1}^s x_i^2$
(with the last $n-s$ coordinates not appearing in the formula).  Using this,
it is possible to see that the above stratification satisfies the Whitney conditions.

Every Whitney stratified space can be triangulated (\cite{goresky}, \cite{Johnson}), so it
follows that $\M_{N}^{\le a}$ is homeomorphic to a finite simplicial complex, hence
$\L^{\le a}$ is homotopy equivalent to a finite simplicial complex.
\qed

\section{Support, Critical values, and level homology}\label{sec-cr-def}
\subsection{} Continue with the notation $M,\L,F,\Sigma$ of \S \ref{sec-freeloop}.
Let $G$ be an Abelian group. A class $\alpha \in \H_i(\L;G)$ is {\em supported} 
(in \Cech homology, cf. Appendix \ref{sec-homology}) on a closed set 
$A \subset \L$ if there is a class $\alpha' \in
\H_i(A;G)$ such that $\alpha = i_*(\alpha')$ where $i:A\to \L$ is the 
inclusion.   This implies that $\alpha \mapsto 0 \in \H_i(\L,A;G)$ but 
the converse does not necessarily hold. Define the {\em critical value}
${\sf cr}(\alpha)$ to be the infemum
\begin{align*}
 {\sf cr}(\alpha) &= \inf \left\{ a \in \mathbb R:\
\alpha \in \image\left(\H_i(\L^{\le a};G) \to H_i(\L;G)\right) \right\} \\
&= \inf \left\{ a \in \mathbb R:\
\alpha \ \text{ is supported on }\ \L^{\le a} \right\}.
\end{align*}
A non-zero homology class $\alpha$ normally gives rise to a non-zero class 
$\beta$ in level homology at the level $\Cr(\alpha).$  Let us say that
two classes $\alpha \in H_i(\L;G)$ and $\beta \in \H_i(\L^{\le a}, \L^{<a};G)$
are {\em associated} if there exists an {\em associating class}
$\omega \in \H_i(\L^{\le a};G)$ with
\begin{equation}\label{eqn-associated}
\begin{diagram}[size=1.7em]
\omega & \rTo & \alpha && \H_i(\L^{\le a}) & \rTo& \H_i(\L)=H_i(\L)\\
\dTo   &         &        &\text{in }\ & \dTo    &     & \\
\beta  &         &        &         &\H_i(\L^{\le a},\L^{<a}) &&
\end{diagram}\end{equation}

\begin{lem}\label{lem-implications}  Let $\alpha \in H_i(\L;G)$, $\alpha \ne 0.$
Then the following statements hold. \begin{enumerate}
\item $\Cr(\alpha)$ is a critical value of $F.$
\item $\Cr(\alpha)$ is independent of the
homology theory (\Cech or singular) used in the definition.
\item 
$\Cr(\alpha) = \inf \left\{ a \in \mathbb R:\ \alpha \in \ker\left(H_i(\L)
\to \H_i(\L, \L^{\le a};G) \right) \right\}.$
\end{enumerate}
If $a\in \mathbb R$ is a nondegenerate critical value in the sense of Bott, or if
$G$ is a field, then
\begin{enumerate}\setcounter{enumi}{3}
\item $\Cr(\alpha)=a$ if and only if there exists
$0 \ne \beta \in H_i(\L^{\le a}, \L^{<a};G)$ associated with $\alpha.$
\item $\Cr(\alpha) < a$ if and only if  $\alpha$  is associated
to the zero class $0=\beta \in H_i(\L^{\le a}, \L^{<a};G).$
 \end{enumerate}
\end{lem}
\proof For part (1), if $\Cr(\alpha)$ were a regular
value than the flow of $-\text{grad}(F)$ would reduce the support of $\alpha$
below this value.  Now let $b_n \downarrow \Cr(\alpha)$ be a convergent 
sequence of regular values (which exists because the regular values of $F$ are
dense in $\mathbb R$).  By Lemma \ref{lem-cech} the \Cech homology and singular
homology of $\L^{\le b_n}$ coincide, which proves (2); and the homology sequence
for the pair $(\L, \L^{\le b_n})$ is exact, which proves (3).
If $a$ is a nondegenerate critical value (in the sense of Bott; cf.
\S \ref{subsec-free-loop}) then the set $\L^{\le a}$ is 
a deformation retract of some open set $U \subset \L$ which therefore 
contains $\L^{\le b_n}$ for sufficiently large $n.$  Hence $\alpha$ is 
the image of some 
\[\omega \in H_i(\L^{\le b_n}) \to H_i(U) \cong H_i(\L^{\le a}).\]
Moreover the homology sequence for $(\L^{\le a}, \L^{<a})$ is
exact, which proves (4) and (5).  If $G$ is a field, see Lemma
\ref{lem-approximation}.\qed

\subsection{}
Similarly, a cohomology class $\alpha \in H^j(\L,\L_0;G)$ is {\em supported} 
on a closed set $B \subset \L - \L_0$ if it maps to zero in $\H^j(\L-B,\L_0;G)$ or
equivalently, if it comes from a class in $\H^{j}(\L, \L-B;G).$  Define the
{\em critical value}
\begin{align*} 
\Cr(\alpha) &= \sup \left\{b:\ \alpha \in \ker\left(
H^j(\L,\L_0;G) \to \H^j(\L^{<b},\L_0;G)
\right) \right\} \\
&= \sup \left\{ b:\ \alpha\ \text{ is supported on }\ \L^{\ge b}
\right\} \end{align*}
Let us say the classes $\alpha \in H^j(\L,\L_0;G)$ and $\beta \in
\H^j(\L^{\le b},\L^{<b};G)$ are {\em associated} if there exists $\omega \in
\H^j(\L,\L^{<b};G)$ with
\begin{equation}
\begin{diagram}[size=1.7em]
\omega & \rTo & \alpha && \H^j(\L,\L^{<b}) & \rTo& H^j(\L,\L_0)\\
\dTo   &         &        &\text{in }\ & \dTo    &     & \\
\beta  &         &        &         &\H^j(\L^{\le b},\L^{<b}) &&
\end{diagram}\end{equation}

\begin{lem} Let $\alpha \in H^j(\L,\L_0;G)$, $\alpha \ne 0.$  Then the following
statements hold.\begin{enumerate}
\item $\Cr(\alpha)$ is a critical value of $F.$
\item $\Cr(\alpha)$ is independent of the homology theory used in its definition.
\end{enumerate}
If $a \in \mathbb R$ is a nondegenerate critical value in the sense of Bott, or
if $G$ is a field then \begin{enumerate}\setcounter{enumi}{2}
\item $\Cr(\alpha) = b$ if and only if there exits $0 \ne \beta \in H^j(\L^{\le b},
\L^{<b};G)$ associated with $\alpha.$
\item $\Cr(\alpha) < b$ if and only if $\alpha$ is associated with the zero
class $0 \in H^j(\L^{\le b}, \L^{<b};G).$
\end{enumerate}\end{lem}
\proof The proof is the same as in Lemma \ref{lem-implications}.  \qed

\section{The Chas-Sullivan Product} \label{sec-CS-product}
\renewcommand{\F}{\mathcal F}
\subsection{} Throughout the next three sections $M$ denotes a connected
Riemannian manifold of dimension $n.$  We assume either (a) the
manifold $M$ is orientable and oriented and coefficients for homology are taken
in the ring of integers $G=\mathbb Z$ or (b) the manifold $M$ is not necessarily
orientable, and coefficients for homology are taken in the ring $G=\mathbb Z/(2).$
We will often suppress mention of the coefficient ring $G$ when assumption
(a) or (b) is in force.

\subsection{}\label{subsec-homology-product}
In \cite{CS} a product  $*:H_i(\L) \times H_j(\L) \to H_{i+j-n}(\L)$
on the homology of the free loop space $\Lambda$ was
defined.  It has since been re-interpreted in a number of different contexts
(\cite{Cohen}, \cite{Cohen1}, \cite{Cohen2}).  Recall, for example from
\cite{Cohen1} or \cite{AS1} that it can be constructed as the following composition:
\begin{diagram}[height=2em]
H_i(\Lambda) \times H_j(\Lambda) && \\
\dTo_{\epsilon\times} && \\
H_{i+j}(\Lambda \times \Lambda)&\rTo^{s}&
H_{i+j}(\Lambda \times \Lambda, \Lambda \times \Lambda - \F) &
\rTo^{\tau}& H_{i+j-n}(\F) &\rTo^{\phi_*}& H_{i+j-n}(\Lambda)
\end{diagram}
In this diagram, $\epsilon = (-1)^{n(n-j)}$ (cf. \cite{Dold} VIII \S 13.3,
\cite{Chataur}), and $\times$ denotes the homology cross product.
The map $\tau$ is the Thom isomorphism (\ref{eqn-Thom-homology}) for the normal
bundle $\nu_{\F}$ of $\F=\L\times_M \L$ in $\L \times \L$
(\cite{Chataur} Prop. 2.2.3, \cite{Baker} Prop. 1.17).  The composition
$\tau\circ s$ is a Gysin homomorphism (\ref{eqn-Gysin-homology}).  The map
$\phi = \phi_{\frac{1}{2}}$ composes the two loops at time $t = 1/2.$
If $M$ is orientable then $\nu_{\F}$ is also orientable (\S \ref{lem-phi}).
We will often substitute the homotopic mapping $\phi_{\min}$ of Lemma
\ref{lem-phi} for $\phi_{\frac{1}{2}}.$  The construction may be summarized
as passing from the left to the right in the following diagram
\begin{equation}
\boxed{\begin{diagram}[size=2em]
\L\times\L & \lTo & \F &\rTo & \L
 \end{diagram}}\end{equation}
It is well known that the
Chas-Sullivan product is (graded) commutative, but this is not entirely
obvious since it involves reversing the order of composition of loops,
and the fundamental group of $M$ may be non-commutative. We include the
short proof because the same method will be used in \S
\ref{sec-cohomology-products}.

\begin{prop}\cite{CS}\label{prop-homology-signs}
If $a \in H_i(\L)$ and $b\in H_j(\L)$ then $b*a = (-1)^{(i-n)(j-n)}a*b.$
\end{prop}
\begin{proof}  The map $\sigma:\L\times\L \to \L \times \L$ that switches
factors satisfies $\sigma_*(a \times b) = (-1)^{ij}(b \times a).$  It
restricts to an involution $\sigma:\F \to \F.$
Identifying $S^1 = \mathbb R/\mathbb Z,$ define $\chi_r:S^1 \to S^1$ by
$\chi_r(t) = t+r.$ The (usual) action, $\widehat{\chi}:
S_1\times\L \to\L$ of $S^1$ on $\L$ is given by
\[ \widehat{\chi}_r(\gamma) = \gamma \circ \chi_r\]
for $r \in S^1.$ The action of $\widehat{\chi}_{\half}$ preserves
$\phi_{\half}(\F)$ and in fact
\begin{equation}\label{eqn-chi-switch}
\sigma(\gamma) = \widehat{\chi}_{\frac{1}{2}}(\gamma)=\beta\cdot\alpha
\end{equation}
for any $\gamma = \alpha \cdot \beta \in \phi_{\frac{1}{2}}(\F)$ which is
a composition of two loops $\alpha,\beta$ glued at time $1/2.$  Let
$\mu_{\F} \in H^n(\L\times\L, \L\times\L - \F)$ be the Thom class of the
normal bundle $\nu_F.$  Then $\sigma^*(\mu_{\F}) = (-1)^n \mu_{\F}.$  Since
$\chi_r:\L \to \L$ is homotopic to the identity we have:
\begin{align*}
(-1)^{n(n-i)}b*a &= \phi_*(\mu_{\F} \cap (b \times a))\\
&= (-1)^{ij} \phi_*(\mu_{\F} \cap \sigma_*(a\times b)) \\
&= (-1)^{ij} \phi_* \sigma_*(\sigma^*(\mu_{\F}) \cap (a \times b))\\
&= (-1)^{ij}(-1)^n\widehat{\chi}_{\half}\phi_*(\mu_{\F} \cap (a\times b))\\
&= (-1)^{ij}(-1)^n(-1)^{n(n-j)} a*b \qedhere \end{align*}
\end{proof}

\begin{prop} \label{prop-product-extended}  If $M$ is oriented let $G =
\mathbb Z$ otherwise let $G = \mathbb Z/(2).$  Let $\alpha,\beta \in H_*(\L;G)$ 
be homology classes supported on closed sets $E,F \subset \L$ respectively.  
Then $\alpha*\beta$ is supported on the closed set $E*F = \phi_{\min}(E \times_MF).$
In particular,
\begin{equation}\label{eqn-cr-sum-homology}
\Cr(\alpha*\beta) \le \Cr(\alpha)+\Cr(\beta).\end{equation}
For any $a,b$ with $0 \le a,b \le
\infty$ the Chas-Sullivan product extends to a family of products,
\[
 \H_i(\Lambda^{\le a};G) \times \H_j(\Lambda^{\le b};G) \to \H_{i+j-n}
(\L^{\le a+b};G)
\]
and for any $0 \le a' < a \le \infty$ and $0 \le b' < b \le \infty$ to products
\begin{align}
\H_i(\L^{\le a}, \L^{\le a'};G) \times \H_j(\L^{\le b}, \L^{\le b'};G) &\to
\H_{i+j-n}(\L^{\le a+b}, \L^{\le \max(a+b',a'+b)};G)\\
\H_i(\L^{\le a}, \L^{<a};G) \times \H_j(\L^{\le b}, \L^{<b};G) &\to
\H_{i+j-n}(\L^{\le a+b}, \L^{<a+b};G).\label{eqn-level-product}
\end{align}
These products are compatible under the mappings induced by inclusion.
If the set $\Cr\subset \mathbb R$ of critical values is discrete then we obtain
a ring structure on the {\em level homology}
\begin{equation}\label{eqn-level-homology}
 \underset{a\in \Cr}{\oplus} \H_*(\L^{\le a},\L^{<a};R).\end{equation}
\end{prop}

\proof
First we construct, for any open sets $A' \subset A \subset \L$ and
$B' \subset B \subset \L$ a product
\begin{equation}\label{eqn-general-product}
H_i(A,A';G) \times H_j(B,B';G) \to H_{i+j-n}(A*B, A*B' \cup A'*B)\end{equation}
on singular homology, as shown in Figure \ref{fig-CSDef}.
\begin{center}\begin{figure}[!h]
{\small\begin{diagram}[size=2em]
H_i(A,A') \times H_j(B,B') & \rTo^{\epsilon\times} & H_{i+j}(A\times B, A' \times B
\cup A \times B')\\
& & \dTo_{=}\\
&&H_{i+j}(A\times B, A\times B - (A-A')\times(B-B'))\\
& & \dTo \\
&&H_{i+j}(A\times B, A\times B - (A-A')\times_M (B-B'))\\
& & \dTo^{\text{Thom isomorphism } \S \ref{prop-Thom}\quad}_{\tau} \\
&&H_{i+j-n}(A\times_MB, A\times_MB - (A-A')\times_M(B-B'))\\
& & \dTo_{=} \\
H_{i+j-n}(A*B, A'*B \cup A*B') & \lTo^{(\phi_{\min})_*}&
H_{i+j-n}(A\times_MB, A'\times_MB \cup A\times_MB')
\end{diagram}}\label{fig-CSDef}
\caption{Relative Chas-Sullivan Product}
\end{figure}\end{center}
In this diagram, $\times$ denotes the homology cross product, $\epsilon =
(-1)^{n(n-j)},$  and the Thom isomorphism $\tau$ of Proposition \ref{prop-Thom}
 is applied to the triple
\[(A-A')\times_M (B-B') \subset A\times_M B \subset
A\times B.\]
The hypotheses of Proposition \ref{prop-Thom} are satisfied
because $A \times B$ is a Hilbert manifold, so $A \times_M B$ has a normal
bundle and tubular neighborhood in $A \times B,$ and because the subspace
$(A-A')\times_M(B-B')$ is closed in $A \times_M B.$

It is easy to see that this product is compatible with the product in \S
\ref{subsec-homology-product} in the sense that the following diagram commutes:
\begin{diagram}[height=2em]
H_i(\L)\times H_j(\L) & \lTo & \H_i(A)\times \H_j(B) & \rTo &
\H_i(A, A') \times \H_j(B, B') \\
\dTo & & \dTo & & \dTo \\
H_{i+j-n}(\L) & \lTo & \H_{i+j-n}(A*B) & \rTo & \H_{i+j-n}
(A*B, A'*B \cup A*B').
\end{diagram}

Now the other products may be obtained by a limiting procedure using
Lemma \ref{lem-approximation}.  The statement about the support of $\alpha*\beta$
follows by taking a sequence of open neighborhoods $A_n \downarrow E$ and
$B_n \downarrow F.$
To construct the product (\ref{eqn-level-product}) for example, start with
\[H_i(\L^{<a+\epsilon}, \L^{<a'-\delta}) \times H_j(\L^{<b+\epsilon},
\L^{<b'-\delta}) \to H_{i+j-n}(\L^{<a+b+\epsilon-\delta},
\L^{<\max(a'+b+\epsilon-\delta, a+b'+\epsilon-\delta)}) \]
where $\delta > \epsilon.$  Taking the (inverse) limit as $\epsilon
\downarrow 0$ gives a pairing
\[\H_i(\L^{\le a},\L^{<a-\delta}) \times \H_j(\L^{\le b}, \L^{<b-\delta})
\to \H_{i+j-n}(\L^{\le a+b}, \L^{<\max(a'+b-\delta, a+b'-\delta)}).\]
Taking the direct limit as $\delta \downarrow 0$ (and recalling from
\S \ref{subsec-Cech-limits} that homology
commutes with direct limits) gives the pairing (\ref{eqn-level-product}).
The other products are similarly constructed.  (The Chas-Sullivan product 
can even be constructed this way, cf. \cite{Chataur}.)  This completes
the proof of Proposition \ref{prop-product-extended}. \qed

\subsection{}
If $\alpha = [A, \partial A]$ and $\beta = [B, \partial B]$
are the fundamental classes of manifolds $(A,\partial A) \subset
(\L^{\le a}, \L^{\le a'})$ and $(B, \partial B) \subset (\L^{\le b},
\L^{\le b'})$ which are transverse over $M$, then the C-S product
$[\alpha]*[\beta]$ is represented by the fundamental class of the manifold
\[\phi_{\min}(A\times_MB, A\times_M \partial B \cup \partial A \times_MB) = 
(A*B, \partial (A*B)).\] For equations (\ref{eqn-level-power}) and (\ref{eqn-diffeo}) 
we will need a similar fact about homology classes $\alpha, \beta$ which are
supported on $(A,\partial A)$ and $(B, \partial B)$ but which are not
necessarily the fundamental classes.  The following proposition is more 
or less the original definition of the product $*$ from \cite{CS}:
\begin{prop}\label{prop-product-energy}
Let $(A, \partial A)$ and $(B, \partial B)$ be smooth manifolds with
boundary.  Let $(A, \partial A) \to (\L^{< \alpha}, \L^{< \alpha'})$ and
$(B, \partial B) \to (\L^{< \beta}, \L^{< \beta'})$ be smooth
embeddings, where $\alpha'<\alpha$ and $\beta' <\beta.$  Assume
the mappings $\ev_0:A \to M$ and $\ev_0:B \to M$ are transverse.  
If $M$ is oriented, then assume $A,B$ are orientable and oriented.
Let $\alpha \in H_i(A, \partial A)$ and $\beta \in
H_i(B, \partial B).$ Denote their images in the homology
of $\L$ by $[\alpha] \in H_i(\L^{<a},\L^{<a'})$ and
$[\beta] \in H_j(\L^{<b},\L^{<b'}).$ Define
\[\alpha*\beta \in H_{i+j-n}(A\times_MB, A\times_M \partial B
\cup \partial A \times_M B)\cong H_{i+j-n}(A\times_MB,
A\times_MB - A' \times_M B')\]
to be the image of $(\alpha, \beta)$ under the following composition.
\begin{equation}\label{eqn-representable}
\begin{diagram}
H_{i}(A,\partial A)\times H_j(B, \partial B)
&\rTo^{\epsilon\times}&
H_{i+j}(A\times B, A\times B - A' \times B')\\
&& \dTo \\
H_{i+j-n}(A\times_M B, A\times_M B - A' \times_M B')&
\lTo_{(\ref{prop-Thom})}^{\tau}&
H_{i+j}(A\times B, A\times B - A' \times_MB')
\end{diagram}\end{equation}
Then 
$ [\alpha]*[\beta] = (\phi_{\min})_*(\alpha *\beta) \in
H_{i+j-n}(\L^{<a+b}, \L^{\max(a+b',a'+b)}).$
\end{prop}

\proof
The transversality assumption is equivalent to the statement that the mapping
\[(\ev_0,\ev_0):A \times B \to M \times M\]
is transverse to the diagonal $\Delta.$  By \cite{Chataur} Prop. 2.2.3 or \cite{
Baker} Prop. 1.17, the intersection $A \times_M B = A \times B \cap \F$ has a
tubular neighborhood in $A \times B$ (with normal bundle $\nu_{\F}|(A \times_MB).$
As in the proof of Proposition \ref{prop-product-extended} this makes it
possible to apply the Thom isomorphism (\ref{prop-Thom}).
Then the diagram (\ref{eqn-representable}) maps, term by term, to the diagram
in the proof of Proposition \ref{prop-product-extended} where the
relative C-S product is defined.  The Proposition amounts to the statement
that these mappings commute, which they obviously do.  \qed

\section{Index growth}
Continue with the notation $M,\L,F,\Sigma$ of \S \ref{sec-freeloop}. 

\begin{prop}\label{prop-index-iterates}
Let $\gamma$ be a closed geodesic with index
$\lambda$ and nullity $\nu.$ Let $\lambda_{m}$ and $\nu_{m}$ denote the index
and nullity of the $m$-fold iterate $\gamma^{m}.$ Then $\nu_{m}\leq2(n-1)$ for
all $m$ and
\begin{align}
|\lambda_{m}-m\lambda|  &  \leq(m-1)(n-1)\label{eqn-Bott-inequality1}\\
|\lambda_{m}+\nu_{m}-m(\lambda+\nu)|  &  \leq(m-1)(n-1).
\label{eqn-Bott-inequality2}%
\end{align}
The average index
\[
\lambda_{av}=\lim_{m\rightarrow\infty}\frac{\lambda_{m}}{m}%
\]
exists and
\begin{equation}
|\lambda-\lambda_{av}|\leq n-1\ \text{ and }\ |\lambda+\nu-\lambda_{av}|\leq
n-1. \label{eqn-average}%
\end{equation}
Now assume $\gamma$ and $\gamma^{2}$ are nondegenerate critical points,
i.e. they lie on isolated nondegenerate critical orbits.  Then the
inequalities in (\ref{eqn-average}) are strict, and if $m$ is sufficiently
large then the inequality (\ref{eqn-Bott-inequality1}) is also strict.
 Moreover if we let%
\begin{align*}
\lambda_{m}^{\min}  &  =m\lambda_{1}-(m-1)(n-1);\\
\lambda_{m}^{\max}  &  =m\lambda_{1}+(m-1)(n-1)
\end{align*}
be the greatest and smallest possible values for $\lambda_m$ that are
compatible with (\ref{eqn-Bott-inequality1}) then
\begin{align}
\lambda_{m}  &  >\lambda_{m}^{\min}\Longrightarrow\lambda_{j}>\lambda
_{j}^{\min}\text{ \ for all }j>m;\label{eqn-nomin}\\
\lambda_{m}  &  <\lambda_{m}^{\max}\Longrightarrow\lambda_{j}<\lambda
_{j}^{\max}\text{ \ for all }j>m.\label{eqn-nomax}
\end{align}

\end{prop}

Much of this is standard and is well-known to experts, (see the references
at the beginning of Appendix \ref{sec-Bott-Morse}) but for completeness we
include a proof based on the following well-known facts:

\subsection{Well-Known Facts}
Let $M$ be an $n$ dimensional Riemannian manifold.  Let $\gamma$ be a
closed geodesic.  The Poincare map $P$ \ (linearization of the geodesic flow
at a periodic point) is in $Sp(2(n-1),\mathbb{R})$ and is defined up to
conjugation.  The index formula of Bott is
\begin{equation}\label{eqn-WellKnown1}
\lambda_{m}=\text{ index }(\gamma^{m})=\underset{\omega^{m}=1}{\sum
\mathbf{\Omega}_{\gamma}}(\omega).
\end{equation}
where \ the $\omega$-index  $\Omega_{\gamma}$ is an integer valued function
defined on the unit circle with  $\mathbf{\Omega}_{\gamma}(\omega
)=\mathbf{\Omega}_{\gamma}(\overline{\omega})$. The function
$\mathbf{\Omega}_{\gamma}$ is constant except at the eigenvalues of $P$.  Its
jump at each eigenvalue is determined by the  \emph{splitting numbers}
$\ S_{P}^{\pm}(\omega)\epsilon \mathbb{Z}:$
\begin{equation}\label{eqn-WellKnown2}
\underset{\varepsilon\rightarrow0^{+}}{\lim}\mathbf{\Omega}_{\gamma}(\omega
e^{\pm i\varepsilon})=\mathbf{\Omega}_{\gamma}(\omega)+S_{P}^{\pm}(\omega).
\end{equation}
which depend only upon the conjugacy class of $P.$ The nullity satisfies%
\begin{equation}
\nu_{m}=\text{ nullity }(\gamma^{m})=\underset{\omega^{m}=1}{\sum}%
\mathcal{N}_{P}(\omega)
\end{equation}
where $\mathcal{N}_{P}(\omega):=$ $\dim\ker(P-\omega I)$. \ The numbers
$S_{P}^{\pm}(\omega)$ and $\mathcal{N}_{P}(\omega)$ are additive on
indecomposable symplectic blocks, and on each block
\begin{align}
&  S_{P}^{\pm}(\omega)\epsilon\{0,1\};\label{eqn-WellKnown3}\\
&  \mathcal{N}_{P}(\omega)-S_{P}^{\pm}(\omega)\epsilon\{0,1\}
\label{eqn-WellKnown4}\end{align}

\subsection{Proof of Proposition \ref{prop-index-iterates}}
Assuming the above facts (\ref{eqn-WellKnown1}-\ref{eqn-WellKnown4}),
the index-plus-nullity  satisfies
\begin{align}
\lambda_{m}+\nu_{m}=\text{ index-plus-nullity }(\gamma^{m}) & =\underset
{\omega^{m}=1}{\sum\mathbf{\Upsilon}_{\gamma}}(\omega);\\
\underset{\varepsilon\rightarrow0^{+}}{\lim}\mathbf{\Upsilon}_{\gamma}(\omega
e^{\pm i\varepsilon}) &  =\mathbf{\Upsilon}_{\gamma}(\omega)-T_{P}^{\pm
}(\omega);
\end{align}
where%
\begin{align*}
\mathbf{\Upsilon}_{\gamma}(\omega) &  :=\mathbf{\Omega}_{\gamma}%
(\omega)+\mathcal{N}_{P}(\omega);\\
T_{P}^{\pm}(\omega) &  :=\mathcal{N}_{P}(\omega)-S_{P}^{\pm}(\omega).
\end{align*}
Thus $T_{P}^{\pm}(\omega)$ is additive and takes values in $\{0,1\}$ on each
indecomposable block, and $-(\lambda_{m}+\nu_{m})$ and $-\mathbf{\Upsilon
}_{\gamma}$ have the same formal properties
(\ref{eqn-WellKnown1}-\ref{eqn-WellKnown3}) as $\lambda_{m}$ and
$\mathbf{\Omega}_{\gamma}$.  Thus a proof of  the statements about
$\lambda_{m}$\ using only these three properties will also serve as a
proof of the statements about $\lambda_{m}+\nu_{m}$.

\quash{
\textbf{Remark} \ The symmetry between $\lambda_{m}$ and $\mathbf{\Omega
}_{\gamma}$, and $(\lambda_{m}+\nu_{m})$ and $\mathbf{\Upsilon}_{\gamma}$, is
another example of Poincar\'{e} duality at play in the free loop space. \}
}
As a consequence of equations (\ref{eqn-WellKnown1}-\ref{eqn-WellKnown3})
we have:
\begin{equation}\label{eqn-omega-difference}
|\mathbf{\Omega}_{\gamma}(\omega)-\mathbf{\Omega}_{\gamma}(\tau)|\leq
n-1\ \text{ for all }\ \omega,\tau.
\end{equation}
Moreover if $|\mathbf{\Omega}_{\gamma}(\omega)-\mathbf{\Omega}_{\gamma}%
(\tau)|=n-1$, with $\operatorname{Re}\omega$ $<\operatorname{Re}$ $\tau$, then
all the eigenvalues of $P$ lie in the unit circle, with real part in
$[\operatorname{Re}\omega$ $,\operatorname{Re}$ $\tau]$. (To see this, note
that each indecomposable block has dimension at least $2.$)  Equation
(\ref{eqn-Bott-inequality1}) and the first half of (\ref{eqn-average}) follow.
Moreover,
\begin{equation}\label{eqn-index-integral}
\lambda_{av}=\frac{1}{2\pi}\int_0^{2\pi}
\mathbf{\Omega}_{\gamma}(e^{it})dt
\end{equation}

Equality in \ref{eqn-Bott-inequality1} implies that $\lambda_{m}=
\lambda_{m}^{\max}$ or $\lambda_{m}=\lambda_{m}^{\min}$.  Equality in
the first half of \ref{eqn-average}, together with (\ref{eqn-omega-difference})
and (\ref{eqn-index-integral}) implies that
$ |\mathbf{\Omega}_{\gamma}(\omega)-\mathbf{\Omega}_{\gamma}(1)|=n-1$
almost everywhere on the circle, that $1$ is the only
eigenvalue of $P$, and thus that $ \lambda_{m}=\lambda_{m}^{\max}$ for all
$m$ or $\lambda_{m}=\lambda_{m}^{\min}$ for all $m$.

If $\gamma$ and $\gamma^{2}$ are nondegenerate, then neither $-1$ nor $1$ is
an eigenvalue of $P$.  Suppose  that $\lambda_{2}=\lambda_{2}^{\max
}=2\lambda_{1}+n-1.$ Then (using the Bott formula)
$\mathbf{\Omega}_{\gamma}(0)=\lambda_{1\text{ \ }}$
and
$\mathbf{\Omega}_{\gamma}(-1)=\lambda_{1\text{ \ }}+n-1.$
It follows that all the eigenvalues of $P$ lie on the unit circle.
Furthermore
\[
 \lambda_{k}=\lambda_{k}^{\max}=\lambda_{1}+(k-1)(n-1)
\]
if and only if all the $k^{th}$ roots of unity except for $1$ lie to
the left of all the eigenvalues of $P$ on the unit circle. \ The statements
about strict inequality and equations (\ref{eqn-nomin}), (\ref{eqn-nomax}) follow.
This concludes the proof of Proposition \ref{prop-index-iterates}. \qed

The following lemma will be used in the proof of Theorem
\ref{prop-level-products}.  

\begin{lem}\label{lem-based-index}
Fix a basepoint $x_0 \in M$ and let $\Omega = \Omega_{x_0} = \ev_0^{-1}(x_0)$
be the Hilbert manifold of loops that are based at $x_0.$  Let $\gamma \in 
\Omega$ be a closed geodesic, all of whose iterates are nondegenerate.
Let $\lambda_r$ be the Morse index of the iterate $\gamma^r.$ \begin{remarks}
\item[(i) ] Suppose the
index growth is maximal for $r = n,$ that is, $\lambda_n = n\lambda_1 +
(n-1)^2.$  Then the index $\lambda_1^{\Omega}$ for $\gamma$ in the based loop
space is equal to the index $\lambda_1$ for $\gamma$ in the free loop space.
\item[(ii) ] Suppose the index growth is minimal for $r=n,$ that is, $\lambda_n =
n\lambda_1 -(n-1)^2.$  Then the difference between $\lambda_1$ and 
$\lambda_1^{\Omega}$ is maximal, that is, $\lambda_1 = \lambda_1^{\Omega} 
+n-1.$\end{remarks}
\end{lem}
\proof
Let $T_{\gamma(0)}^{\perp}M$ be the subspace that
is orthogonal to the tangent vector $\gamma'(0).$
Let $T_{\gamma}^{\perp}\L$ (resp. $T_{\gamma}^{\perp}
\Omega$) be the subspace of vector fields $V(t)$ along $\gamma$ with 
$V(t) \perp \gamma'(t)$ for all $t.$  By a standard argument, for all $r \ge 1,$
$\lambda_r$  (resp. $\lambda^{\Omega}_r$) is the dimension of a maximal negative subspace 
of $T_{\gamma^r}^{\perp}\L$ (resp. of $T_{\gamma^r}^{\perp}(\Omega)$).
Let $W_r \subset T_{\gamma^r}^{\perp}\L$ be a maximal
negative subspace.  Let $K_r$ be the kernel of the map
\vskip-20pt\begin{diagram}[size=2em]
 W_r &\rTo^{\phi} T_{\gamma(0)}^{\perp}M \times T_{\gamma(0)}^{\perp}M
\times \cdots \times T_{\gamma(0)}^{\perp}M & \\
V &\mapsto \left( V(0), V(\textstyle{\frac{1}{r}}), \cdots, 
V(\textstyle{\frac{r-1}{r}})\right).&
\end{diagram}
The dimension of the image of this map is $\le r(n-1)$ so $\dim(K_r) \ge
\lambda_r - r(n-1).$  On the other hand, the kernel of $\phi$ on
$T^{\perp}_{\gamma^r}\L$ is the direct sum of $r$ copies of 
$T^{\perp}_{\gamma}\Omega$ in a way that is compatible with the Hessian
of $F,$ so the index of $F$ on this kernel is $r\lambda_1^{\Omega}.$  Thus,
for all $r \ge 1$ we have,
\begin{equation}\label{eqn-Kr}
 r\lambda_1^{\Omega} \ge \dim(K_r) \ge \lambda_r - r(n-1).\end{equation}
Now consider Part (i). Clearly $\lambda_1^{\Omega} \le \lambda_1$ so we need to verify
the opposite inequality.  But taking $r=n$ and  $\lambda_n = n\lambda_1 + (n-1)^2$
in (\ref{eqn-Kr}) gives $\lambda_1^{\Omega} \ge \lambda_1 - \frac{n-1}{n}$ as claimed.

Now consider Part (ii).  Taking $r=1$ in (\ref{eqn-Kr}) gives $\lambda_1^{\Omega}
\ge \lambda_1-(n-1).$  However it is also clear that 
$\lambda_n \ge \lambda_n^{\Omega} \ge n\lambda_1^{\Omega}$ (because we can concatenate
$n$ negative vector fields along $\gamma$ to obtain a negative vector field along
$\gamma^n$).  Taking $\lambda_n = n \lambda_1 - (n-1)^2$ gives
\[\lambda_1 - \frac{(n-1)^2}{n} = \lambda_1 -(n-1) + \frac{n-1}{n} \ge \lambda_1^{\Omega}.
\qedhere\]
\quash{
\[ \lambda_1^{\Omega} \le \lambda_1 - \frac{(n-1)^2}{n} = \lambda_1 -(n-1) + \frac{n-1}{n}.
\qedhere \]
}

\section{Level nilpotence}\label{sec-nilpotence}
\subsection{}
  Let $R$ be a commutative ring with unit, let $a \in\mathbb R,$ and let
$\beta \in \H_i(\L^{\le a}, \L^{<a};R).$  We way that
$\beta$ is {\em level nilpotent} if there exists $m$ such that the
Chas-Sullivan product in level homology vanishes:
\[0= \beta^{*m} = \beta*\beta*\cdots*\beta \in \H_{mi+(m-1)n}(\L^{\le
ma},\L^{<ma};R).\]
Let $\alpha \in \H_i(\L;R).$  We say that $\alpha$ is {\em
level nilpotent} if there exists $m$ such that
$\Cr (\alpha^{*m}) < m \Cr(\alpha).$

\begin{lem}\label{lem-nilpotent}
  Let $\alpha \in \H_i(\L;R),$ $\alpha \ne 0,$ and let $a = \Cr(\alpha).$  Let
$\beta \in \H_i(\L^{\le a}, \L^{<a};R)$ be an associated class.  (A non-zero
such class $\beta$ exists when $a$ is a nondegenerate critical value in
the sense of Bott, cf. Lemma \ref{lem-implications}.)  If
$\beta$ is level nilpotent then $\alpha$ is also level nilpotent.
\end{lem}
\proof
Let $\omega \in H_i(\L^{\le a};R)$ be a class which associates $\alpha$
and $\beta.$  Then $\omega^{*m}\in \H_b(\L^{\le ma};R)$ associates
$\alpha^{*m}$ and $\beta^{*m},$ where $b=mi+(m-1)n.$  But $\beta^{*m}=0$
if $m$ is sufficiently large, which implies by Lemma \ref{lem-implications}
that $\Cr(\alpha^{*m}) < ma.$
\qed

\begin{thm} \label{thm-nilpotence}
 Let $F:\L \to \mathbb R$ as above.  If $M$ is orientable let $G=
\mathbb Z$, otherwise let $G = \mathbb Z/(2).$  Suppose all the critical
orbits of $F$ are isolated and nondegenerate.  Then
every homology class $\alpha \in H_i(\L;G)$ is level nilpotent, and for
every $a \in \mathbb R,$ every level homology class
$\beta \in H_i(\L^{\le a},\L^{<a};G)$ is level nilpotent.
\end{thm}

\proof  By Lemma \ref{lem-nilpotent} it suffices to prove that $\beta
\in H_i(\L^{\le a}, \L^{<a};G)$ is level nilpotent, where $a\in \mathbb R$
is a nondegenerate critical value.  We suppress mention of the coefficient
ring $G.$   The critical
set $\Sigma^{=a}= \Sigma(F) \cap F^{-1}(a)$ consists of the $S^1$ orbits of
finitely many closed geodesics, say, $\gamma_1,\cdots, \gamma_r.$  Let
$\overline{\gamma}$ denote the $S^1$ orbit of $\gamma.$  For $1
\le j \le r$ let $U_j \subset \L$ be a neighborhood of
$\overline{\gamma}_j,$ chosen  so that $U_j \cap U_k = \phi$ whenever
$j \ne k.$ We may choose the ordering so that there exists $s \le r$ such
that $H_i(\L^{\le a} \cap U_j, \L^{<a} \cap U_j)\ne 0$ if and
only if $1 \le j \le s.$  The $\gamma_j$ with $1 \le j \le s$ are the
critical points that are relevant to $\beta,$ and Theorem
\ref{thm-Morse-Bott} implies that the index of
$\gamma_j$ ($1 \le j \le s$) is either $i$ or $i-1.$

Set $\Sigma^{=a}_0 =\cup_{j=1}^s \overline{\gamma}_j.$    Using Proposition
\ref{prop-level}, the level homology group is a direct sum
\begin{equation*}
H_i(\L^{\le a}, \L^{<a})  \cong \oplus_{j=1}^r H_i(\L^{\le a}
\cap U_j, \L^{<a} \cap U_j)
\end{equation*}
Since these factors vanish for $j>s$ we have canonical isomorphisms
\begin{align*}
H_i(\L^{\le a}, \L^{<a})
&\cong \oplus_{j=1}^s H_i(U_j^{<a}\cup \overline{\gamma}_j, U_j^{<a})\\
&\cong H_i(\L^{<a} \cup \Sigma^{=a}_0, \L^{<a})
\end{align*}
using excision and homotopy equivalences.  By comparing the long exact
sequence for the pair $(\L^{\le a}, \L^{<a})$ with the long exact sequence
for the pair $(\L^{<a} \cup \Sigma^{=a}_0, \L^{<a})$ and using the five lemma, we conclude that
the
inclusion induces an
isomorphism
\[ H_i(\L^{<a}\cup \Sigma^{=a}_0) \cong H_i(\L^{\le a}).\]
Therefore $\beta$ is supported on $\L^{<a} \cup \Sigma^{=a}_0$
so for any $m \ge 1,$ $\beta^{*m}$ is supported on
\[ (\L^{<a} \cup \Sigma_0^{=a})^{*m} \subset \L^{<ma} \cup
(\Sigma_0^{=a})^{*m}.\]
The only critical points in $(\Sigma_0^{=a})^{*m}$ are the $m$-fold iterates
of the geodesics in $\overline{\gamma}_j$ (with $1 \le j \le s$).  Thus, using an arbitrarily
brief
flow along the
trajectories of $-\nabla F$ we obtain an isomorphism
\begin{equation}\label{eqn-target}
H_b(\L^{<ma}\cup (\Sigma_0^{=a})^{*m}, \L^{<ma}) \cong
\overset{s}{\underset{j=1}{\oplus}} H_b (U_{j,m}^{<ma} \cup
\overline{\gamma_j^m}, U_{j,m}^{<ma}) \ni \beta^{*m}
\end{equation}
where the $U_{j,m}$ are disjoint neighborhoods of the
$\overline{\gamma_j^m}$ containing no other critical points, and where
\[b = mi - (m-1)(n-1).\]
(There may be other critical points at level $ma$,
but the support of $\beta^{*m}$ does not contain such points.)

We will now
show that each of the summands on the right hand side of equation
(\ref{eqn-target}) vanishes if $m$ is sufficiently large.
Fix $j$ with $1 \le j \le s$ and let $\lambda_m$ denote the index of
$\gamma_j^m.$  According to Theorem \ref{thm-Morse-Bott},
$\lambda_1 \in \left\{ i-1,i \right\}.$  If the $j$-th summand in
(\ref{eqn-target}) is not zero then $\lambda_m \in \left\{b-1, b\right\}.$
But this contradicts the strict inequality in (\ref{eqn-Bott-inequality1}),
which holds if $m$ is sufficiently large.  \qed

\section{Cohomology products}\label{sec-cohomology-products}
\subsection{}\label{subsec-first-cohomology-product}
Throughout this section we take cohomology with coefficients in
the integers $\mathbb Z$ if $M$ is orientable (in which case we assume
an orientation has been chosen); otherwise in $\mathbb Z/(2).$  To
simplify the notation, we suppress further mention of these coefficients.
It is possible to define a product in cohomology, exactly along the same
lines as the Chas-Sullivan product in \S \ref{subsec-homology-product}, as
the composition
\begin{equation*}
 H^i(\Lambda) \times H^j(\L)\to H^{i+j}(\L \times \L) \to H^{i+j}(\F)
 \overset{\cong}{\longrightarrow} H^{i+j+n}(\L, \L-\phi_{\frac{1}{2}}(\F)) \to
H^{i+j+n}(\L)
\end{equation*}
(with $\phi_{\frac{1}{2}}: \F \to \L$ as in \S \ref{subsec-evaluation})
using the Thom isomorphism for the normal bundle of the figure
eight space $\phi_{\frac{1}{2}}(\F)\subset\L.$  However we will show below
that this product is zero when $i,j > n.$  Instead, we will construct a cohomology product,
\begin{equation}\label{eqn-cohomology-product}
\begin{diagram}[size=2em]
H^i(\L, \L_0) \times H^j(\L, \L_0) &\rTo^{\circledast}& H^{i+j+n-1}(\L, \L_0)
\end{diagram}
\end{equation}
using the following mapping
\[J: \Lambda \times [0,1] \to \Lambda\ \text{ given by }\
J(\alpha,s) = \alpha \circ \rh.\]  
Here, $\theta=\rh:[0,1] \to [0,1]$ is the reparametrization
function that is linear on $[0,\frac{1}{2}]$, linear on $[\frac{1}{2},1]$, has
$\theta(0)=0$, $\theta(1)=1$, and $\theta(\frac{1}{2}) = s,$ see Figure \ref{fig-theta}
and \S \ref{subsec-aboutJ}

\setlength{\unitlength}{0.7pt}

\begin{figure}[h!]\begin{center}
\begin{picture}(200,200)\label{fig-theta}
\linethickness{.7pt}
\put(20,20){\vector(1,0){170}} \put(195,15){$t$}
\put(20,20){\vector(0,1){170}} \put(18,195){$\theta$}
\put(100,20){\circle*{3}} \put(96,5){$\scriptstyle{\frac{1}{2}}$}
\put(5,175){$\scriptstyle{1}$} \put(177,5){$\scriptstyle{1}$}
\put(20,20){\circle*{3}} \put(20,180){\circle*{3}} \put(180,20){\circle*{3}}
\put(18,5){$\scriptstyle{0}$}  \put(5,18){$\scriptstyle{0}$}
\put(20,20){\line(4,1){80}}
\put(100,40){\line(4,7){80}}
\put(20,20){\line(4,7){80}}
\put(100,160){\line(4,1){80}}
\dottedline{5}(180,20)(180,180)\dottedline{5}(20,180)(180,180)
\put(120,60){$\scriptstyle{s=\frac{1}{8}}$}
\put(50,120){$\scriptstyle{s=\frac{7}{8}}$}
\end{picture}\end{center}\caption{Graph of $\rh$}
\end{figure}
Let $\F^{>0,>0} = (\L - \L_0)\times_M (\L - \L_0)$ denote the set of composable
pairs $(\alpha,\beta)$ such that $F(\alpha)>0$ and $F(\beta)>0.$
The relative Thom isomorphism (\ref{eqn-relative-Thom-cohomology}) for cohomology gives
\[ H^m(\F, \F - \F^{>0,>0}) \cong H^{m+n}(\L, \L -
\phi_{\frac{1}{2}}(\F^{>0,>0}).\]
The cohomology product is then the composition down the left hand column of
Figure \ref{fig-def-cohomology}.
Here, $\omega = (-1)^{j(n-1)}$, $I=[0,1]$,  $\F^{\bullet,0}
=\L \times_M\L_0$ and $\F^{0,\bullet} = \L_0 \times_M \L.$  The mapping $\tau$ is
the Thom isomorphism given by the cup product with the Thom class $\mu_{\F}$ of
the normal bundle of $\phi_{\half}\F$ in $\L,$ and $\kappa$ is given by the K\"unneth
theorem.  It uses the fact that $J(\L_0 \times [0,1])$ and $J(\L \times \left\{0,1\right\})$
are disjoint from $\phi_{\half}(\F^{>0,>0}).$ 
\begin{figure}[!h]
{\footnotesize\begin{diagram}[height=2em]
H^i(\L,\L_0)\times H^j(\L,\L_0) & \rTo & H^i(\L) \times H^j(\L) \\
\dTo_{\omega\times} &      &   \dTo_{\omega\times}  \\
H^{i+j}(\L\times\L,\L\times\L_0 \cup \L_0 \times \L)   & \rTo &
H^{i+j}(\L\times\L)\\
\dTo  &  &  \dTo \\
H^{i+j}(\F, \F^{\bullet,0} \cup \F^{0,\bullet}) & \rTo & H^{i+j}(\F) \\
\dEquals &  &  \dEquals \\
H^{i+j}\lp\F, \F- \F^{>0,>0}\rp &  \rTo  & H^{i+j}\lp\F\rp \\
\dTo^{\tau}_{(\ref{eqn-relative-Thom-cohomology})} & &
\dTo^{\tau}_{(\ref{eqn-Gysin-cohomology})} \\
H^{i+j+n}\lp\L, \L-\phi_{\half}(\F^{>0,>0})\rp & \rTo & H^{i+j+n}\lp\L,
\L - \phi_{\half}(\F)\rp \\
\dTo^{J^*}  & &  \dTo^{J^*} \\
H^{i+j+n}\lp\L \times I, \L \times \partial I \cup \L_0 \times I\rp
& \rTo^{\eta} & H^{i+j+n}(\L\times I)=H^{i+j+n}(\L) \\
\dTo^{\kappa}    &    &  \\
H^{i+j+n-1}(\L, \L_0)
\end{diagram}}
\caption{Definition of $\circledast$}\label{fig-def-cohomology}
\end{figure}

\bigskip

\noindent
 Denote the cohomology product
of two classes $\alpha,\beta\in H^*(\L,\L_0)$ by $\alpha\circledast\beta.$
The construction may be summarized as passing from the left to the right in
the following diagram.
\begin{equation*}
\boxed{\begin{diagram}[size=2em]
\L \times \L & \lTo& \F &\rTo& \L &\lTo^J& \L\times I
\end{diagram}
}\end{equation*}

\quash{ 
\begin{equation*}
\boxed{\begin{diagram}[size=2em]
(\L,\L_0)\times (\L,\L_0) &\lTo& (\F, \F - \F^{>0,>0}) &\rTo& 
(\L,\L-\F^{>0,>0}) &\lTo^J& (\L,\L_0)\times(I,\partial I) \end{diagram}
}\end{equation*}
}

In Figure \ref{fig-def-cohomology}, the first horizontal mapping is an isomorphism if
$i,j > n$ since $\Lambda_0 \cong M$ has dimension $n.$  The mapping $\eta$ is
part of the long exact sequence for the pair $(\L,\L_0) \times (I,\partial I)$
so it is zero.  It follows that the cohomology product defined in \S
\ref{subsec-first-cohomology-product}, which is the composition down the
right hand column, vanishes if $i,j > n.$

\subsection{Remarks} Although the finite dimensional approximation $\M_N$ of
Morse plays only a minor role in this paper, the cohomology product was first
discovered as a product defined on the level cohomology of the subspace
$\A_N \subset \M_N,$ consisting of loops in $\M_N$ for
which the $N$ pieces all have the same length, as follows. Using Poincar\'e
duality, $H^*(\A_N^{\le a}, \A_N^{<a}) \cong H_*(\A_N^{\ge a}, \A_N^{>a}).$
Then apply the (homology) Chas-Sullivan product, followed by Poincar\'e
duality again.  Although the space $\A_N$ is a smooth retraction of $\M_N$
in a neighborhood of the critical set, it has singularities and (what is worse),
spurious critical points, so it is preferable to work with $\M_N.$  When
the construction is translated to $\M_N$ it becomes the Chas-Sullivan product
on Poincar\'e duals, followed by multiplication by a one parameter family
of reparametrizations (the mapping $J$).  This
``level product" on $\M_N$ then extends to a product on $H^*(\M_N, \M_N^{=0})$
that is independent of $N$, and hence it extends to a product on 
$H^*(\L,\L_0).$

\begin{prop}\label{prop-cohomology-commutative}
The cohomology product $\circledast$ is associative.
If $x\in H^i(\L,\L_0)$ and $y \in H^j(\L,\L_0)$ then
\begin{equation}\label{eqn-cohomology-commutative}
y\circledast x = (-1)^{(i+n-1)(j+n-1)}x\circledast y.
\end{equation}\end{prop}

\begin{proof}
First we prove that the product $\circledast$ is commutative.
As in \S \ref{prop-homology-signs} let $\sigma:\L\times\L \to \L \times \L$
switch the two factors and let $\widehat{\chi}_r:\L \to \L$ be the action of
$r \in S^1$ given by $\widehat{\chi}_r(\gamma) = \gamma\circ \chi_r$ where
$\chi_r(t) = r+t \mod 1.$  Then
\begin{equation}\label{eqn-chi-theta}
\chi_{1-s} \circ \rh = \theta_{\half\to(1-s)}\circ \chi_{\half}
\end{equation}
(which may be seen from a direct calculation).  For $0 \le r \le 1$ define
$J_r:\L \times I \to \L$ by
\[ J_r(\gamma,s) = \gamma \circ \chi_{r(1-s)} \circ \rh.\]
Then $J_r\left((\L \times \partial I) \cup (\Lambda_0 \times I)\right)
\subset \L - \phi_{\half}(\F^{>0,>0})$ because  when $s=0$ or $s=1$ the
loop $J_r(\gamma,s)$ stays fixed on either $[0,\half]$ or $[\half,1].$
Therefore the mapping $J^*$ in Figure \ref{fig-def-cohomology} may be
replaced by $J_r^*$ for any $r \in [0,1].$ However, $J_0 = J$ and \begin{align*}
J_1(\gamma,s) &= \gamma \circ \chi_{1-s} \circ \rh = \gamma \circ
\theta_{\half \to (1-s)} \circ \chi_{\half} \\
&= \widehat{\chi}_{\half}(J(\gamma, 1-s)).
\end{align*}
This means that $J_1$ reverses the $s\in I$ coordinate and it switches the
front and back half of any figure eight loop $\gamma \in \F,$ cf. equation
(\ref{eqn-chi-switch}). Consequently, if $i:\F \to \L \times \L$ denotes the
inclusion, we have:
\begin{align*}
(-1)^{i(n-1)}y\circledast x &= \kappa J_0^* (\mu_{\F} \cup i^*(y\times x))\\
&= (-1)^{ij}\kappa J_0^* (\mu_{\F} \cup i^* \sigma^*(x\times y))\\
&= (-1)^{ij}(-1)^n\kappa J_0^* \sigma^*(\mu_{\F} \cup i^*(x \times y)) \\
&= (-1)^{ij}(-1)^n\kappa J_1^*\sigma^*(\mu_{\F} \cup i^*(x\times y)) \\
&= (-1)^{ij}(-1)^n(-1) \kappa J_0^*(\mu_{\F}\cup i^*(x\times y))\\
&= (-1)^{ij+n+1}(-1)^{j(n-1)}x\circledast y.
\end{align*}
giving equation (\ref{eqn-cohomology-commutative}).
The proof of associativity can be found in Appendix \ref{sec-associative}.  \qed

\subsection{Based loop space and Pontrjagin product}\label{subsec-based}
Fix a base point $x_0 \in M$ and let $\Omega =
\Omega_{x_0} = \ev_0^{-1}(x_0)$ be the (based) loop space.  It is a Hilbert
submanifold of codimension $n = \dim(M)$ in $\L$ with a trivial normal bundle.
The Pontrjagin product
\begin{diagram}[size=2em]
\Omega \times \Omega & \rTo^{\bullet}& \Omega
\end{diagram}
(speed up by a factor of 2 and concatenate at time $1/2$) is an embedding.  
Denote its image (which is the ``based loops'' analog of the figure eight space)
by $\F_{\Omega}=\Omega\bullet\Omega.$  It is a Hilbert submanifold of $\Omega,$
with trivial $n$-dimensional normal bundle.  We obtain a fiber (or ``Cartesian'') 
square,
\begin{diagram}[size=2em]
\Omega \times \Omega & \rTo^{h\times h} & \L \times_M \L \\
\dTo_{\bullet} && \dTo_{\phi_{\half}}\\
\Omega & \rTo^{h}& \L.
\end{diagram}
It follows that the Pontrjagin product (which is not necessarily
commutative) and the Chas-Sullivan product are related (\cite{CS} Prop. 3.4) by
\begin{equation}\label{eqn-homology-Pont}
 h^!(a*b)=h^!(a)\bullet h^!(b) \end{equation}
for all $a,b \in H_*(\L),$ where $h^!:H_*(\L) \to H_{*-n}(\Omega)$
denotes the Gysin homomorphism (\ref{eqn-Gysin-homology}). 

A similar construction holds in cohomology.  We use the same
letter $J$ to denote its restriction,
$J:\Omega \times I \to \Omega$ with $I = [0,1].$

\begin{prop}\label{prop-based-cohomology}
Replacing $(\L,\L_0)$ by $(\Omega, x_0)$ in Figure 
\ref{fig-def-cohomology} gives a product
\begin{diagram}[size=2em]
H^i(\Omega,x_0) \times H^j(\Omega,x_0) & \rTo^{\circledast} &
H^{i+j+n-1}(\Omega,x_0) \end{diagram}
such that 
\[h^*(a\circledast b) = h^*(a) \circledast h^*(b)\]
for any $a\in H^i(\L,\L_0)$ and $b \in H^j(\L,\L_0).$
This product is often nontrivial.
Suppose $X,Y \subset \Omega$ are smooth compact oriented submanifolds of 
dimension $i,j$ respectively, and suppose $Z \subset \Omega$ is an oriented
compact submanifold of dimension $i+j+n-1$ such that the mapping
$J: Z \times I \to \Omega$ is transverse to $\F_{\Omega}=\Omega\bullet\Omega$
and such that 
\[J^{-1}(\F_{\Omega}) = (X \bullet Y) \times \textstyle{\left\{
\frac{1}{2} \right\}}.\]  Then
\[ \langle a\circledast b, Z \rangle = \langle a,[X] \rangle \cdot
\langle b, [Y] \rangle\]
for any $a\in H^i(\Omega,x_0)$ and $b \in H^j(\Omega,x_0).$
\end{prop}
\proof  The nontriviality of the product will be taken up in \S
\ref{sec-based-loop-coproduct-nontrivial}.
The following diagram is a fiber (or ``Cartesian'') square so the Thom class of $X\bullet Y$ in 
$Z \times I$ is $J^*(\mu_{\F})$:
\begin{diagram}[size=2em]
Z \times [0,1] & \rTo^J & \Omega & \rTo & \L \\
 \uTo && \uTo && \uTo_{\phi_{\textstyle\half}} \\
(X\bullet Y)\times \left\{\textstyle{\frac{1}{2}}\right\} & \rTo^J & \F_{\Omega}
& \rTo & \F
\end{diagram}
and the following diagram commutes. 
{\footnotesize
\begin{diagram}[size=2em]
H^i(X)\times H^j(Y) &\lTo & H^i(\Omega,x_0)\times H^j(\Omega,x_0) & \lTo & 
H^i(\L,\L_0) \times H^j(\L,\L_0)\\
\dTo && \dTo && \dTo  \\
H^{i+j}(X\times Y) & \lTo& H^{i+j}(\F_{\Omega}, 
\Omega \times x_0 \cup x_0 \times\Omega)
& \lTo & H^{i+j}(\L\times\L, \L \times\L_0 \cup \L_0 \times \L)\\
\dTo_{=} && \dTo_{=} && \dTo \\
H^{i+j}(X\bullet Y) & \lTo & H^{i+j}(\F_{\Omega}, 
\F_{\Omega} - \F_{\Omega}^{>0,>0}) & \lTo & H^{i+j}(\F, \F - \F^{>0,>0})\\
 && \dTo^{\cup \mu_{\F}}_{(\ref{rel-Gysin-cohomology})} && 
\dTo^{\cup \mu_{\F}}_{(\ref{rel-Gysin-cohomology})} \\
\dTo^{\cup J^*(\mu_{\F})}&& H^{i+j+n}(\Omega, \Omega - \F_{\Omega}^{>0,>0}) & \lTo & H^{i+j+n}(\L,\L-\F^{>0,>0})\\
 && \dTo_{J^*} && \dTo_{J^*}\\
H^{i+j+n}(Z\times I, Z \times \partial I) & \lTo & H^{i+j+n}(\Omega \times I, \partial(\Omega \times I)) & \lTo &
H^{i+j+n}(\L\times I, \partial(\L \times I))\\
\dTo_{\cong} &&\dTo_{\cong} && \dTo_{\cong}\\
H^{i+j+n-1}(Z) & \lTo & H^{i+j+n-1}(\Omega,x_0) & \lTo & H^{i+j+n-1}(\L,\L_0)
\end{diagram}  } 
Here, $\partial(\Omega \times I) = x_0 \times I \cup \Omega \times \partial I$ and
$\partial (\L \times I) = \L_0 \times I \cup \L \times \partial I.$ 
\qed

\begin{prop}\label{prop-coproduct-levels}
Let $0 \le a' < a \le \infty$ and $0 \le b' < b < \infty.$  Then the cohomology
product induces a family of compatible products
\begin{align}\label{eqn-co-product-level}
H^i\left(\L^{<a}, \L^{<a'}\right) \times H^j\left(\L^{<b}, \L^{<b'}\right)
&\overset{\circledast}{\longrightarrow} H^{i+j+n-1}\left(\L^{<c}, \L^{<c'}\right)\\
\H^i\left(\L^{\le a}, \L^{\le a'}\right) \times \H^j\left(\L^{\le b}, \L^{\le b'}\right)
&\overset{\circledast}{\longrightarrow} \H^{i+j+n-1}\left(\L^{\le c}, \L^{\le c'}\right)
\label{eqn-co-product2}\\
\H^{i}\left(\L^{\le a}, \L^{<a}\right) \times \H^j\left(\L^{\le b},
\L^{<b}\right) &\overset{\circledast}{\longrightarrow} \H^{i+j+n-1}\left(\L^{\le a+b},
\L^{<a+b}\right).\label{eqn-co-product3} \end{align}
where $c = \min(a+b',a'+b)$ and $c'=a'+b'.$  It is compatible with the homomorphisms
induced by inclusions $\L^{<e} \to \L^{<f}$ whenever $e \le f.$\end{prop}
The construction of the product (\ref{eqn-co-product-level}) will be taken up in the next
few sections.  The existence of (\ref{eqn-co-product2}) and 
(\ref{eqn-co-product3}) follows from (\ref{eqn-co-product-level}) 
and Appendix \ref{subsec-limits}.

\subsection{}\label{subsec-PPAL}
For technical reasons it is easiest to construct this product using the space
$\A$ of loops PPAL, rather than $\L.$ This choice allows us to work with the
mapping $\phi_{\half},$ which is an embedding, rather than more awkward
mapping $\phi_{\min}.$  It also allows us to keep
track of the effect of the function $J$ on the energy; see (\ref{eqn-J-energy}).
Recall (\S \ref{subsec-free-loop}) if
$\alpha \in \A$ is a loop parametrized proportionally to arclength
then $F(\alpha) = \sqrt{E(\alpha)} = L(\alpha)$ is its length.  Let $\F_{\A}
= \A \times_M \A$ be the associated figure eight space consisting of pairs
of composable loops, each parametrized proportionally to arclength.  Let
$\A_{\frac{1}{2}}$ be the set of loops $\alpha\in \L$ such that $\alpha|[0,
1/2]$ is PPAL and $\alpha|[1/2,1]$ is PPAL.
Similarly let
\begin{align*}
\F_{\A}^{<u,<v} &= \left\{ (\alpha,\beta)\in \A \times_M \A:\
L(\alpha)<u \ \text{and }\ L(\beta)<v.  \right\} \\
\A_{\frac{1}{2}}^{<u,<v} &= \left\{ \alpha \in \A_{\frac{1}{2}}:\
L(\alpha|[0,1/2])<u\ \text{and }\ L(\alpha|[1/2,1] < v. \right\}
\end{align*}
and similarly for $\A_{\half}^{=u,=v},$ etc.
Then the mapping $\phi_{\half}$ restricts to a closed embedding
\begin{equation}\label{eqn-A-phi}
 \phi_{\half}: \F_{\A}^{<u,<v} \to \A_{\half}^{<u,<v}.\end{equation}

\begin{lem}  The orientation of $M$ induces a Thom isomorphism
\begin{equation}\label{eqn-A-isomorphism}
 H^i(\F_{\A}^{<u,<v}) \cong H^{i+n}(\A_{\half}^{<u,<v}, \A_{\half}^{<u,<v} - \phi_{\half}(\F_{\A}^
{<
u,<v}))\end{equation}
for the image, $\phi_{\half}(\F_{\A}^{<u,<v}).$ If $Z \subset \F_{\A}^{<u,<v}$
is a closed subset then this restricts to a relative Thom isomorphism,
\begin{equation}\label{eqn-A-isomorphism2}
H^i(\F_{\A}^{<u,<v}, \F_{\A}^{<u,<v} - Z) \cong H^{i+n}(\A_{\half}^{<u,<v},
\A_{\half}^{<u,<v}-Z).  \end{equation}
\end{lem}
\proof
 The space $\phi_{\half}(\F_{\A}^{<u,<v})$ may be described as the pre-image of the diagonal under
the
mapping
\[ (\ev_0,\ev_{\half}):\A_{\half}^{<u,<v} \to M\times M.\]
If $\A_{\half}$ were a Hilbert manifold, this would imply the existence of a normal
bundle and tubular neighborhood for  $\phi_{\half}(\F_{\A}^{<u,<v})$ in
$\A_{\half}^{<u,<v}.$  Unfortunately $\A_{\half}$ is probably not a Hilbert manifold,
and even though it may be a Banach manifold, we do not know of any standard reference
for the existence of tubular neighborhoods which can be applied in this setting.
However we have the following commutative diagram, where the vertical maps are homotopy
equivalences, cf. Proposition \ref{prop-PPAL-space}:
\begin{diagram}[height=1.5em]
\phi_{\half}(\F^{<u,<v}) & \rTo & \L_{\half}^{<u,<v}&     & \\
                        &      &                    &\rdTo&  \\
\dTo^Q\uTo                &      &\dTo^Q\uTo         &     &\ M\times M \\
                        &      &                    &\ruTo&  \\
\phi_{\half}(\F_{\A}^{<u,<v}) &\rTo&\A_{\half}^{<u,<v}&   &
\end{diagram}
Here, $\L_{\half}^{<u,<v}$ denotes the set of all $\alpha \in \L$ such that
$E(\alpha_1) < u^2$ and $E(\alpha_2) < v^2$ where $\alpha_1(t) = \alpha(2t)$
($0 \le t \le 1/2$) and $\alpha_2(t) = \alpha(2t-1)$ ($1/2 \le t \le 1$).
(In other words, if $\alpha|[0,1/2]$ and
$\alpha|[1/2,1]$ are both expanded into paths defined on $[0,1]$ then their respective
energies are bounded by $u^2$ and $v^2$.)  Since $\L$ is a Hilbert manifold the same
holds for $\L_{\half}^{<u,<v}$ hence $\phi_{\half}(\F^{<u,<v})$ has a tubular
neighborhood and normal bundle in $\L^{<u,<v}$ and we have a Thom isomorphism
\[ H^i(\F^{<u,<v}) \cong H^{i+n}(\L^{<u,<v}, \L^{<u,<v} - \phi_{\half} (\F^{<u,<v})).\]
The vertical homotopy equivalence in this diagram assigns to any $\alpha \in \L_{\half}^{<u,<v}$
the
same curve but with $\alpha|[0,1/2]$ reparametrized proportionally to arclength and with
$\alpha|[1/2,1]$ similarly reparametrized.  It restricts to a homotopy equivalences
$\phi_{\half}(\F^{<u,<v}) \to \phi_{\half}(\F_{\A}^{<u,<v})$ and also
\[ \L_{\half}^{<u,<v} - \phi_{\half}(\F^{<u,<v}) \to \A_{\half}^{<u,<v} -
\phi_{\half}(\F_{\A}^{<u,<v})\]
because the points $\alpha(0), \alpha(1/2)$ are fixed.  The Thom isomorphism
(\ref{eqn-A-isomorphism}) follows.  The relative Thom isomorphism
(\ref{eqn-A-isomorphism2}) follows as in Proposition \ref{prop-Thom}.\qed

\subsection{}\label{subsec-aboutJ}
Let $I=[0,1]$ denote the unit interval.
The mapping $J:\L \times I\to\L$ restricts to a mapping
$J_{\A}:\A \times I \to \A_{\half}$ which factors through the quotient
$\A \times I/(\A_0 \times I)$ that is obtained from $\A \times I$ by
identifying $\A_0 \times I$ to a point.  The resulting mapping
\[ \A \times I/(\A_0\times I) \to \A_{\half}\]
is a homeomorphism, and in fact the inverse mapping
can be described as follows.  Let $\alpha \in \A_{\half} - \A_0.$
Let $L_0$ denote the length of the segment $\alpha|[0,\half]$ (which is PPAL) and
let $L_1$ denote the length of the segment $\alpha|[\half,1].$  Set $s = L_0/(L_0+L_1).$
Assume for the moment that $0<s<1.$
Let $\hr= (\rh)^{-1}$ be the inverse function to
$\rh;$ it is linear on $[0,s]$, linear on $[s,1]$, and takes the values
$\hr(0)=0$, $\hr(s)= \half$, $\hr(1)=1.$ Then
$\alpha \circ \hr$ is PPAL throughout the interval $[0,1]$
so we may set
\[ J_{\A}^{-1}(\alpha) = (\alpha \circ \hr, s).\]
If $s=0$ (resp. $s=1$) then this formula still makes sense because in this
case, the loop $\alpha$ will be constant on $[0,\half]$ (resp. on $[\half,1]$).
If $\alpha_0,\alpha_1 \in \A$ are composable loops, not both constant,
then the composed loop $\phi_{\half}(\alpha_0,\alpha_1)\in \A_{\half}$ and
\begin{equation}\label{eqn-K}
J_{\A}^{-1}(\phi_{\half}(\alpha_0,\alpha_1)) =(\phi_s(\alpha_0,\alpha_1),s) =
(\phi_{\min}(\alpha_0,\alpha_1),s)
\end{equation}
where $s= L(\alpha_0)/(L(\alpha_0)+L(\alpha_1))$ is the unique energy minimizing
value, cf. Lemma \ref{lem-phi}.

It follows that {\em if $A,B \subset \A$ then the mapping
$J_{\A}:\A \times [0,1] \to \A_{\half}$ takes}
\begin{equation}\label{eqn-J-sets}
 (\A - A*B) \times [0,1] \ \text{ into }\ \A_{\half} - \phi_{\half}(A \times_MB).
\end{equation}
For $J(\A_0 \times [0,1]) = \A_0$ which is contained in the right hand side, so it suffices
to check that $J(A*B \times [0,1]) \supset \phi_{\half}(A \times_M B),$
which follows from (\ref{eqn-K}).

Similarly the mapping $J_{\A}$ satisfies
\begin{equation}\label{eqn-J-energy}
J_{\A}(\alpha,s) \in \A_{\half}^{=sL(\alpha),=(1-s)L(\alpha)}
\end{equation}
which means the following: if we express $J_{\A}(\alpha,s) =
\phi_{\half}(\beta_1,\beta_2)$ as a composition of two (not necessarily
closed) paths $\beta_1,\beta_2,$ each PPAL, and joined at time $1/2,$ then
$L(\beta_1)=sL(\alpha)$ and $L(\beta_2) = (1-s)L(\alpha).$

\subsection{}Define
\begin{align*} T^{<a,<b} &= J_{\A}^{-1}\left( \F_{\A}^{<a,<b} \right) \\
&= \left\{(\alpha,s) \in \A\times I:\
sL(\alpha) <a \text{ and } (1-s)L(\alpha) < b \right\} \\
T^{[a',a),[b',b)}&=
J_{\A}^{-1}\left( \F_{\A}^{[a',a),[b',b)}\right)\\
&= \left\{(\alpha,s)\in \A_{\half}\times I:\
a'\le sL(\alpha)<a, \text{ and } b'\le (1-s)L(\alpha) < b\right\}. \end{align*}
Figure \ref{fig-T-space} consists of three diagrams of $L=\sqrt{E}$ versus
$s \in [0,1]$ illustrating the
curves $sL=a$ and $(1-s)L=b$ that occur in the definition of $T^{<a,<b}.$  These
curves intersect at the point with coordinates $s = a/(a+b)$ and $L = a+b.$  The
diagrams on the right illustrate the corresponding regions for
$T^{[a'.a),[b',b)}.$  (The interval $I'$ is defined in the next paragraph,
\S \ref{subsec-construction-coprod}.)

\medskip
\setlength{\unitlength}{1.1pt}
\newcommand{\q}{\qbezier}
\newsavebox{\curvea}\newsavebox{\curveb}\newsavebox{\curvec}
\newsavebox{\curved}
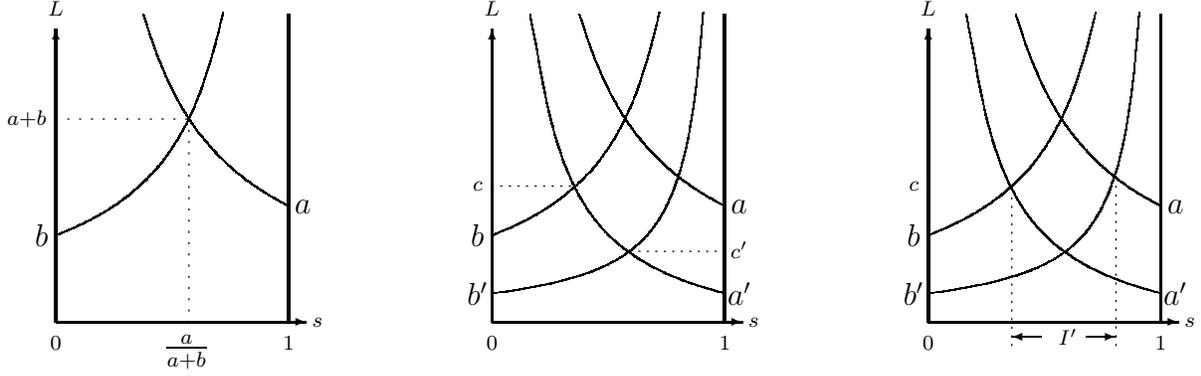
\begin{figure}[h!]\begin{center}
\begin{picture}(385,110)(0,-10)
\linethickness{.9pt}
\multiput(0,0)(150,0){3}{\line(0,1){100}}  
\multiput(0,0)(150,0){3}{\line(1,0){80}}
\multiput(80,0)(150,0){3}{\line(0,1){106}}
\multiput(0,0)(150,0){3}{\vector(1,0){86}}
\multiput(90,0)(150,0){3}{\makebox(0,0){$\scriptstyle{s}$}}
\multiput(0,0)(150,0){3}{\vector(0,1){101}}
\multiput(0,108)(150,0){3}{\makebox(0,0){$\scriptstyle{L}$}}
\multiput(0,-7)(150,0){3}{\makebox(0,0){$\scriptstyle{0}$}}
\multiput(80,-7)(150,0){3}{\makebox(0,0){$\scriptstyle{1}$}} 
\linethickness{.25pt}
\savebox{\curvea}(80,100)[bl]{
\q(0,30)(5,32)(10,34.29)
\q(5,32)(10,34.29)(15,36.93)
\q(10,34.29)(15,36.93)(20,40)
\q(15,36.93)(20,40)(25,43.62)
\q(20,40)(25,43.62)(30,48)
\q(25,43.62)(30,48)(35,53.1)
\q(30,48)(35,53.1)(37.5,56.46)
\q(35,53.1)(37.5,56.46)(40,60)
\q(37.5,56.46)(40,60)(42.5,64)
\q(40,60)(42.5,64)(45,68.55)
\q(42.5,64)(45,68.55)(47.5,73.83)
\q(45,68.55)(47.5,73.83)(50,79.98)
\q(47.5,73.83)(50,79.98)(52.5,87.87)
\q(50,79.98)(52.5,87.87)(55,96)
\q(52.5,87.87)(55,96)(57.5,106.65)
} 

\savebox{\curveb}(80,100)[bl]{%
\q(0,10)(5,10.6)(10,11.43)
\q(5,10.6)(10,11.43)(15,12.31)
\q(10,11.43)(15,12.31)(20,13.33)
\q(15,12.31)(20,13.33)(25,14.54)
\q(20,13.33)(25,14.54)(30,16)
\q(25,14.54)(30,16)(35,17.7)
\q(30,16)(35,17.7)(40,20)
\q(35,17.7)(37.5,18.82)(40,20)
\q(37.5,18.82)(40,20)(42.5,21.33)
\q(40,20)(42.5,21.33)(45,22.85)
\q(42.5,21.33)(45,22.85)(47.5,24.61)
\q(45,22.85)(47.5,24.61)(50,26.66)
\q(47.5,24.61)(50,26.66)(52.5,29.09)
\q(50,26.66)(52.5,29.09)(55,32)
\q(52.5,29.09)(55,32)(57.5,35.55)
\q(55,32)(57.5,35.55)(60,40)
\q(57.5,35.55)(60,40)(62.5,45.71)
\q(60,40)(62.5,45.71)(65,53.3)
\q(62.5,45.71)(65,53.3)(66,57.14)
\q(65,53.3)(66,57.14)(67,61.54)
\q(66,57.14)(67,61.54)(68,66.66)
\q(67,61.54)(68,66.66)(69,72.72)
\q(68,66.66)(69,72.72)(70,80)
\q(69,72.72)(70,80)(71,88.88)
\q(70,80)(71,88.88)(72,100)
\q(71,88.88)(72,100)(72.5,106.66)} 

\savebox{\curvec}(80,100)[bl]{
\q(30,106)(31,103.22)(32,100)
\q(31,103.22)(32,100)(33,96.96)
\q(32,100)(33,96.96)(34,94.11)
\q(33,96.96)(34,94.11)(35,91.43)
\q(34,94.11)(35,91.43)(37.5,85.33)
\q(35,91.43)(37.5,85.33)(40,80)
\q(37.5,85.33)(40,80)(42.5,75.29)
\q(40,80)(42.5,75.29)(45,71.11)
\q(42.5,75.29)(45,71.11)(47.5,67.37)
\q(45,71.11)(47.5,67.37)(50,64)
\q(47.5,67.37)(50,64)(55,58.18)
\q(50,64)(55,58.18)(60,53.33)
\q(55,58.18)(60,53.33)(65,49.23)
\q(60,53.33)(65,49.23)(70,45.71)
\q(65,49.23)(70,45.71)(75,42.66)
\q(70,45.71)(75,42.66)(80,40)} 

\savebox{\curved}(80,100)[bl]{
\q(13,106)(15,96)(17.5,81.43)
\q(15,96)(17.5,81.43)(20,70)
\q(17.5,81.43)(20,70)(22.5,61.11)
\q(20,70)(22.5,61.11)(25,54)
\q(22.5,61.11)(25,54)(27.5,48.18)
\q(25,54)(27.5,48.18)(30,43.33)
\q(27.5,48.18)(30,43.33)(32.5,39.23)
\q(30,43.33)(32.5,39.23)(35,35.71)
\q(32.5,39.23)(35,35.71)(37.5,32.66)
\q(35,35.71)(37.5,32.66)(40,30)
\q(37.5,32.66)(40,30)(42.5,27.65)
\q(40,30)(42.5,27.65)(45,25.55)
\q(42.5,27.65)(45,25.55)(50,22)
\q(45,25.55)(50,22)(55,19.09)
\q(50,22)(55,19.09)(60,16.66)
\q(55,19.09)(60,16.66)(65,14.61)
\q(60,16.66)(65,14.61)(70,12.85)
\q(65,14.61)(70,12.85)(75,11.33)
\q(70,12.85)(75,11.33)(80,10)} 
\put(0,0){\usebox{\curvea}}
\put(0,0){\usebox{\curvec}}
\dottedline{4}(0,70)(45,70)
\dottedline{4}(45.7,70)(45.7,0)
\put(-5,30){\makebox(0,0){$b$}}
\put(85,40){\makebox(0,0){$a$}}
\put(-10,70){\makebox(0,0){$\scriptstyle{a+b}$}}
\put(45,-10){\makebox(0,0){$\frac{a}{a+b}$}}
\put(150,0){\usebox{\curvea}} \put(150,0){\usebox{\curvec}}
\put(150,0){\usebox{\curveb}} \put(150,0){\usebox{\curved}}
\put(145,30){\makebox(0,0){$b$}} \put(145,10){\makebox(0,0){$b'$}}
\put(235,40){\makebox(0,0){$a$}} \put(235,10){\makebox(0,0){$a'$}}
\dottedline{3}(150,47)(177,47) \put(145,47){\makebox(0,0){$\scriptstyle{c}$}}
\dottedline{3}(197.5,24.5)(230,24.5) \put(235,25){\makebox(0,0){$\scriptstyle{c'}$}}
\put(300,0){\usebox{\curvea}}
\put(300,0){\usebox{\curveb}}
\put(300,0){\usebox{\curvec}}
\put(300,0){\usebox{\curved}}
\dottedline{3}(328.7,-8)(328.7,46) 
\dottedline{3}(364.5,-8)(364.5,50) 
\put(338.7,-5){\vector(-1,0){10}} \put(354.4,-5){\vector(1,0){10}}
\put(348,-5){\makebox(0,0){$\scriptstyle{I'}$}}
\put(295,47){\makebox(0,0){$\scriptstyle{c}$}}
\put(295,30){\makebox(0,0){$b$}}\put(385,40){\makebox(0,0){$a$}}
\put(295,10){\makebox(0,0){$b'$}}\put(385,10){\makebox(0,0){$a'$}}
\end{picture}

\end{center}\label{fig-T-space}

\caption{The regions $T^{<a,<b}$ and $T^{[a',a),[b',b)}$}\end{figure}

\subsection{}\label{subsec-construction-coprod}
The product (\ref{eqn-co-product-level}) is constructed in several steps.
First, use the cross product,
\[
H^i(\A^{<a},\A^{<a'}) \times H^j(\A^{<b}, \A^{<b'}) \to
H^{i+j}(\A^{<a} \times \A^{<b}, \A^{<a}\times \A^{<b'} \cup \A^{<a'}\times \A^{<
b})\]
then restrict to
\[H^{i+j}(\F_{\A}^{<a,<b}, \F_{\A}^{<a,<b'}\cup \F_{\A}^{<a',<b})
=
H^{i+j}(\F_{A}^{<a,<b}, \F_{\A}^{<a,<b} - \F_{\A}^{[a',a),[b',b)}).\]
Using the embedding (\ref{eqn-A-phi}) and the Thom isomorphism we arrive at
\begin{align*}
&H^{i+j}\left(
\phi_{\half}(\F_{\A}^{<a,<b}),
\phi_{\half}(\F_{\A}^{<a,<b}) -
\phi_{\half}(\F_{\A}^{[a',a),[b',b)})\right) \\ \cong
&H^{i+j+n}\left(\A_{\half}^{<a,<b}, \A_{\half}^{<a,<b} - \phi_{\half}
(\F_{\A}^{[a',a),[b',b)}) \right).\end{align*}
Pulling back under $J_{\A}$ gives a class in
\begin{equation}\label{eqn-Tspace}
H^{i+j+n}\left( T^{<a,<b}, T^{<a,<b} - T^{[a',a),[b',b)}.
\right)\end{equation}
Let $I'$ denote the interval $\left[ \frac{a'}{a'+b}, \frac{a}{a+b'} \right].$
Then $\A^{<c} \times I' \subset T^{<a,<b}$ and
\[ ( \A^{<c} \times \partial I' ) \cup ( \A^{<c'} \times I' )
\subset T^{<a.<b} - T^{[a',a),[b',b)} \]
where $c = \min(a+b',a'+b)$ and $c'=a'+b'.$  In other words, the mapping $J_{\A}$
restricts to a map of pairs,
\begin{equation}\label{eqn-JA-pairs}
J_{\A}:(\A^{<c},\A^{<c'})\times(I', \partial I') \to
\left( \phi_{\half}\left(\F_{\A}^{<a,<b}\right),
\phi_{\half}\left(\F_{\A}^{<a,<b}\right)-\phi_{\half}\left(
\F_{\A}^{[a',a),[b',b)} \right) \right)
\end{equation}
Therefore the class in (\ref{eqn-Tspace}) pulls back to a class in
\[ H^{i+j+n}\left( (\A^{<c}, \A^{<c'}) \times (I', \partial I') \right)
\cong H^{i+j+n-1}\left( \A^{<c}, \A^{<c'}\right) \]
as claimed.  This completes the proof of Proposition
\ref{prop-coproduct-levels}\qed

Taking $a,b = \infty$ and $a',b'=0$ gives an equivalent construction of the
cohomology product $\circledast$ using the space $\A$ rather than $\L.$  This
will also be important in the next section.

\section{Support and critical levels}\subsection{}
As in the previous sections, we take cohomology with coefficients in
$G = \mathbb Z$ if $M$ is orientable, or $G = \mathbb Z/(2)$ otherwise,
but we suppress mention of $G$ in our notation for cohomology. 
Proposition \ref{prop-coproduct-levels} gives:
\begin{prop}\label{prop-coprod-critical}
If $\alpha,\beta \in H^*(\L,\L_0)$ then
$\Cr(\alpha \circledast \beta) \ge \Cr(\alpha) + \Cr(\beta).$ \qed
\end{prop}

As in \S \ref{subsec-free-loop} and \S \ref{subsec-PPAL} let $\A$ be the
set of loops parametrized
proportionally to arclength, let $\A_0 = \L_0$ be the constant loops and let
$\A_{\half}\subset \L$ be the collection of those loops which are PPAL on
$[0,1/2]$ and are PPAL on $[1/2,1].$  We have continuous mappings
\[J_{\A}:\A \times [0,1] \to \A_{\half},\quad
\phi_{\half}:\A \times_M \A \to \A_{\half},\ \text{ and }\
\phi_{\min}:\A \times_M\A \to \A \subset \A_{\half}.\]
\begin{prop}\label{prop-cohomology-supports}
Suppose $\alpha \in H^i(\A,\A_0)$ is supported on a closed set $A \subset \A - \A_0.$
Suppose $\beta \in \H^j(\A,\A_0)$ is supported on a closed set $B \subset \A - \A_0.$
Then $\alpha \circledast \beta$ is supported on the closed set $A*B =
\phi_{\min}(A \times_MB)\subset \A - \A_0.$
\end{prop}

\proof
Using (\ref{eqn-J-sets}) we obtain a cohomology product
\[ H^i(\A, \A - A) \times H^j(\A, \A - B) \to H^{i+j+n-1}(\A, \A - A*B)\]
as the composition
\begin{diagram}[size=2em]
H^{i+j}(\A \times \A, \A \times \A - A \times B) &\rTo&
H^{i+j}(\F_{\A}, \F_{\A} - A \times_M B)\\
&& \dTo_{\phi_{\half}}^{(\ref{eqn-A-isomorphism2})}\\
H^{i+j+n}( (\A, \A - A*B) \times (I, \partial I)) &\lTo^{J_{\A}^*}&
H^{i+j+n} (\A_{\half}, \A_{\half}-\phi_{\half}(A \times_M B))
\end{diagram}
where $I = [0,1].$  \qed

\subsection{} \label{subsec-cup}
We remark that the analogous statement for the cup product in
cohomology says that $\alpha \smile \beta$ is supported on the intersection
$A \cap B.$  In particular, for any $0 \le a' < a \le \infty$
and $0 \le b' <b \le \infty$ the cup product gives mappings (with coefficients
in $\mathbb Z$),
\begin{equation*}
H^i(\L^{\le a}, \L^{\le a'})\times H^j(\L^{\le b}, \L^{\le b'})  \to
H^{i+j}(\L^{\le \min(a,b)}, \L^{\le \max(a',b')}).
\end{equation*}

\subsection{On a question of Eliashberg}\label{subsec-Eliashberg}
In a lecture at Princeton University in June 2007, Y. Eliashberg asked the following
question.  Let $M$ be a smooth compact Riemannian manifold and let $\L$ be its
free loop space.  Given $0 < t \in \mathbb R$ let $d(t)$ be the maximal degree of an
essential homology class at level $t,$ that is,
\begin{equation}\label{eqn-d}
 d(t) = \max \left\{ k:\ \text{Image}\left(H_k(\L^{\le t},\mathbb Q) \to
H_k(\L;\mathbb Q)\right) \ne 0 \right\}.\end{equation}
Does there exist a constant $C \in \mathbb R,$ independent of the metric, so 
that for all $t_1,t_2 \in \mathbb R^{+}$ the following holds:
\[ d(t_1+t_2) \le d(t_1) + d(t_2) + C?\]
Inequalities in the opposite direction are known.  If the cohomology ring
$(H^*(\L,\L_0;\mathbb Q), \circledast)$ is finitely generated, then we are 
able to give an affirmative answer to this question.  Moreover, in 
\S \ref{sec-coH-closed} we show that
the cohomology ring (with rational coefficients) is indeed finitely generated 
if $M$ is orientable and admits a metric in which all geodesics are closed. 
(If the orientability condition is dropped then the cohomology ring with 
$\mathbb Z/(2)$ coefficients is finitely generated.)  This includes the case 
of spheres and projective spaces.

\begin{thm}  Let $M$ be a smooth Riemannian $n$-dimensional manifold. Fix a 
coefficient field $G,$ and for $t \in \mathbb R^{+}$ define $d(t) = d(t;G)$ by 
(\ref{eqn-d}), but replacing the coefficients $\mathbb Q$ with the field $G.$  
Assume the cohomology ring $(H^*(\L,\L_0;G),\circledast)$ is finitely generated, 
with all generators having degree $\le g.$  If $t_1,t_2 \in \mathbb R^{+}$ then
\begin{equation}\label{eqn-Eliashberg}
d(t_1+t_2) \le d(t_1) + d(t_2) + 2n+g-2. \end{equation}
\end{thm}
\proof 
Let $t_1,t_2, t_3\in \mathbb R^{+}$ and let $d_i=d(t_i;G)$ for $i = 1,2,3.$  We
will show:  if
\begin{equation}\label{eqn-d3}
d_3 > d_1 + d_2 + 2n + g -2\end{equation}
then $t_3 > t_1 + t_2.$  From the definition (\ref{eqn-d}) and since 
$d_3>n= \dim(\L_0),$  there exist non-zero homology classes $z'$ and $z$, with:
\begin{diagram}[size=2em]
H_{d_3}(\L^{\le t_3}, \L_0) & \rTo_{i_*} & H_{d_3}(\L,\L_0) \\
   z' & \rTo & z = i_*(z')
\end{diagram}
Let $Z\in H^{d_3}(\L,\L_0)$ be a cohomology class with non-zero Kronecker product,
\begin{equation}\label{eqn-Kronecker}
\langle Z, z \rangle \ne 0.  \end{equation}
The class $Z$ is a sum of products of generators of the ring $\left(
H^*(\L,\L_0), \circledast \right),$ and at least one term in this sum has a
non-zero Kronecker product with $z.$  Replacing $Z$ by this term, it may be
expressed as a product of generators,
\begin{equation}\label{eqn-decompose-Z}
Z = U_1 \circledast U_2 \circledast \cdots \circledast U_q \end{equation}
with each $\deg(U_i) \le g.$

We claim there exist $X,Y \in H^*(\L,\L_0)$ such that
$Z = X \circledast Y,$ with $\deg(X) \ge d_1+1$ and  $\deg(Y) \ge d_2+1.$  
For, choose $p$ so that 
\begin{align*}
\deg(U_1 \circledast U_2 \circledast \cdots \circledast U_{p-1}) &\le d_1 \\
\deg(U_1\circledast U_2 \circledast \cdots \circledast U_{p}) &\ge d_1+1
\end{align*}
Take $X =U_1 \circledast \cdots \circledast U_p$ and $Y = U_{p+1} \circledast \cdots \circledast U_q.$  Then $\deg(U_p) \le g$ so $\deg(X) \le d_1 + g + n -1,$ 
while $\deg(X\circledast Y)=d_3 > d_1+d_2+2n+g-2,$ so $\deg(Y) > d_2.$

Using this claim and (\ref{eqn-d}), and setting $j = \deg(X)$ and $k=\deg(Y),$
there exists $\widehat{X}$ which maps to $X$ in the following exact sequence,
\begin{diagram}[size=2em]
 H^j(\L, \L^{\le t_1}) & \rTo & H^j(\L,\L_0) & \rTo & H^j(\L^{\le t_1},\L_0)\\
\widehat{X} & \rTo & X
\end{diagram}
Similarly the class $Y$ has some lift $\widehat{Y} \in H^k(\L,\L^{\le t_2}).$
Then $\widehat{X} \circledast \widehat{Y}\in H^{d_3}(\L, \L^{t_1+t_2})$ maps to 
$Z.$  But this implies that $t_3 > t_1+t_2.$  Otherwise, $\L^{\le t_3} \subset
\L^{\le t_1+t_2}$ so in the following diagram,
\begin{diagram}[size=2em]
&H^{d_3}(\L,\L^{\le t_3}) & \rTo& H^{d_3}(\L,\L_0) & \rTo_{i^*}& 
H^{d_3}(\L^{\le t_3},\L_0) \\
& \uTo && \uTo && \uTo \\
& H^{d_3}(\L, \L^{\le t_1+t_2}) & \rTo & H^{d_3}(\L,\L_0) 
& \rTo & H^{d_3} (\L^{\le t_1+t_2},\L_0)\\
& \widehat{X} \circledast \widehat{Y} &\rTo & Z
\end{diagram}
we would have $\langle Z,z\rangle = \langle Z, i_*(z')\rangle = \langle i^*(Z),z' 
\rangle = 0$
which contradicts (\ref{eqn-Kronecker}). \qed

\section{Level nilpotence for cohomology}\label{sec-nilpotent-cohomology}
\subsection{}
We say that a class $\alpha \in H^i(\L,\L_0)$ is {\em level nilpotent} if there
exists $m$ so that $\Cr(\alpha^{\circledast m}) > m\Cr(\alpha).$
We say that a class $\beta \in \H^i(\L^{\le a},\L^{<a})$ is level nilpotent if
there exists $m$ so that $\beta^{\circledast m}=0$ in $\H^{mi+(m-1)(n-1)}
(\L^{\le ma}, \L^{<ma}).$

Let us say that two classes $\alpha \in H^i(\L, \L_0)$ and $\beta \in
\H^i(\L^{\le a}, \L^{<a})$ are {\em associated} if there exists an
{\em associating class} $\omega \in \H^i(\L, \L^{< a})$ with
\begin{diagram}[size=1.7em]
\omega & \rTo & \alpha && \H^i(\L,\L^{< a}) & \rTo& H_i(\L,\L_0)\\
\dTo   &         &        &\text{in }\ & \dTo    &     & \\
\beta  &         &        &         &\H_i(\L^{\le a},\L^{<a}) &&
\end{diagram}
Then $\Cr(\alpha) > a$ if and only if $\alpha$ is associated to the zero class
$\beta = 0 \in \H^i(\L^{\le a}, \L^{<a}).$

\begin{lem}
Suppose $\alpha \in H^i(\L,\L_0)$ and $\beta \in \H^i(\L^{\le a}, \L^{<a})$
are associated, where $a = \Cr(\alpha).$  If $\beta$ is level nilpotent,
then $\alpha$ is also level-nilpotent.\end{lem}
\proof The proof is exactly parallel to that of Lemma \ref{lem-nilpotent}.\qed

\begin{thm}\label{thm-nilpotent-cohomology}
Let $M$ be a compact $n$ dimensional Riemannian manifold and suppose that all
critical points of the function $F = \sqrt{E}:\L \to \mathbb R$ are
nondegenerate (i.e. they lie on isolated nondegenerate critical orbits).  
If $M$ is orientable let $G = \Z$, otherwise let
$G = \Z/(2).$  Then every class $\alpha \in H^i(\L,\L_0;G)$ is
level-nilpotent and every class $\beta \in H^i(\L^{\le a}, \L^{<a};G)$
is level nilpotent (for any $i>0$ and any $a\in \mathbb R$).
\end{thm}
\proof  The proof is similar to that of Theorem \ref{thm-nilpotence}. 
\qed

\section{Level products in the nondegenerate case}\label{sec-nondegenerate}
\subsection{}  Throughout this section homology and cohomology will be taken
with coefficients in $G = \mathbb Z.$  Let $\Sigma\subset\Lambda$ be a 
nondegenerate critical orbit of index $\lambda$ and let $U\subset\L$ be a
sufficiently small neighborhood of $\sigma.$  Assume the negative bundle
$\Gamma \to \Sigma$ is orientable.  Then the (local, level) homology groups are
\begin{equation}\label{eqn-table-local-level}
\H_i(\L^{<c}\cup \Sigma, \L^{<c}) \cong
H_{i}(\L ^{\leq c}\cap U,\L ^{<c}\cap U)\cong 
\begin{cases}\mathbb Z&\text{if}\ i=\lambda,\lambda+1\\
             0&\text{otherwise}\end{cases}\end{equation}
and the same holds for the cohomology groups $H^i(\L ^{\leq c}\cap U,\L ^{<c}\cap U).$

\quash{
\begin{center}%
\begin{figure}[!h]\begin{tabular}
[c]{|l||c|c|c|c|}\hline
$\phantom{H^{H^T}}$ & $H_{\lambda}$ & $H_{\lambda+1}$ & $H^{\lambda}$ &
$H^{\lambda+1}$\\\hline\hline
$G=\mathbb{Z}/(2)$ & $\mathbb{Z}/(2)$ & $\mathbb{Z}/(2)$ & $\mathbb{Z}/(2)$ &
$\mathbb{Z}/(2)$\\\hline
$G=\mathbb{Z}$, $\Gamma$ orientable & $\mathbb{Z}$ & $\mathbb{Z}$ &
$\mathbb{Z}$ & $\mathbb{Z}$\\\hline
$G=\mathbb{Z}$, $\Gamma$ not orientable & $\mathbb{Z}/(2)$ & $0$ & $0$ &
$\mathbb{Z}/(2)$\\\hline
\end{tabular}\caption{Level homology and cohomology groups}\label{fig-level-groups}
\end{figure}
\end{center}
}
\subsection{}
Now suppose $\gamma \in \L$ is a prime geodesic, all of whose
iterates are nondegenerate.  Let $a$ be its length.  
Let $\gamma^r$ denote the $r$-fold iterate; $\lambda_r$
its Morse index; and let $\Sigma_r \subset \L$ be its $S^1$-saturation.
Assume the negative bundle $\Gamma_r$ over $\Sigma_r$ is
orientable\footnote{If $M$ is orientable and $\gamma$ is prime then 
$\Gamma_1$ is orientable.   This follows from
the argument of \cite{Rademacher} \S 2.2, which is reproduced in the
proof of Proposition \ref{prop-orientable} below.  In addition,
$\lambda_1$ and $\lambda_2$ have the same parity $\Longleftrightarrow$
all the negative bundles $\Gamma_m$ are orientable $\Longleftrightarrow$
all the $\lambda_i$ have the
same parity, cf. \cite{Rademacher, Wilking}.\label{foot1}}, 
and let $\sigma_r, \overline{\tau}_r, \overline{\sigma}_r, \tau_r$
be generators for the local level (co)homology classes, that is,
\begin{equation}\label{eqn-local-level-classes}
 \sigma_r \in H_{\lambda_r},\quad \overline{\sigma}_r \in H_{\lambda_r+1},\quad
\ov{\tau}_r \in H^{\lambda_r},\quad {\tau}_r \in H^{\lambda_r+1}.\end{equation}

As a consequence of the nilpotence results from \S \ref{sec-nilpotence}
and \S \ref{sec-nilpotent-cohomology}, the index $\lambda_r$ can be neither 
minimal nor maximal for all $r$ (in the language of Proposition
\ref{prop-index-iterates}), and the local level homology and cohomology rings 
\begin{equation*}
\left(\oplus H_i(\L^{<ar} \cup \Sigma_r, \L^{<ar};G), * \right)\
\text{ and }\
 \left( \oplus H^i(\L^{<ar}\cup \Sigma_r, \L^{<ar};G), \circledast \right)
\end{equation*}
are not finitely generated.  
However, , if the index growth is minimal up to the $n$-th iterate (with $\lambda_r =
\lambda_r^{\min}$ for all $r \le n$) then nontrivial (level) homology products exist, 
and if the index growth is maximal up to the $2n$-th iterate
(with $\lambda_r = \lambda_r^{\max}$ for all $r \le 2n$) then nontrivial (level) 
cohomology products exist, as described in the following theorem.  

\begin{thm}\label{prop-level-products}
Assume the manifold $M$ is orientable and the negative
bundle $\Gamma_r$ is orientable for all $r.$  Assume $r \ge 2.$ Then
the following statements hold in the local level (co)homology group
$H(\L^{< ra}\cup \Sigma_r, \L^{<ra};\mathbb Z).$ \begin{enumerate}
\item $(\sigma_1)^{*r}=0$ and $(\tau_1)^{*r}=0.$
\item Some further products are described in the following tables:

\begin{center}\begin{figure}[!h]{\small
\begin{tabular}{|c|c|c|}\hline
& $(\ov{\sigma}_1)^{*(r-1)}*\sigma_1\strut$ & $(\ov{\sigma}_1)^{*r}$ \\
\hline\hline
$\lambda_r = \lambda_r^{\min}$\ {\rm and}\ $\lambda_n=\lambda_n^{\min}$& 
$\sigma_{r}$ & $\ov{\sigma}_{r}$ \\
\hline$\lambda_r \ne \lambda_r^{\min}$ & $0$ & $0$ \\
\hline\end{tabular}} 
\caption{Homology level products}\label{fig-level-products1}\vskip-13pt
\end{figure}
\end{center}
\vskip-70pt

\begin{center}\begin{figure}[!h]{\small
\begin{tabular}{|c|c|c|}\hline
& $(\ov{\tau}_1)^{\circledast (r-1)} \circledast \tau_1$ &
$(\ov{\tau}_1)^{\circledast r} \strut$ \\
\hline$\lambda_{rn}=\lambda_{rn}^{\max}$ & ${\tau}_{r}$ & $\ov\tau_{r}\strut$ \\
\hline$\lambda_r \ne \lambda_r^{\max}\strut$ & $0$ & $0$\\
\hline
\end{tabular}}
\caption{Cohomology level products}\label{fig-level-products2}\vskip-13pt
\end{figure}
\end{center}
\vskip-70pt
\item If $n-\lambda_1$ is even then $(\ov{\sigma}_1)^{*r}=0.$
\item If $n-\lambda_1$ is odd then $(\ov{\tau}_1)^{\circledast r}=0.$ 
\end{enumerate}
\end{thm}

\proof  We begin with the parity statements (3) and (4).
In general, if $\sigma \in H_k(\L;\mathbb Z)$ and $\tau \in H^k(\L;
\Z)$ then $2\sigma*\sigma=0$ if $n-k$ is odd and $2\tau\circledast\tau=0$ if
$n-k$ is even.  This follows from Propositions \ref{prop-homology-signs} and
\ref{prop-cohomology-commutative}, and it implies the vanishing of
$(\ov{\sigma}_1)^{*r}$ and $(\ov{\tau}_1)^{\circledast r}$
($n-\lambda_1$ even, odd, respectively).  

Statement (1) follows from the fact that the homology class $\sigma_1$ is
supported on a closed subset $A \subset \L^{\le a}$ such that
$A \cap \Sigma_1$ consists of a single point.  By Proposition
\ref{prop-product-extended} the product $\sigma_1*\sigma_1$ is supported
on the set $A *A'$ where $A'\subset \L$ is a support set for $\sigma_1$
that intersects $\Sigma_1$ in a different point.  Consequently
$A*A' \subset \L^{<2a}.$  Similarly,  the cohomology
class $\tau_1 \in H^{\lambda}(\L^{\le a+\epsilon}, \L^{<a})$ is supported
on a closed set $B \subset \L^{\ge a}$ that intersects $\Sigma_1$ in a
single point.  The zeroes in the second row of each of the tables are
also easily explained.  The (level) homology classes $(\ov\sigma_1)^{*r}$
and $\ov\sigma_r$ have the same degree if and only if $\lambda_r = 
\lambda_r^{\min}.$  The (level) cohomology classes 
$(\ov\tau_1)^{\circledast r}$ and $\ov\tau_r$ have the same degree iff
$\lambda_r = \lambda_r^{\max}.$  But all the $\lambda_i$ have the same
parity, so if $\lambda_r$ does not attain its maximum or minimum value
(i.e. if $\lambda_r \not\in \left\{ \lambda_r^{\min}, \lambda_r^{\max}
\right\}$) then
\[ \lambda_r^{\min} +2 \le \lambda_r \le \lambda_r^{\max}-2.\]
In this case it follows from (\ref{eqn-table-local-level}) that 
$(\ov\sigma_1)^{*r}=0$ and $(\ov\tau_1)^{\circledast r}=0.$  The
other calculations are similar.  The remaining statements in Theorem
\ref{prop-level-products} will be proven in the next two sections.

\subsection{Case of maximal growth} 
In this section we assume $\lambda_i = \lambda_i^{\max}$ for $i\le rn.$
Choose $x_0 = \gamma(0)$ for the base point of $M.$  
By Lemma \ref{lem-based-index} the index $\lambda_r$ equals
the index $\lambda_r^{\Omega}$ of $\gamma^r$ in the based loop space 
$\Omega = \Omega_{x_0},$ and it coincides with the index of $\gamma^r$
in the spaces $T_{\gamma^r}^{\perp}\L$ and $T_{\gamma^r}^{\perp}
\Omega$ of vector fields $V(t)$ along $\gamma^r$ such that
$V(t) \perp \gamma'(t)$ for all $t.$

Let $W_{1}$ be a maximal negative subspace of $T_{\gamma}^{\perp}\Omega$ (so
$\dim(W_{1})=\lambda_{1}=\lambda_{1}^{\Omega}$). Let $W_{1}^{\bullet r}$ be
the $r\lambda_{1}$-dimensional negative subspace of $T_{\gamma^{r}}\Omega$
consisting of concatenations $V_{1}\bullet V_{2}\bullet\cdots\bullet V_{r}$ of
vector fields $V_{i}\in W_{1}.$ Then $W_{1}^{\bullet r}$ is a maximal negative
subspace of the kernel of
\begin{align*}
\nu:T_{\gamma^{r}}^{\perp}\Omega &  \rightarrow T_{\gamma(0)}^{\perp}%
M\times\cdots\times T_{\gamma(0)}^{\perp}M\\
V  &  \mapsto\left(  V(\textstyle{\frac{1}{r}}),\cdots,V(\textstyle{\frac
{r-1}{r}})\right)
\end{align*}
Choose a maximal negative subspace $W_{r}\subset T_{\gamma^{r}}^{\perp
}\Omega$  containing $W_{1}^{\bullet r}$.  Then
\[ \dim(W_r) = \lambda_r^{\Omega} = \lambda_r = \lambda_r^{\max}
= r\lambda_{1}+(r-1)(n-1).\]
It follows that {\em the restriction of $\nu$
to $W_{r}$ is surjective} because its kernel has dimension $r\lambda_1.$

Let $\mathcal{C}_{r}$ be the $r$-leafed clover consisting of loops $\eta
\in\L $ such that $\eta(0)=\eta(i/r)$ for $i=0,1,2,\cdots,r.$ The exponential
map $T_{\gamma^{r}}(\L )\rightarrow\L $ takes $W_{r}$ to a relative cycle in
$(\L ^{\leq ra},\L ^{\leq ra-\epsilon})$ which we also denote by $W_{r},$
whose (relative) homology class is $[W_{r}]=\sigma_{r}$ (and $[W_{1}%
]=\sigma_{1}$).

The function $J:\L \times\lbrack0,1]\rightarrow\L $ extends in an obvious way
to a family of reparametrizations,
\[
J:W_{r}\times\left[  \textstyle{\frac{1}{r}}-\epsilon,\textstyle{\frac{1}{r}%
}+\epsilon\right]  \times\left[  \textstyle{\frac{2}{r}}-\epsilon
,\textstyle{\frac{2}{r}}+\epsilon\right]  \times\cdots\times\left[
\textstyle{\frac{r-1}{r}}-\epsilon,\textstyle{\frac{r-1}{r}}+\epsilon\right]
\]
which is transverse\footnote{The restriction $J|W_r$ is transverse to
$\mathcal C_r$ in the directions normal to $\gamma'(0)$ because $\nu|W_r$
is surjective.  The intervals $[\frac{i-1}{r},\frac{i}{r}]$ take care of the 
tangential directions.} 
to $C_{r}$ and such that
\[
J\left(  W_{r}\times\left[  \textstyle{\frac{1}{r}}-\epsilon,\textstyle{\frac
{1}{r}}+\epsilon\right]  \cdots\times\left[  \textstyle{\frac{r-1}{r}%
}-\epsilon,\textstyle{\frac{r-1}{r}}+\epsilon\right]  \right)  \cap
\mathcal{C}_{r}=W_{1}\bullet W_{1}\bullet\ldots\bullet W_{1}%
\]
(Pontrjagin product). 
By (a relative version of) Proposition \ref{prop-based-cohomology} we conclude that
\begin{align*}
\langle \ov\tau_1 \circledast \ov\tau_1 \circledast \ldots \circledast \ov\tau_1,
\sigma_r\rangle
&= \langle\ov\tau_1\circledast\ov\tau_1 \circledast \ldots \circledast \ov\tau_1,
[W_r]\rangle\\
&= \langle\ov\tau_1,[W_1]\rangle\cdot 
\langle\ov\tau_1,[W_1]\rangle \cdot \ldots \cdot 
\langle\ov\tau_1,[W_1]\rangle =1.
\end{align*}
It follows that $\ov\tau_1^{\circledast r} = \ov{\tau_r}.$  The calculation
for $\ov\tau_1^{\circledast (r-1)}\circledast \tau_1$ is similar.
The same technique, by explicitly displaying cycles, may be used to
prove Theorem \ref{thm-cohomology-nontrivial} below.
\qed

\subsection{Case of minimal index growth}
If $(D\Gamma_r, S\Gamma_r)$ denote the $\epsilon$-disk and sphere bundle of
the negative bundle $\Gamma_r \to \Sigma_r$ then, for sufficiently small
$\epsilon > 0$ the exponential mapping $\exp:D\Gamma_r \to \L$ is a
smooth embedding whose image
\[ (\Sigma_r^-, \partial \Sigma_r^-) = (\exp(D\Gamma_r), \exp(S\Gamma_r))\]
is a smoothly embedded submanifold with boundary in $\L$ that ``hangs
down'' from the critical set $\Sigma_r.$ Its dimension is $\lambda_r+1$ and
its fundamental class is 
\[\overline{\sigma}_r = [\Sigma_r^{-}, \partial \Sigma_r^{-}] \in
H_{\lambda_r+1}(\L^{<r\ell}\cup \Sigma_r, \L^{<r\ell}) \cong \Z\]
where $\ell$ denotes the length of $\gamma.$

Now assume that $\lambda_n = \lambda_n^{\min}.$  By Lemma \ref{lem-based-index} this
implies that the difference between $\lambda_1^{\Omega}$ 
and $\lambda_1$ is the {\em maximum} possible:  $\lambda_1 = \lambda_1^{\Omega} +n-1.$  
Let $W_1 \subset T_{\gamma}^{\perp}\L$ be a maximal negative subspace.  Then the mapping
\begin{equation}\label{eqn-neg-perp}
\nu:W_1 \to T_{\gamma(0)}^{\perp}M,\qquad \nu(V)=V(0)\end{equation} 
is surjective. Consequently
\[ \ev_0:\Sigma_1^{-} \to M\]
is a submersion in a neighborhood of the closed geodesic $\gamma.$
It follows that
$\Sigma_r^{-}$ and $\Sigma_1^{-}$ are transverse over $M$ (in some 
neighborhood of $\gamma^r$ and $\gamma$) and
\begin{diagram}[size=2em]
 \Sigma_r^{-} \times_M \Sigma_1^{-} &\rTo^{\cong}_{\phi_{\min}}  & 
\Sigma_r^{-}*\Sigma_1^{-} \end{diagram}
is a smooth submanifold of $\L$ in a neighborhood of
\begin{diagram}[size=2em]
\Sigma_r \times_M \Sigma_1 & \rTo^{\cong}_{\phi_{\min}} & \Sigma_r *\Sigma_1 
= \Sigma_{r+1} \end{diagram}
and it is contained in $\Sigma_{r+1}\cup \L^{<(r+1)\ell}.$
Now assume the index growth is minimal up to level $r,$ that is, 
$\lambda_r = r\lambda_1 -(r-1)(n-1),$ so that
\[\dim(\Sigma_r^- \times_M \Sigma_1^-) = \dim(\Sigma_{r+1}^-).\] 
Then we may apply\footnote{The condition on the eigenvalues of the second derivative 
(in the hypotheses of Theorem \ref{thm-Morse-Bott}) is satisfied by the energy 
functional, as a consequence of Theorem 2.4.2 in \cite{Klingenberg}. } 
Theorem \ref{thm-Morse-Bott} to the embeddings
\begin{diagram}
\Sigma_{r+1} \subset \Sigma_r^- \times_M \Sigma_1^- &\rTo_{\phi_{\min}}&
\L^{\le (r+1)\ell} \end{diagram}
to conclude that 
\[ [\Sigma_r^{-}*\Sigma_1^{-}, \partial\left(\Sigma_r^{-}*\Sigma_1^{-}\right)]=
[\Sigma_{r+1}^{-}, \partial \Sigma_{r+1}^{-}] \in H_{\lambda_{r+1}+1}
(\Sigma_{r+1} \cup \L^{<(r+1)\ell}, \L^{<(r+1)\ell}).
\]
(In fact we even obtain a local diffeomorphism $\tau:\Sigma_{r+1}^- \to \Sigma_{r}^- *\Sigma_1^-$
between the negative submanifolds, by equation (\ref{eqn-pi-V}).)
Using Proposition \ref{prop-product-energy} we conclude that
\[ [\Sigma_r^{-}, \partial \Sigma_r^{-}] * [\Sigma_1^{-}, \partial \Sigma_1^{-}] = 
[\Sigma_{r+1}^{-}, \partial \Sigma_{r+1}^{-}]\]
and, by induction, that
\begin{equation}\label{eqn-level-power}
\ov{\sigma}_{r+1}  =  \ov{\sigma}_r * \ov{\sigma_1}=  (\ov\sigma_1)^{*r}*\ov\sigma_1
= (\ov\sigma_1)^{*(r+1)}\end{equation}
as claimed.  The geometric calculation of the product $\ov{\sigma_r} * \sigma_1$ 
is similar.  A similar procedure will be used to prove Theorem 
\ref{thm-non-nilpotent} below.\qed

\subsection{The non-nilpotent case}
The case of isolated closed geodesics with slowest possible index growth
was studied in \cite{Nancy2}; fastest possible index growth was
studied in \cite{Nancy1} in a slightly different language because the
$*$ and $\circledast$ products were not available at the time.  
The Chas-Sullivan homology product is modeled in the
local geometry of an isolated closed geodesic with the slowest possible
growth rate.  The symmetry between the geometry in the case of slowest
possible index growth (non-nilpotent level homology) and that of fastest
possible index growth (non-nilpotent level cohomology) inspired the search
for the cohomology product.  We give statements here of two theorems on
non-nilpotent products that are restatements of the ``complementary theorem''
(p. 3100 of \cite{Nancy1}) and the theorem (p. 3099 of \cite{Nancy1}):

\begin{thm}  Let $\gamma$ be an isolated closed geodesic with non-nilpotent
level homology.  Let $L = \text{length}(\gamma).$  Then, for any $\epsilon >0,$
if $m\in \mathbb Z$ is sufficiently large there is a closed geodesic with
length in the open internval $(mL, mL+\epsilon).$  It follows that $M$ has
infinitely many closed geodesics.  \end{thm}

\begin{thm}
Let $\gamma$ be an isolated closed geodesic with non-nilpotent level
cohomology.  Let $L = \text{length}(\gamma).$  Then, for any $\epsilon >0,$
if $m \in \mathbb Z$ is sufficiently large there exists a closed geodesic
with length in the open interval $(mL-\epsilon, mL).$  It follows that  $M$
has infinitely many closed geodesics.  \end{thm}

\section{Homology product when all geodesics are closed}\label{sec-all-closed}
\subsection{}\label{subsec-def-all-closed}
In this section $M$ denotes a compact $n$ dimensional Riemannian manifold.  The
coefficient group $G$ for homology is taken to be $G = \Z$ if $M$ is orientable,
and $G = \Z/(2)$ otherwise.
Throughout this section we assume that all geodesics $\gamma$ are closed
and simply periodic with the same prime length $\ell,$ meaning that:
$\gamma(0) = \gamma(1)$, $\gamma'(0) = \gamma'(1),$ $\gamma$ is injective on $(0,1),$
and $L(\gamma) = \ell$ if $\gamma$ is prime. 

For $r \ge 1$ denote by $\Sigma_r \subset \L$ the critical
set consisting of $r$-fold iterates of prime closed geodesics.
There is a diffeomorphism $SM \cong \Sigma_r$ between the unit sphere bundle
of $M$ and $\Sigma_r,$ which assigns to each unit tangent vector $v$ the $r$-fold
iterate of the prime geodesic with initial condition $v.$  It follows that the
nullity of each geodesic is at least $\dim(\Sigma_r)-1 = 2n-2.$  Since this is the
maximum nullity possible, we see that the nullity $\nu_{r}$ of every closed
geodesic is $2n-2.$  In particular, each $\Sigma_r$ is a nondegenerate critical
submanifold (in the sense of Bott), with critical value $F(\Sigma_r) = r\ell,$ and this
accounts for all the critical points of the Morse function $F= \sqrt{E}.$
Moreover, for any $c \in \mathbb R$, the singular and \Cech homology
$H_*(\L^{\le c})$ agree, by Proposition \ref{prop-fda}.
Every geodesic $\gamma  \in \Sigma_r$ has the same index, (\cite{Besse} Thm.
7.23) say, $\lambda_r.$  By Proposition (\ref{eqn-Bott-inequality1}),
$\lambda_r \le r\lambda_1 + (r-1)(n-1).$  By (\ref{eqn-Bott-inequality2}),
$\lambda_r \ge r\lambda_1 + (r-1)(n-1),$ hence {\em the index growth is maximal},
\begin{equation}\label{eqn-lambda-r}
 \lambda_r = \lambda_r^{\max}=r\lambda_1 + (r-1)(n-1).\end{equation}
As in Theorem \ref{thm-Morse-Bott}, let $\Gamma_r \to \Sigma_r$ be the negative
definite bundle.  It is a real vector bundle whose rank is $\lambda_r.$

\begin{prop} \label{prop-orientable}  If $M$ is orientable (and all geodesics
on $M$ are closed with the same prime period) then for any $r$ the
negative bundle $\Gamma_r$ is also orientable.  \end{prop}
\proof
Fix $r$ and let $\gamma_0 \in \Sigma_r$ be a basepoint.  Set $x_0 = \gamma_0(0)
\in M.$  Using the long exact sequence for the fibration $SM \to M$ we see that
the projection $\Sigma_r \to M$ induces an isomorphism
$ \pi_1(\Sigma_r, \gamma_0) \cong \pi_1(M,x_0).$ If $\lambda_1 >0$ then by
\cite{Besse} Thm. 7.23 the manifold $M$ is simply connected, so if $\dim(M)\ge 3$
the same is true of $\Sigma_r$, hence every vector bundle on $\Sigma_r$ is
orientable.  If $\dim(M)=2$ then $M = S^2$ is the 2-sphere and $\Gamma_r$
is orientable by inspection.

So we may assume that $\lambda_1 = 0.$  By \cite{Besse} Thm 7.23 this
implies that $M$ is diffeomorphic to real projective space and $\pi_1(M,x_0) \cong
\mathbb Z/(2).$  Since $M$ is orientable, $n = \dim(M)$ is odd.

The bundle $\Gamma_r \to \Sigma_r$ is orientable iff its restriction to each loop
in $\Sigma_r$ is orientable, and it suffices to check this on any loop in the
single non-trivial class in $\pi_1(\Sigma_r, \gamma_0).$  We may even take that loop
to be the canonical lift $\widetilde{\gamma}:[0,1]\to SM,$
\[ \widetilde\gamma(t) = \left( \gamma(t), \gamma'(t)/ || \gamma'(t)|| \right)\]
of a periodic prime geodesic $\gamma:[0,1] \to M$ with $\gamma(0) = x_0.$
(Since each geodesic is determined by its initial conditions, it follows that
$\gamma_0 = \gamma^r.$ This geodesic loop is contractible in $M$ iff $r$ is even.)

Following \cite{Rademacher} \S 2.2, the $S^1$ action on $\Gamma_r\to \Sigma_r$
corresponds to  an operator $T:E \to E$ such that $T^r=I$, where
$E = \Gamma_{r, \gamma_0}$ is the fiber of $\Gamma_r$ at the basepoint $\gamma_0.$
Then $\det(T) = \pm 1$ and the bundle $\Gamma_r|\widetilde{\gamma}$ is orientable
iff $\det(T) =+1.$  If $\lambda$ is an eigenvalue of $T$ then $\lambda^r=1.$  Those
eigenvalues which are not equal to $\pm 1$ come in complex
conjugate pairs.  Hence, $\det(T) \ne 1$ iff the dimension of the $-1$-eigenspace
of $T$ is odd.  If $r$ is odd then $-1$ is never an eigenvalue of $T$ so $\det(T)
=1.$  Thus we may assume that $r$ is even.  As remarked in \cite{Rademacher}, by
\cite{Klingenberg} \S 3.2.9, \S 4.1.5 (see pp.~128,129), since $r$ is even,
the dimension of the $-1$-eigenspace of $T$ is equal to the quantity
$I_{\gamma}(-1)$ of \cite{Rademacher} (\S 1.1) and \cite{Klingenberg} (\S 4.1),
namely
\[I_{\gamma}(-1)= \lambda_2 - \lambda_1 = n-1\]
which is even.  Therefore $\Gamma_r|\widetilde{\gamma}$ is orientable, so the
bundle $\Gamma_r \to \Sigma_r$ is orientable.  \qed

It follows from Theorem \ref{thm-Morse-Bott} that
a choice of orientation for $\Gamma_r$ determines an isomorphism
\begin{equation}\label{eqn-hr}
 h_r:H_i(\Sigma_r;G) \cong H_{i+\lambda_r}(\L^{\le r\ell},\L^{<r\ell};G)
\end{equation}
where $G = \Z$ if $M$ is orientable and $G = \Z/(2)$ otherwise.

\subsection{The non-nilpotent homology class}\label{subsec-Xi}
For any $c\in \mathbb R,$ the long exact sequence for the pair
$(\L^{\le c}, \L_0)$ is canonically split by the evaluation mapping
${\bf ev}_0:\L^{\le c} \to \L_0$ so for
any Abelian group $G$ we obtain a canonical isomorphism
\begin{equation}\label{eqn-split-homology}
H_i(\L^{\le c};G) \cong H_i(\L_0;G) \oplus H_i(\L^{\le c}, \L_0;G). \end{equation}
Taking $c = \ell = F(\Sigma_1)$ and using Theorem \ref{thm-Morse-Bott} gives a
canonical isomorphism
\begin{equation}\label{eqn-bottom-homology}
H_i(\L^{\le \ell};G) \cong H_i(\L_0;G) \oplus H_{i-\lambda}(\Sigma_1;G).
\end{equation}
where $G = \mathbb Z$ if $M$ is orientable, and $G = \mathbb Z/(2)$ otherwise.
The manifold $\Sigma_1$ is orientable (whether or not $M$ is), since
$T\Sigma_1 \oplus \mathbf{1} \cong h^*(TM) \oplus h^*(TM).$  Choose an
orientation of $\Sigma_1$ with a resulting
fundamental class $[\Sigma_1] \in H_{2n-1}(\Sigma_1;G).$  Define
\[ \Theta \in H_{2n-1+\lambda_1}(\L^{\le \ell})\]
to be its image under the isomorphism (\ref{eqn-bottom-homology}).  Set
$b=\lambda_1+n-1.$

\begin{thm}\label{thm-non-nilpotent}
Let $M$ be an $n$ dimensional compact Riemannian manifold, all of whose
geodesics are simply periodic with the same prime length $\ell.$
If $M$ is orientable let $G = \mathbb Z,$ otherwise set $G = \mathbb Z/(2).$
Then the following statements hold. \begin{enumerate}
\item the energy $E:M \to \mathbb R$ is a perfect Morse-Bott function for
$H_*(\L;G),$ that is, for each $r \ge 1$ every connecting homomorphism vanishes
in the long exact sequence
\begin{diagram}[size=2em]
&\rTo& H_i(\L^{<r\ell};G) & \rTo& H_i(\L^{\le r\ell};G) &\rTo& H_i(\L^{\le r\ell},
\L^{<r\ell};G) &\rTo \end{diagram}
\item the product $*\Theta:H_i(\L,\L_0;G) \to H_{i+b}(\L,\L_0;G)$
with the class $\Theta$ is injective, and
\item for all $r \ge 1$ this product induces an isomorphism on level homology,
\[w_r: H_i(\L^{\le r\ell}, \L^{<r\ell};G) \to H_{i+b}(\L^{\le (r+1)\ell},
\L^{<(r+1)\ell};G).\]
\end{enumerate}
\end{thm}

\proof
Assume, by induction on $r$ that $\beta_r:H_i(\L^{\le r\ell},\L_0;G)
\to H_i(\L^{\le r\ell}, \L^{<r\ell};G)$ is surjective for all $i.$  The
case $r=1$ is handled by equation (\ref{eqn-bottom-homology}).
Consider the following commutative diagram, where the vertical mappings are
given by the Chas-Sullivan product $*\Theta,$
\begin{diagram}\label{diag-wr}
H_i(\L^{<r\ell},\L_0) & \rInto_{\alpha_r} & H_i(\L^{\le r\ell},\L_0) & \rOnto_{\beta_r} &
H_i(\L^{\le r\ell}, \L^{<r\ell}) \\
\dTo^{u_r} && \dTo^{v_r} && \dTo^{w_r} \\
H_{i+b}(\L^{<(r+1)\ell},\L_0) &\rTo_{\alpha_{r+1}} & H_{i+b}(\L^{\le (r+1)\ell},\L_0) &
\rTo_{\beta_{r+1}} &
H_{i+b}(\L^{\le (r+1)\ell}, \L^{<(r+1)\ell})
\end{diagram}
We will show below that the mapping $w_r$ is an isomorphism.  Assuming this for
the moment, it follows that $\beta_{r+1}$ is surjective in all degrees. Hence the
horizontal sequences in this diagram split into short exact sequences (so the
Morse function is perfect).  Therefore $u_r$ is injective if and only if $v_r$ is
injective.  However $v_r$ may be identified with the mapping $u_{r+1}$ under the
isomorphism $H_i(\L^{\le r\ell}) \cong H_i(\L^{<(r+1)\ell})$ so it is injective
by induction. (The mapping $u_1$ is trivially injective.)
The rest of \S \ref{sec-all-closed} will be devoted to
proving that $w_r$ is an isomorphism.\qed

\begin{thm}\label{thm-calculating-CS}  After composing with the isomorphism
\[h_r: H_*(SM) \to H_*(\L^{\le r\ell},\L^{<r\ell})\]
(where $SM$ denotes the unit sphere bundle of the tangent bundle to $M$),
the Chas-Sullivan product becomes the intersection product on homology,
which is to say that the following diagram commutes:
\begin{diagram}[size=2em]
H_i(\L^{\le r\ell},\L^{<r\ell}) \times H_j(\L^{\le \ell}, \L^{<\ell}) & \rTo^{*} &
H_{i+j-n}(\L^{\le (r+1)\ell}, \L^{<(r+1)\ell}) \\
\uTo^{h_r \times h_1}_{\cong} && \dTo^{h_{r+1}}_{\cong}\\
H_{i-\lambda_r}(SM)\times H_{j-\lambda_1}(SM) & \rTo &
H_{i-\lambda_r+j-\lambda_1-2n+1}(SM)
\end{diagram}
where the bottom row denotes the intersection product in homology\end{thm}

We remark that this immediately implies that $w_r$ is an isomorphism because
the mapping $w_r$ is the C-S product with the unique top dimensional class in
$H_*(\L^{\le \ell}, \L^{<\ell})$ which becomes the fundamental class
$[SM] \in H_{2n-1}(SM)$ under the vertical isomorphism in the above diagram.
But the intersection with the fundamental class is the identity mapping
$H_*(SM) \to H_*(SM).$

\proof The set $\Sigma_r*\Sigma_1 = \phi_{\frac{r}{r+1}}(\Sigma_r\times_M\Sigma_1)$
consists of pairs of composable loops; the first is an $r$-fold iterate of a prime
geodesic and the second is a single prime geodesic; all parametrized proportionally
with respect to arclength.  This set contains $\Sigma_{r+1}$ as a submanifold of
codimension $n-1.$  In fact the inclusion
\begin{equation}\label{eqn-diagonal-sphere}
 \Sigma_{r+1} \to \Sigma_r*\Sigma_1 \to \Sigma_r \times \Sigma_1\end{equation}
is the diagonal mapping  $SM \to SM \times SM.$

Let $(\Sigma_r^-, \partial \Sigma_r^-) = \left(\exp(D\Gamma_r),
\exp(\partial D\Gamma_r)\right)$
be the negative submanifold that hangs down from $\Sigma_r,$ as in Proposition
\ref{thm-Morse-Bott}, where $D\Gamma_r$ denotes a sufficiently small disk
bundle in the negative bundle $\Gamma_r \to \Sigma_r$ and where $\partial
D\Gamma_r$ denotes its bounding sphere bundle.  Then $\dim(\Sigma_r^- \times_M
\Sigma_1^-) = \dim(\Sigma_{r+1}^-)$ so we can apply Theorem \ref{thm-Morse-Bott}
to the embeddings
\begin{diagram}
\Sigma_{r+1} \subset \Sigma_r^- \times_M \Sigma_1^- &\rTo_{\phi_{\min}}&
\L^{\le (r+1)\ell} \end{diagram}
followed by an arbitrarily brief flow under the vector field $-\nabla F.$
The condition on the eigenvalues of the second derivative (in the hypotheses
of Theorem \ref{thm-Morse-Bott}) is satisfied by the energy functional, as a
consequence of Theorem 2.4.2 in \cite{Klingenberg}.  As in (\ref{eqn-pi-V}) we 
obtain a local (in a neighborhood of $\Sigma_{r+1}$) diffeomorphism
\begin{equation}\label{eqn-diffeo}
\tau:\Sigma_{r+1}^- \to \Sigma_{r}^- \times_M\Sigma_1^-\end{equation}
between the negative submanifolds, see Figure \ref{fig-finally}. By Proposition
\ref{prop-product-energy} the  Chas-Sullivan product is given by the composition
down the right side of this figure.

\begin{figure}[!h]
{\footnotesize
\begin{diagram}[size=2em]
 &  & H_i(\L^{\le r\ell},\L^{<r\ell}) \otimes
H_j(\L^{\le \ell},\L^{<\ell})\\
 && \dTo_{\cong}\\
H_{i-\lambda_r}(\Sigma_r)\otimes H_{j-\lambda_1}(\Sigma_1) & \rTo^{\cong}&
H_i(\Sigma_r^-,\Sigma_r^- -\Sigma_r)\otimes H_j(\Sigma_1^-,\Sigma_1^- -\Sigma_1)\\
\dTo^{\times} && \dTo_{\times}\\
H_{i-\lambda_r+j-\lambda_1}(\Sigma_r \times \Sigma_1) &\rTo^{\cong}&
H_{i+j}(\Sigma_r^- \times \Sigma_1^-, \Sigma_r^- \times \Sigma_1^- -
\Sigma_r \times \Sigma_1)\\
\dTo^{(\ref{eqn-Gysin-homology})} && \dTo \\
H_{i-\lambda_r+j-\lambda_1-n}(\Sigma_r \times_M \Sigma_1) &\rTo^{\cong} &
H_{i+j}(\Sigma_r^-\times\Sigma_1^-, \Sigma_r^- \times \Sigma_1^- -
\Sigma_r \times_M \Sigma_1)\\
\dTo && \dTo_{(\ref{eqn-relative-Thom-homology})}\\
H_{i-\lambda_r+j-\lambda_1-n}(\Sigma_r \times_M \Sigma_1) & \rTo^{\cong}&
H_{i+j-n}(\Sigma_r^- \times_M \Sigma_1^-, \Sigma_r^- \times_M \Sigma_1^-
- \Sigma_r\times_M\Sigma_{1})\\
\dTo^{(\ref{eqn-Gysin-homology})} && \dTo\\
H_{i+j-n-\lambda_{r+1}}(\Sigma_{r+1}) &\rTo^{\cong} &
H_{i+j-n}(\Sigma_r^-\times_M \Sigma_1^-, \Sigma_r^- \times_M \Sigma_1^- -\Sigma_{r+1})\\
\dTo && \dTo_{(\ref{eqn-diffeo})} \\
H_{i+j-n-\lambda_{r+1}}(\Sigma_{r+1}) &\rTo^{\cong} &
H_{i+j-n}(\Sigma_{r+1}^-, \Sigma_{r+1}^- - \Sigma_{r+1}) \\
&& \dTo_{\cong} \\
&& H_{i+j-n}(\L^{\le(r+1)\ell},\L^{<(r+1)\ell})
\end{diagram}}
\caption{The C-S product when all geodesics are closed}\label{fig-finally}
\end{figure}
On the other hand, the composition down the left side of the diagram is the
intersection pairing because the composition down the middle four rows is
just the Gysin pull back for the (diagonal) embedding (\ref{eqn-diagonal-sphere}).
This completes the proof of Theorem \ref{thm-calculating-CS} and hence also
of Theorem \ref{thm-non-nilpotent}. \qed

\subsection{}\label{subsec-polynomial-level}
Define the filtration $0 \subset I_0 \subset I_1 \cdots \subset
H_*(\Lambda,\Lambda_0;G)$ by
\[I_r = \text{Image}\left( H_*(\Lambda^{\le r\ell},
\Lambda_0;G) \to  H_*(\Lambda,\Lambda_0;G)\right).\]
By Proposition \ref{prop-product-extended}, $I_r * I_s \subset I_{r+s}$
so the Chas-Sullivan product induces a product on the associated graded group,
\[\text{Gr}_I H_*(\Lambda, \Lambda_0;G)=
\underset{r=1}{\overset{\infty}\oplus} I^r/I^{r-1}\cong\underset{r=1}
{\overset{\infty}\oplus} H_*(\L^{\le r\ell},\L^{<r\ell})\]
which therefore coincides with the level homology ring (\ref{eqn-level-homology}).
Let $H_*(SM;G)$ be the homology (intersection) ring of the unit
sphere bundle and let $H_*(SM)[T]_{\ge 1} = TH_*(SM)[T]$ be the ideal of 
polynomials of degree $\ge 1.$
\begin{cor}\label{cor-level-ring}  The mapping
\begin{equation}\label{eqn-polynomial-algebra}
\Phi: H_*(SM;G)[T]_{\ge 1} \to \text{Gr}_I H_*(\Lambda, \Lambda_0;G)\end{equation}
\[\Phi(aT^m) = h_1(a)*\Theta^{*(m-1)}\in H_{\deg(a)+\lambda_1+(m-1)b}(\L^{\le (m)\ell},
\L^{<(m)\ell};G)\]
is an isomorphism of rings.
\end{cor}
\proof This follows immediately from Theorems \ref{thm-non-nilpotent}
and \ref{thm-calculating-CS}.  \qed


\section{Cohomology products when all geodesics are closed}\label{sec-coH-closed}
\subsection{}
As in \S \ref{subsec-def-all-closed}, assume that $M$ is compact $n$ dimensional and all
geodesics on $M$ are simply periodic with the same prime length, $\ell.$
Let $\Sigma_r \subset \L$ denote the submanifold consisting of the $r$-fold
iterates of prime geodesics.  It is a nondegenerate critical submanifold, diffeomorphic
to the unit sphere bundle $SM,$ having index $\lambda_r = r\lambda_1 + (r-1)(n-1)$
and critical value $F(\Sigma_r) = ra.$  Let $D\Gamma_r,S\Gamma_r$ be the unit disk
bundle and unit sphere bundle of the negative bundle $\Gamma_r \to \Sigma_r.$
If $M$ is orientable, take cohomology with coefficients in the ring $G=\Z$ and
choose orientations of $M$ and $\Gamma_r$, otherwise take coefficients in
in $G = \Z/(2).$ Theorem \ref{thm-Morse-Bott} then gives an isomorphism
\[h_r:H^i(\Sigma_r) \overset{\cong}{\longrightarrow}
H^{i+\lambda_r}(\L^{\le r\ell}, \L^{<r\ell})\]
by identifying each with $H^{i+\lambda_r}(D\Gamma_r,S\Gamma_r).$
Let $b = \lambda_1 + n-1.$  Define
\[ \Omega \in H^{\lambda_1}(\L^{\le \ell}, \L_0) \cong H^0(\Sigma_1)\]
to be the image $h_1(1)$ of the element $1.$

\begin{thm}\label{thm-cohomology-nontrivial}
  Assume $M$ satisfies the above hypotheses.  Then \begin{enumerate}
\item The energy $E:\L \to \mathbb R$ is a perfect Morse function for cohomology,
meaning that for each $r\ge 1$ the connecting homomorphism vanishes in the long
exact sequence
{\small\begin{diagram}[size=2em]
&\rTo& H^i(\L^{\le (r+1)\ell},\L^{\le r\ell};G) &\rTo& H^i(\L^{\le (r+1)\ell},\L_0;G)
&\rTo& H^i(\L^{\le r\ell},\L^0;G)  & \rTo \\
&&\uTo_{\cong} &\\ && H^i(\L^{\le (r+1)\ell}, \L^{<(r+1)\ell};G)
\end{diagram}}
\item  The product $\circledast \Omega:H^i(\L,\L_0) \to H^{i+b}(\L,\L_0)$ is injective and
\item this product induces an isomorphism
\[ w_r:H^i(\L^{\le r\ell}, \L^{<r\ell}) \to H^{i+b}(\L^{\le (r+1)\ell}, \L^{<(r+1)\ell})\]
for all $r \ge 1$ and all $i \ge 0.$
\end{enumerate}\end{thm}

As in the proof of Theorem \ref{thm-non-nilpotent}, part (3) implies parts
(1) and (2).  Part (3) follows from the stronger statement,
\begin{thm}\label{thm-stronger-cohomology}
  After composing with the isomorphism
\[ h_r:H^*(SM) \to H^*(\L^{\le r\ell}, \L^{<r\ell})\]
the cohomology product becomes the cup product on cohomology, which is to say
that the following diagram commutes 
(recall that $\lambda_{r+1}=\lambda_1 +\lambda_r+n-1$),
\begin{diagram}[size=2em]\label{diag-cup-product}
H^a(\Sigma_r) \otimes H^b(\Sigma_1) & \rTo_{h_r \otimes h_1} &
H^{a+\lambda_r}(\L^{\le r\ell},\L^{<r\ell}) \otimes
H^{b+\lambda_1}(\L^{\le\ell}, \L^{<\ell})\\
\dTo^{\smile} && \dTo_{\circledast} \\
H^{a+b}(\Sigma_{r+1}) &\rTo_{h_{r+1}}& H^{a+b+\lambda_{r+1}}
(\L^{\le (r+1)\ell}, \L^{<(r+1)\ell})
\end{diagram}
\end{thm}
The proof appears in the next few sections.  In order to use Proposition
\ref{prop-coproduct-levels} we will need to work in the space $\A$ of
PPAL loops.

\quash{
\begin{equation}\label{eqn-inclusions}
 \Sigma_{r+1} \subset \Sigma_r*\Sigma_1 \cong \Sigma_r \times_M \Sigma_1 \subset
\Sigma_r \times \Sigma_1 \to \Sigma_r\end{equation}
and the composition down the right side is the cohomology product with the class $\Ty.$
All the horizontal maps are isomorphisms.
{\footnotesize\begin{diagram}[size=2em]
H^i(\Sigma_r) & \lTo^{h_r} & H^{i+ \lambda_r}(\A^{\le r\ell}, \A^{<r\ell})
\cong H^{i+\lambda_r}(\A^{<r\ell+2\epsilon},\A^{<r\ell-\epsilon}) \\
\dTo^{\otimes 1} && \dTo_{\otimes \Omega}\\
H^i(\Sigma_r) \otimes H^0(\Sigma_1) & \lTo &
H^{i+\lambda_r}(\A^{< r\ell+2\epsilon},\A^{<r\ell-\epsilon}) \otimes
H^{\lambda_1}(\A^{< \ell+2\epsilon}, \A^{<\ell-\epsilon})\\
\dTo^{(\ref{eqn-inclusions})} &\bo{main}& \dTo_{\circledast} \\
H^i(\Sigma_{r+1}) &\lTo& H^{i+\lambda_{r+1}}
(\A^{<(r+1)\ell+\epsilon}, \A^{(r+1)\ell-2\epsilon})\\
&\luTo_{h_{r+1}}& \dTo^{\cong}\\
&& H^{i+\lambda_{r+1}}(\A^{\le (r+1)\ell}, \A^{<(r+1)\ell})
\end{diagram}}
}

\subsection{}  Fix $r \ge 1.$
Let $r\ell^+ = r\ell+2\epsilon$, $r\ell^- = r\ell-\epsilon,$ $\ell^+ = \ell+2\epsilon$, and
$\ell^- = \ell-\epsilon.$  Set $j=a+b+\lambda_r+\lambda_1$ so that
$j+n-1=a+b+\lambda_{r+1}.$ It is convenient to replace the gluing map
$\phi_{\half}:\A\times_M\A \to \A_{\half}$ with the topologically equivalent
embedding $\phi_{\r}:\A \times_M \A \to \A_{\r},$ which approximates $\phi_{\min}$
near $\Sigma_r \times \Sigma_1,$ in fact,
\[ \Sigma_r * \Sigma_1 = \phi_{\min}(\Sigma_r \times_M \Sigma_1) =
\phi_{\r}(\Sigma_r \times_M \Sigma_1).\]
We will write
\[ \F_{\r}^{<a,<b} \ \text{ for }\ \phi_{\r}(\F_{\A}^{<a,<b}).\]
Similarly we replace mapping $J_{\A}:\A \times[0,1] \to \A_{\half}$
with the mapping
\[J_r:\A \times[0,1] \to \A_{\r}\]
given by $J_r(\alpha,s) =
\alpha \circ \theta_{\r \to s}$ where $A_{\r}$ and $\theta_{\r \to s}$ are
defined by replacing $\half$ with $\r$ in \S \ref{subsec-first-cohomology-product}.
The mapping $\A^{\le (r+1)\ell} \to \A_{\r}^{\le r\ell,\le \ell}$ given by $\alpha \mapsto
J_r(\alpha, \r)$ is a homotopy equivalence; its inverse assigns to a pair of
joinable PPAL paths $\alpha,\beta \in \A_{\r}$ (with $\alpha (1)=\beta(0)$ and
$\beta(1)=\alpha(0)$) with lengths $\le r\ell$ and $\le \ell$ respectively, the
path $\phi_{\min}(\alpha,\beta)$ obtained by joining them at time $L(\alpha)/
(L(\alpha)+L(\beta)).$

Let $I'$ be the closed interval $I'=\left[\frac{r\ell-\epsilon}{(r+1)\ell+\epsilon},
\frac{r\ell+2\epsilon}{(r+1)\ell+\epsilon}\right]$ as in \S
\ref{subsec-construction-coprod}.  Then
\[ J_r(\A^{<(r+1)\ell+\epsilon} \times I') \subset \A_{\r}^{<r\ell^+, <\ell^+}\]
and $J_r$ takes both $\A^{<(r+1)\ell+\epsilon} \times \partial I'$ and
$\A^{<(r+1)\ell-2\epsilon} \times I'$ into the subset
\[ \A_{\r}^{<r\ell^+, <\ell^+} - \phi_{\r}(\F_{\A}^{[r\ell^-,r\ell^+),[\ell^-,\ell^+)}).\]

\subsection{}\label{subsec-proof-cohomology}
  Recall from \S \ref{subsec-def-all-closed} that the negative
bundle $\Gamma_r$ over $\Sigma_r$ is orientable if $M$ is orientable, and that
the exponential defines a diffeomorphism $e_r$ of a sufficiently small disk bundle
and its bounding sphere bundle, $(D\Gamma_r, \partial D\Gamma_r)$ onto a submanifold
with boundary, $(\Sigma_r^-,\partial \Sigma_r)$ in $\L^{\le r\ell}$,
such that $e_r(D\Gamma_r - \Sigma_r) \subset \L^{< r\ell}$ (where $\Sigma_r$
is the zero section). Using the homotopy equivalence $Q:\L^{\le r\ell}
 \to \A^{\le r\ell}$ of Proposition \ref{prop-PPAL-space}, we may assume that
$\Sigma_r^- \subset\A^{\le r\ell}$ so we obtain isomorphisms which we also denote by
\[h_r: H^i(\Sigma_r) \cong H^{i+\lambda_r}(\Sigma_r^-, \partial \Sigma_r^-) \cong
H^{i+\lambda_r}(\A^{\le r\ell}, \A^{<r\ell}).\]
Moreover, equation (\ref{eqn-diffeo}) gives a diffeomorphism
$\tau: \Sigma_{r+1}^- \to \Sigma_r^- \times_M \Sigma_1^-\cong \Sigma_r^-*\Sigma_1^-.$
The following diagram may help in sorting out these different spaces.
\begin{diagram}[size=2em]
\Sigma_{r+1} & \subset & \Sigma_{r+1}^- & \subset & \A^{\le (r+1)\ell} \\
\dTo && \dTo && \dTo^{\eta}  \\
\Sigma_r*\Sigma_1 &\ \subset\ & \Sigma_r^- *\Sigma_1^- &\subset& \A^{\le r\ell,\le \ell}_{\r} &\\
\dTo && \dTo &&\dTo&\\
\F_{\r}^{[r\ell^-,r\ell^+),[\ell^-,\ell^+)} &\quad \subset\quad &
\F_{\r}^{<r\ell^+,<\ell^+}&
\quad \subset\quad & \A_{\r}^{<r\ell^+,<\ell^+}\end{diagram}
In order to compact the notation, for the rest of the section we will write
\[ H^*(Y,\bs A) \ \text{ for }\ H^*(Y, Y-A).\]

\subsection{}\label{subsec-huge-diagram}
We are now in a position to expand the diagram in Theorem \ref{thm-stronger-cohomology}.
This is accomplished in Figure \ref{fig-expanded}.
Here, $j=a+b+\lambda_r +\lambda_1$ so that $j+n-1 = a+b+\lambda_{r+1}.$
Each of the rectangles in this diagram is obviously commutative except possibly
for the portion denoted $\boxed{1},$ which we now explain, as it involves
the somewhat mysterious degree shift of 1, and its relationship to the mapping $J_r.$
\begin{figure}
{\tiny\begin{diagram}[size=2.2em,PostScript]
H^a(\Sigma_r)\otimes H^b(\Sigma_1) & \rTo & H^{a+\lambda_r}(\Sigma_r^-, \bs \Sigma_r)
\otimes H^{b+\lambda_1}(\Sigma_1^-, \bs \Sigma_1) & \lTo &
H^{a+\lambda_r}(\A^{<r\ell^+},\A^{<r\ell^-}) \otimes H^{b+\lambda_1}(\A^{<\ell+},
\A^{<\ell^-})\\ 
\dTo_{\times} && \dTo_{\times} && \dTo_{\times}\\
H^{a+b}(\Sigma_r \times \Sigma_1) & \rTo_{\cong} & H^{j}(\Sigma_r^- \times \Sigma_1^-, \bs
\Sigma_r \times \Sigma_1)
& \lTo & H^{j}\left((\A^{<r\ell^+},\A^{<r\ell^-})\times
(\A^{<\ell^+}\times \A^{<\ell^-})\right) &\\ 
\dTo && \dTo && \dTo\\
H^{a+b}(\Sigma_r \times_M \Sigma_1) & \rTo_{\cong} & H^{j}(\Sigma_r^- \times_M \Sigma_1^-,
\bs \Sigma_r \times_M \Sigma_1) & \lTo & H^{j}(\F_{\A}^{<r\ell^+,<\ell^+}, \bs
\F_{\A}^{[r\ell^-,r\ell^+),[\ell^-,\ell^+)})\\  
\dTo_{\phi_{r/(r+1)}^*}^{\cong}  && \dTo_{\phi_{r/(r+1)}^*}^{\cong} &&
\dTo_{\phi_{r/(r+1)}^*}^{\cong}\\
H^{a+b}(\Sigma_r*\Sigma_1)& \rTo_{\cong} &
H^{j}(\Sigma_r^-*\Sigma_1^-, \bs \Sigma_r*\Sigma_1) &\lTo&
H^{j}(\F_{\r}^{<r\ell^+,<\ell^+}, \bs \F_{\r}^{[r\ell^-,r\ell^+),[\ell^-,\ell^+)})\\ 
 &&  && \dTo_{\smile \mu_{TM}}^{(\ref{eqn-A-isomorphism2})}  \\
\dTo && \dTo^{\S\ref{cor-relative-coGysin}}_{\smile\mu_{\Upsilon}} &\boxed{1}&
H^{j+n}(\A_{\r}^{<r\ell^+,<\ell^+}, \bs \F_{\r}^{[r\ell^-,r\ell^+),[\ell^-,\ell^+)})\\ 
&& && \dTo_{J_r^*}^{(\ref{eqn-JA-pairs})} \\
H^{a+b}(\Sigma_{r+1}) & \rTo_{\cong} & H^{j+n-1}(\Sigma_r^-*\Sigma_1^-, \bs \Sigma_{r+1})
 && H^{j+n}\left((\A^{<(r+1)\ell^+},
\A^{<(r+1)\ell^-}) \times (I', \partial I') \right)\\ 
\dTo^{\cong}&&\dTo_{\cong}^{(\ref{eqn-diffeo})} &&\uTo^{\cong}_{\smile\mu_{\mathbf O}}  \\
H^{a+b}(\Sigma_{r+1}) &\rTo_{\cong}& H^{j+n-1}(\Sigma_{r+1}^-, \bs \Sigma_{r+1})&
\rTo_{\cong}& H^{j+n-1}(\A^{\le(r+1)\ell}, \bs \Sigma_{r+1})
\end{diagram}}\caption{$\circledast$ product}\label{fig-expanded}
\end{figure}

The figure eight space $\F$ has a normal bundle in $\L$ that is isomorphic
to the (pullback of the) normal bundle of the diagonal in $M \times M,$ and hence
to the tangent bundle $TM$ of $M.$  Its Thom class is denoted $\mu_{TM}$ and the
Thom isomorphism (\ref{eqn-A-isomorphism2}) is given by the cup product with this
Thom class.  The normal bundle of $\Sigma_{r+1}$ in $\Sigma_r \times_M \Sigma_1$ is
denoted $\Upsilon.$  The Gysin mapping (labeled \S\ref{cor-relative-coGysin} in the
diagram) is given by the cup product with the Thom class $\mu_{\Upsilon}.$  The
K\"unneth isomorphism at the lower right corner of the diagram is given by the
cup product with the generator of $H^1(I',\partial I')$ which may be identified
with the Thom class $\mu_{\mathbf O}$ of the trivial one dimensional
bundle $\mathbf O$ on the interval $I'.$

So to prove that $\boxed{1}$ commutes we need to compare the Thom class
$\mu_{TM}$ with the product of Thom classes $\mu_{\Upsilon} \cup \mu_{\mathbf O}.$
It suffices to construct a vector bundle isomorphism
$J_r^*(TM) \cong \Upsilon \oplus {\mathbf O}.$

 The critical set $\Sigma_{r+1}$ is a submanifold of codimension
$n-1$ in $\Sigma_r * \Sigma_1.$  A point in the latter space is an $r$-fold
iterate of a prime closed geodesic followed by a prime closed geodesic with
the same base point, all parametrized proportionally with respect to arclength,
so it is determined by a triple $(p,u,v)$ where $p \in M$ and $u,v\in S_p$ are unit
tangent vectors at $p.$  This point lies in $\Sigma_{r+1}$ if and only if $u=v.$
It follows that the normal bundle $\Upsilon$ of $\Sigma_{r+1}$ in
$\Sigma_r * \Sigma_1$ may be naturally identified with the bundle
$\ker(d\pi)$ of tangents to the fibers of the projection $\pi:SM \to M.$
But there is another way to view this bundle.

Let $\nu$ be the tautological (trivial) bundle over $SM$ whose fiber at the
point $v\in S_p$ is the 1-dimensional span $\langle v \rangle \subset T_pM.$
Then $\pi^*(TM) \cong \nu \oplus \nu^{\perp}$ where
$\nu^{\perp}$ is the bundle whose fiber over $v \in S_p$ is $v^{\perp}.$
For any $v\in S_p$ the inclusion of the unit tangent sphere
$S_p \subset T_pM$ induces an injection $\ker(d\pi)
\hookrightarrow T_{p}M$ whose image is $\nu^{\perp}.$  In this way we obtain a
canonical isomorphism $\Upsilon \cong \nu^{\perp}$ and therefore an isomorphism
$\pi^*(TM) \cong \Upsilon \oplus\nu.$

Consider the restriction $J_r:\Sigma_{r+1} \times I' \to \L,$ say $\beta = J_r(\alpha,s).$
Then $\frac{\partial J_r}{\partial s}(\alpha,s)$ is a vector field along $\beta$
that is a multiple of the tangent vector $\beta'$ since the $s$ factor only changes the
parametrization.  This gives an isomorphism between $J_r^*(\nu)$ and the trivial
1-dimensional tangent bundle $TI'$ on $\Sigma_{r+1} \times I'.$ In summary we have
constructed an isomorphism $J_r^*(TM) \cong \Upsilon \oplus {\mathbf O}.$   This completes the
proof
that the diagram in
Theorem \ref{thm-stronger-cohomology}
commutes, so it completes the proof of Theorem \ref{thm-cohomology-nontrivial}.  \qed

\subsection{Level cohomology ring}  Continue with the assumption that all
geodesics on $M$ are simply periodic with the same prime length $\ell.$
Define the filtration  $H^*(\L,\L_0;G)= I^0 \supset I^1 \supset \cdots$ by
\[ I^r =\text{Image}\left(H^*(\L,\L^{\le r\ell};G)\to H^*(\L,\L_0;G)\right).\]
Each $I^r \subset H^*(\L,\L_0)$ is an ideal (with respect to the $\circledast$
product) and $I^r \circledast I^s \subset I^{r+s}.$
Since the Morse function is perfect it induces an isomorphism
\[ Gr^IH^*(\L,\L_0) \cong \underset{r \ge 1}{\oplus}
H^*(\L^{\le r\ell}, \L^{<r\ell})\]
between the associated graded ring and the level cohomology ring.  Let $H^*(SM;G)$
denote the cohomology ring of the unit sphere bundle of $M.$

\begin{cor}\label{cor-cohomology-graded}
The mapping (cf.~ \S \ref{subsec-polynomial-level})
\[
 \Psi: H^*(SM;G)[T]_{\ge 1} \to \text{Gr}^IH^*(\L,\L_0;G) =
\underset{r \ge 1}{\oplus} H_*(\L^{\le r\ell},\L^{<r\ell})\]
given by
\[ \Psi(aT^m) = h_1(a)\circledast\Omega^{\circledast (m-1)} \in H^{\deg(a)+\lambda_1+(m-1)b}
(\L^{\le (m)\ell},\L^{<(m)\ell};G)\]
is an isomorphism of rings.
\end{cor}
\proof  Just as in Corollary \ref{cor-level-ring}, the $\circledast$ product on $\Sigma_1$
may be identified with the cup product because the diagram in \S
\ref{subsec-huge-diagram} commutes.\qed

\subsection{Based loop space} \label{sec-based-loop-coproduct-nontrivial}
Let $\Omega=\Omega_{x_0}\subset \L$ denote the space of
loops in $M$ that are based at $x_0.$  Suppose as above that all geodesics on
$M$ are simply periodic with the same prime length $\ell.$  Since the index
growth is maximal, cf.~(\ref{eqn-lambda-r}), the index of each critical point
in $\Omega$ is the same as that in $\L$, cf.~Lemma \ref{lem-based-index}.
The critical set $\Sigma_r^{\Omega} \subset\Omega$ at level $r\ell$ is
parametrized by the unit sphere $S^{n-1}\subset T_{x_0}M.$
The arguments of the preceding section may be applied to the based loop space
with its product $\circledast$ (cf. Proposition \ref{prop-based-cohomology})
and we conclude that {\em the cohomology algebra $(H^*(\Omega_{x_0},x_0), \circledast)$ 
is filtered by the energy and the associated graded algebra is isomorphic to the
polynomial algebra $H^*(\Sigma_1^{\Omega})[T]$ where $\deg(T) = b = \lambda_1+n-1$ 
and where $H^*(\Sigma_1^{\Omega})$ is the cohomology algebra of the sphere
$S^{n-1}.$  The restriction mapping $H^*(\L) \to H^*(\Omega)$ induces the
mapping on the associated graded algebras
\begin{equation}\label{eqn-graded-pointed}
 H^*(SM)[T] \to H^*(S^{n-1})[T]\end{equation}
which is determined by the restriction homomorphism }
\begin{equation}\label{eqn-restriction-cohomology}
H^*(\Sigma_1) \to H^*(\Sigma_1^{\Omega}).\end{equation}
If $M=S^n$ is the n-sphere,  the restriction mapping 
(\ref{eqn-restriction-cohomology}) is surjective when 
$n$ is odd.

\section{Related products}\label{sec-bracket}
\subsection{} \label{subsec-CS-bracket}
Composing the K\"unneth isomorphism with the action 
$S^1 \times \L \to \L$ gives a map
\[ \Delta_*:H_i(\L;G) \to H_{i+1}(\L;G)\ \text{ and }\
\Delta^*:H^i(\L;G) \to H^{i-1}(\L;G)\]
for any coefficient group $G.$  Then \cite{CS} define $H_i(\L) \times
H_j(\L) \overset{\left\{\cdot,\cdot\right\}}{\longrightarrow} H_{i+j-n+1}(\L)$ such that
\begin{equation}\label{eqn-CS-bracket}
 \left\{ \sigma, \delta \right\} =
(-1)^{|\sigma|}\Delta_*(\sigma * \delta) - (-1)^{|\sigma|}\Delta_*(\sigma)*\delta
- \sigma * \Delta_*(\delta)\end{equation}
where $|\sigma| = i-n$ if $\sigma \in H_i(\L).$  They prove that the bracket is
(graded) anti-commutative, it satisfies the (graded) Jacobi identity, and it is
a derivation in each variable, that is,
\begin{enumerate}
\item $\left\{ \sigma,\tau \right\} = -(-1)^{(|\sigma|+1)(|\tau|+1)}
\left\{ \tau, \sigma \right\}$
\item $ \left\{\sigma, \left\{ \tau, \omega\right\} \right\} =
\left\{\left\{ \sigma, \tau\right\}, \omega\right\}
+ (-1)^{(|\sigma|+1)
(|\tau|+1)}\left\{\tau, \left\{ \sigma, \omega \right\} \right\}$
\item $\left\{\sigma, \tau *\omega\right\} = \left\{ \sigma,\tau\right\}
* \omega + (-1)^{|\tau|(|\sigma|+1)}\tau *
\left\{\sigma,\omega\right\}.$
\end{enumerate}
Since $\Delta_*$ preserves the energy, it follows that {\em the bracket operation is also
defined on the relative homology groups $\H_*(\L^{\le a}, \L^{\le a'};G)$ and it satisfies
the energy estimates of Proposition \ref{prop-product-extended}.}

Similarly, we may define  
$H^i(\L,\L_0) \times H^j(\L,\L_0) 
\overset{\left\{\cdot,\cdot\right\}}{\longrightarrow}H^{i+j+n-2}(\L,\L_0)$ by
\begin{equation}\label{eqn-cohomology-bracket}
 \left\{ \tau,\omega\right\} = (-1)^{|\tau|} \Delta^*(\tau \circledast \omega)
- (-1)^{|\tau|} \Delta^*(\tau)\circledast \omega - \tau\circledast\Delta^*(\omega)\end{equation}
where $|\tau| = i+n-1$ if $\tau \in H^i(\L).$
\begin{thm}\label{thm-X}
The cohomology bracket satisfies the following for any $\sigma,\tau,\omega \in H^*(\L,\L_0)$.  
\begin{enumerate}\renewcommand{\labelenumi}{(\Alph{enumi})}
\item $\left\{\tau,\omega\right\} = - (-1)^{(|\tau|+1)(|\omega|+1)}
\left\{\omega,\tau \right\}$
\item $\left\{\sigma, \left\{\tau,\omega\right\}\right\} =
\left\{ \left\{\sigma,\tau\right\},\omega\right\} + 
(-1)^{(|\tau|+1)(|\omega|+1)}\left\{\tau,\left\{\sigma,\omega\right\}\right\}$
\item $\left\{\sigma, \tau\circledast\omega\right\} = \left\{ \sigma,\tau\right\}
\circledast \omega + (-1)^{|\tau|(|\sigma|+1)}\tau \circledast
\left\{\sigma,\omega\right\}.$
\end{enumerate}\end{thm}
\proof  Part (A) follows directly from the definition.  The proof of parts (B)
and (C) will appear in Appendix \ref{sec-proof-of-theoremX}.

\subsection{Nondegenerate case}  \label{subsec-bracket-nondegenerate}
As in Theorem \ref{prop-level-products},
assume the manifold $M$ is orientable, $\gamma$ is a closed geodesic such that all
of its iterates are nondegenerate, and assume the negative bundle $\Gamma_r$ 
is orientable for all $r,$ cf. footnote \ref{foot1}. Let $a = L(\gamma).$ Assume $r \ge 2.$  
Let $\sigma_r, \ov\sigma_r, \tau_r, \ov \tau_r$ be the local (level) homology and 
cohomology classes described in equation (\ref{eqn-local-level-classes}). In the local level 
(co)homology group $H(\L^{<ra}\cup \Sigma_r, \L^{<ra})$ we have:
\begin{align*}
\Delta_*(\sigma_r) = r\ov\sigma_r;\qquad & \Delta_*(\ov\sigma_r)=0\\
\Delta^*(\tau_r) = r \ov\tau_r;\qquad & \Delta^*(\ov\tau_r)=0.
\end{align*}
Using Theorem \ref{prop-level-products}, 
if $\lambda_{j+k} = \lambda_{j+k}^{\min}$ then in 
$H_*(\L^{<(j+k)a}\cup\Sigma_{j+k}, \L^{<(j+k)a})$ we have
\begin{align*}
\left\{ \sigma_j,\sigma_k\right\} &= -(k+(-1)^{|\sigma_1|}j)\sigma_{j+k} \\
\left\{ \sigma_j,\ov\sigma_k\right\} &= (-1)^{|\sigma_1|}k \ov \sigma_{j+k}
\intertext{while if $\lambda_{j+k} = \lambda_{j+k}^{\max}$ then in 
$H^*(\L^{<(j+k)a}\cup\Sigma_{j+k}, \L^{<(j+k)a})$ we have}
\left\{ \tau_j, \tau_k\right\} &= (-k+(-1)^{|\tau_1|}j)\tau_{j+k}\\
\left\{\tau_j, \ov\tau_k\right\} &= (-1)^{|\tau_1|}k\ov\tau_{j+k}.
\end{align*}

\subsection{Equivariant homology and cohomology}
As in \cite{CS}, one may consider the $T=S^1$-equivariant homology
$H_*^{T}(\L)$ of the free loop space $\L.$  Let $ET \to BT$ be the classifying
space and universal bundle for $T=S^1$; it is the limit of finite dimensional
approximations $S^{2n+1} \to \mathbb CP^n$ and let $\pi:\L \times ET \to \L_T =
\L \times_T ET$ be the Borel construction.  There are Gysin (exact) sequences
(\cite{Spanier} \S 5.7 p. 260) with coefficients in $\mathbb Z,$
\begin{diagram}[size=2em]
\phantom{.}&\rTo&H_{i+1}(\L)&\rTo^{\pi_*}&H_{i+1}^T(\L)&\rTo&H_{i-1}^T(\L)&\rTo^{\pi^*}&H_{i}(\L)
&
\rTo& \\
&\rTo&H^i_T(\L) & \rTo_{\pi^*}& H^i(\L) &\rTo_{\pi_*}&H^{i-1}(\L)&\rTo& H^{i+1}(\L) & \rTo&.
\end{diagram}

The Chas-Sullivan ``string bracket'' (homology) product on equivariant homology 
is defined to be $(-1)^{i-n}$ times the composition
\begin{diagram}[size=2em]
H_i^T(\L) \times H_j^T(\L) &\rTo^{\pi^*\times\pi^*}& H_{i+1}(\L) \times H_{j+1}(\L)
&\rTo^{*}& H_{i+j+2-n}(\L) &\rTo^{\pi_*}& H_{i+j+2-n}^T(\L),\end{diagram}
that is, $[\sigma,\delta] = (-1)^{|\sigma|}\pi_*\left(\pi^*(\sigma)*\pi^*(\delta)\right).$
The action of $T=S^1$ preserves the energy function, so {\em the (homology) string bracket 
extends to products on relative homology
$\H_*^T(\L^{\le a}, \L^{\le a'})$ and on level homology $\H_*^T(\L^{\le a},
\L^{<a})$ which satisfy the same energy estimates as those in Proposition
\ref{prop-product-extended}.}
\quash{  
For any $0 \le a' <a \le \infty$ and $0 \le b' < b \le \infty$ the Chas-Sullivan
product restricts to products
\begin{diagram}[size=2em]
\H_i^T(\L^{\le a},\L^{\le a'};G) \times \H_j^T*(\L^{\le b}, \L^{\le b'};G)
&\rTo^{[\cdot,\cdot]}& \H_{i+j+2-n}^T(\L^{\le a+b}, \L^{\le \max(a+b',a'+b)};G)\\
\H_i^T(\L^{\le a}, \L^{<a};G) \times \H_j^T(\L^{\le b}, \L^{<b};G)
&\rTo^{[\cdot,\cdot]}& \H_{i+j+2-n}^T(\L^{\le a+b}, \L^{<a+b}).\end{diagram}
\end{prop}  } 

Similarly the (cohomology) product $\circledast$ gives rise to a product in equivariant cohomology
as $(-1)^{i+n-1}$ times the composition
\begin{diagram}[size=2em]
H^i_T(\L)\times H^j_T(\L) &\rTo^{\pi^*\times\pi^*}& H^i(\L) \times H^j(\L) 
&\rTo^{\circledast}& H^{i+j+n-1}(\L) &\rTo^{\pi_*}& H^{i+j+n-2}_T(\L),\end{diagram}
or $\tau \circledcirc \omega = (-1)^{|\tau|}\pi_*\left(\pi^*(\tau)
\circledast \pi^*(\omega)\right).$  It also gives products in relative
equivariant cohomology $\H^*_T(\L^{\le a}, \L^{\le b})$ 
with energy estimates as in Proposition \ref{prop-coproduct-levels}. 

\subsection{}
The string bracket is discussed in \cite{CS}, p. 24 in the case when $M$ is a
surface of genus $>1.$  When $n=2$ it gives a non-trivial map
\begin{diagram}[size=2em]
H_0^T(\L) \times H_0^T(\L) &\rTo^{[\cdot,\cdot])}& H_0^T(\L)\end{diagram}
which turns out to be a product discovered by Goldman \cite{Goldman}
and Wolpert \cite{Wolpert}.  In this case the
equivariant cohomology product $\circledcirc$ is also non-trivial in degree zero,
\begin{diagram}[size=2em]
 H^0_T(\L,\L_0) \times H^0_T(\L,\L_0) &\rTo^{\circledcirc}& H^0_T(\L,\L_0).
\end{diagram}
The group $H^0_T(\L,\L_0)$ can be identified with the set of maps from the set
of free homotopy classes of loops in $M$ to the coefficient group $G,$ which
take the homotopy class of trivial loops to the identity element in $G.$
 We give
a simple example:  consider the case when $\alpha,\beta\in H^0_T(\L,\L_0)$
are given by $\cap A$ and $\cap B$ where $A$ and $B$ are disjoint simple
closed loops in $M.$  Assume that $M$ also supports loops $A'$ and $B'$ with
\begin{align*}
  A \cap A' &= B \cap B' = 1\\
  A\cap B' &= B \cap A' = A \cap B = 0.
\end{align*}
Let $C \in \L.$  By
chasing through the definitions we see that the class
\[ \alpha \circledcirc \beta \in H^0_T(\L,\L_0)\]
has the following properties, when considered as a map from the set of free
homotopy classes of loops in $M$ to $G:$ \begin{enumerate}
\item $\langle\alpha \circledcirc \beta, [C]\rangle = 0$ if $C$ is an embedded loop,
\item $\langle\alpha \circledcirc \beta, [C]\rangle=1$ if there exist loops 
$A^{\prime\prime}, B^{\prime\prime}$, homotopic to $A^{\prime},B^{\prime}$
(respectively), which intersect transversally at the point 
$A^{\prime\prime}(0) = B^{\prime\prime}(0)$ such that $C$ is homotopic to
the composed loop $A^{\prime\prime}\cdot B^{\prime\prime}.$
\end{enumerate}
This product appears to be related to the Turaev cobracket described in
\cite{Chas}, p. 27.

\appendix

\section{\Cech homology and cohomology}\label{sec-homology}
\subsection{}
Throughout this paper, the symbols $H_i$ and $H^j$ denote the singular homology
and cohomology while $\H_i$ and $\H^j$ denote the \Cech homology and
cohomology as described, for example in \cite{Eilenberg} \S 9, \cite{Dold} p. 339, 
\cite{Bredon} p. 315 (\Cech homology), and \cite{Spanier} \S 6.7 Ex. 14 p. 327 
(\Cech cohomology).

The problem is that the space $\Lambda^{\le a}$ and even its finite dimensional
approximation $\M_N^{\le a}$ might be pathological if $a$ is a
critical value of the function $F.$  However, for each regular value $a + \epsilon$
the space $\Lambda^{\le a + \epsilon}$ has the homotopy type of a finite simplicial
complex.  Thus one might hope to describe the homology and cohomology of
$\Lambda^{\le a}$ using a limiting process.  The \Cech homology and cohomology are
better behaved under limiting processes than the singular homology and cohomology.
Unfortunately, the \Cech homology does not always satisfy the exactness axiom for a
homology theory (although the \Cech cohomology does satisfy the exactness axiom).
These issues are explained in detail in \cite{Eilenberg}.  We now review the
relevant properties of these homology theories that are used in this article.

\subsection{}  \label{subsec-Cech-limits}
Let $G$ be an Abelian group.  If $A \subset X$ are topological
spaces then the composition of any two homomorphisms in the homology sequence
for the pair $\H_*(X,A;G)$ is always zero.  If $X$ and $A$ are compact and if
$G$ is finite or if $G$ is a field then the homology sequence for $\H_*(X,A;G)$
is exact.

If a topological space $X$ has the homotopy type of a finite simplicial complex then
the natural transformations $H_j(X;G) \to \H_j(X;G)$ and $\H^j(X;G) \to H^j(X;G)$
are isomorphisms for all $j.$

By \cite{Hatcher} Thm 3.33, if a topological space $X$ is an increasing union of
subspaces $X_1 \subset X_2 \subset \cdots$ and if every compact subset
$K\subset X$ is contained in some $X_n$ then for all $j$ the inclusions $X_n \to X$
induces isomorphisms
\[H_j(X;G)  \cong \underset{\longrightarrow}{\lim}H_j(X_n;G)\
\text{ and } \H_j(X;G)  \cong \underset{\longrightarrow}{\lim}\H_j(X_n;G).\]

\subsection{}\label{subsec-limits}
Let $A$ be a closed subset of a paracompact Hausdorff space $X$.
Let $U_1 \supset U_2 \supset \cdots$
be a sequence of subsets of $X$ such that $\cap_{n=1}^{\infty} U_n = A.$
Then the following table describes sufficient conditions that
\[ \H^q(A;G) \cong \underset{\longrightarrow}{\lim} \H^q(U_n;G)\
\text{ and }
\H_q(A;G) \cong \underset{\leftarrow}{\lim} \H_q(U_n;G).\]

\begin{center}
\begin{tabular}{|c||c|c|}\hline
& $U_n$ open & $U_n$ closed \\
\hline\hline
cohomology & no restriction & $X$ is compact \\
homology & X is a manifold & $X$ is compact\\
\hline
\end{tabular}\end{center}

These facts are classical and the proofs may be found in the textbooks, for
example \cite{Eilenberg} \S IX, \S X, \cite{Spanier} \S 6.6 Thm 2, Thm 6;
\cite{Dold} VIII \S 6.18, \S 13.4, \S 13.16; and \cite{Bredon}.
(By \cite{Spanier} \S 6.8 Cor. 8, the \Cech cohomology coincides with the
Alexander-Spanier cohomology on the class of paracompact Hausdorff spaces.)

For the remainder of this Appendix, continue with the notation $M, \L, F,
\Sigma$ of \S \ref{sec-freeloop}.
\begin{lem}\label{lem-cech} Let $G$ be an Abelian group and let $a \in \mathbb R.$
Then the natural homomorphisms
\[ H_*(\L;G) \to \H_*(\L;G) \ \text{ and }\
H_*(\L^{<a};G) \to \H_*(\L^{<a};G)\]
are isomorphisms.  If $a \in \mathbb R$ is a regular value of
$F,$ or if $a$ is a nondegenerate critical value of $F$ in the sense of Bott,
then the morphism $H_*(\L^{\le a};G) \to \H_*(\L^{\le a};G)$ is an isomorphism.
The same statements hold for \Cech cohomology.
\end{lem}
\proof  This follows from Proposition \ref{prop-fda}: the space
$\L^{<a}$ has the homotopy type of a finite dimensional manifold, and if $a$
is a regular value then $\L^{\le a}$ is homotopy equivalent to a
finite dimensional compact manifold with boundary.  \qed

\begin{lem} \label{lem-approximation}
If $a' < a \in \mathbb R$ then the inclusion $\L^{\le a}
\to \L^{\le a+\epsilon}$ induces canonical isomorphisms
\begin{eqnarray}
 \H_i(\L^{\le a};G) &\cong& \underset{0\longleftarrow\epsilon}{\lim}H_i(\L^{<a+
\epsilon};G)\label{eqn-lim1}\\
\H_i(\L^{\le a},\L^{\le a'};G) &\cong& \label{eqn-lim2}
\underset{0\longleftarrow\epsilon}{\lim}H_i(\L^{<a+\epsilon},\L^{< a'+\epsilon};G)
\end{eqnarray}
with \Cech homology on the left and singular homology on the right.  If $G$ is a
field, $\alpha \in H_i(\L;G)$ and if $a = \Cr(\alpha)$ is its critical value
(\S \ref{sec-cr-def}) then there exists $\omega \in \H_i(\L^{\le a};G)$ which
maps to $\alpha.$
\end{lem}
\proof
By Proposition \ref{prop-fda} the space $\L^{\le a}$ is homotopy equivalent to
the finite dimensional space $\M_N^{\le a}$ which is contained in a manifold.
Therefore
\[ \H_i(\L^{\le a})\cong H_i(\M_N^{\le a}) =
\underset{0\longleftarrow\epsilon}{\lim} H_i(\M_N^{< a+\epsilon}) \cong
\underset{0\longleftarrow\epsilon}{\lim} H_i(\L^{\le a+\epsilon})\]
which proves (\ref{eqn-lim1}).  The relative case (\ref{eqn-lim2}) is similar.  
Now suppose $G$ is a field and let $b_n\downarrow a=\Cr(\alpha)$
be a convergent sequence of regular values of $F.$ Then $\H_i(\L^{\le a};G)$ is
the limit of the sequence of finite dimensional vector spaces
\[ H_i(\L^{\le b_1}) \leftarrow H_i(\L^{\le b_2}) \leftarrow H_i(\L^{\le b_3}) 
\leftarrow \cdots\]
and for each $n\ge 1$ there is an element $\omega_n\in H_i(\L^{\le b_n})$ that
maps to $\alpha.$  Let $H_n = \text{Image}\left(
H_i(\L^{\le b_n};G) \to H_i(\L^{\le b_1};G)\right).$  These form a decreasing chain of
finite dimensional vector spaces which therefore stabilize after some finite point, say,
\[ H_N = \text{Image}\left(H_i(\L^{\le b_N};G) \to H_i(\L^{\le b_1};G)\right)
= \cap_{n=1}^{\infty} H_n = \H_i(\L^{\le a};G).\]
It then suffices to take $\omega \in H_N$ to be the image of $\omega_N \in 
H_i(\L^{\le b_N};G).$  \qed

\begin{prop}\label{prop-level}  
Fix $c\in \mathbb R.$  Let $U \subset \L$ be a neighborhood of $\Sigma^{=c}.$ 
The inclusions
\[ (\L^{<c}\cup  \Sigma^{=c})\cap U \hookrightarrow \L^{<c} \cup
\Sigma^{=c} \hookrightarrow \L^{\le c}\]
induce isomorphisms on \Cech homology,
{\footnotesize\begin{diagram}[height=2em]
\H_i((\L^{<c}\cup  \Sigma^{=c})\cap U, \L^{<c} \cap U;G) &\rTo&
\H_i(\L^{\le c}\cap U, \L^{< c}\cap U;G) \\
\dTo_{\beta} && \dTo_{\gamma} \\
\H_i(\L^{<c}\cup  \Sigma^{=c}, \L^{<c};G) & \rTo^{\alpha} & \H_i(\L^{\le c},
\L^{<c};G) &\rTo^{\tau}&\lim_{0 \leftarrow \epsilon}H_i(\L^{<c+\epsilon}, \L^{<c};G)
\end{diagram}} 
\end{prop}

\begin{proof}
It follows from excision that the relative homology group
 $\H_i((\L^{<c}\cup \ \Sigma^c)\cap U, \L^{<c} \cap U)$ is independent of $U.$
Taking $U=\L$ gives the isomorphism $\beta.$  The same argument applies to the
isomorphism $\gamma.$ The mapping $\tau$ is an isomorphism by \S
\ref{subsec-limits}.  Finally, the mapping $\alpha$ is an isomorphism because
the inclusion $(\L^{<c}\cup \Sigma^c, \L^{<c}) \to (\L^{\le c},\L^{<c})$ is a
homotopy equivalence.  A homotopy inverse is given by the time $t$ flow
$\psi_t:\L^{\le c} \to \L^{\le c}$ of the vector field $- \grad(F),$ for
any choice of $t>0$  (cf.~\cite{Klingenberg} \S 1, \cite{Chang} \S I.3).
\end{proof}

\section{Thom isomorphisms}
\subsection{}\label{subsec-tubes}  The constructions in this paper necessitate
the use of various relative versions of the Thom isomorphism for finite and
infinite dimensional spaces in singular and \Cech homology and cohomology.  
In this section we review these standard facts.

Recall that a neighborhood $N$ of a closed subset $X$ of a
topological space $Y$ is a {\em tubular neighborhood} if there exists a
finite dimensional (``normal'') real vector bundle $\pi:E \to X,$ and a
homeomorphism $\psi: E \to N\subset Y$  which takes the zero section to
$X$ by the identity mapping.
In this case, excising $Y-N$ gives an isomorphism
\[H(E,E-X;G)\cong H(N, N-X;G) \cong H(Y, Y-X;G)\]
where $H$ denotes either singular homology or cohomology (with coefficients in
an Abelian group $G$) and where $E-X$ is the complement of the zero section. Let
us take the coefficient group to be $G = \mathbb Z$ if the normal bundle $E$ is
orientable (in which case we fix an orientation), and $G = \mathbb Z/(2)$
otherwise. The Thom class
\[ \mu_E \in H^n(E,E-X;G)\]
is the unique cohomology class which restricts to the chosen homology generator
of each fiber $\pi^{-1}(x).$  The cup product with this class gives
the Thom isomorphism in cohomology,
\begin{equation}\label{eqn-Thom-cohomology}
 H^i(X;G) \cong H^i(E;G) \to H^{i+n}(E,E-X;G)\cong H^{i+n}(Y, Y-X;G)\end{equation}
and the cap product with this class gives the Thom isomorphism in homology,
\begin{equation}\label{eqn-Thom-homology}
 H_i(X;G) \cong H_i(E;G) \leftarrow H_{i+n}(E,E-X;G) \cong H_{i+n}(Y,Y-X;G).
\end{equation}
See [Spanier] Chapt.~5 Sec.~7 p.~259.  The same results hold for
\Cech homology and cohomology. We need to establish relative versions of
these isomorphisms.

As usual we take the coefficient group $G$ to be either
$\mathbb Z$ or $\Z/(2).$

\begin{prop}\label{prop-Thom}
 Let $A \subset X$ be closed subsets of a topological space $Y.$
Assume that $X$ has a tubular neighborhood $N$ in $Y$
corresponding to a homeomorphism $\phi:E \to N$ of a normal bundle
$E \to X$ of fiber dimension $n.$ If $E$ is orientable then choose an
orientation and set $G = \mathbb Z,$ otherwise set $G = \mathbb Z/(2).$
Then the Thom isomorphism induces an isomorphism
\begin{equation}\label{eqn-relative-Thom-cohomology}
 H^i(X,X-A;G) \cong H^{i+n}(Y,Y-A;G)\end{equation}
in singular cohomology, and an isomorphism
\begin{equation}\label{eqn-relative-Thom-homology}
H_i(X,X-A;G) \cong H_{i+n}(Y,Y-A;G)\end{equation}
in singular homology.  Taking $A=X$ gives Gysin homomorphisms
\begin{align}\label{eqn-Gysin-cohomology}
H^i(X;G) &\cong H^{i+n}(Y,Y-X;G) \to H^{i+n}(Y;G)\\
\label{eqn-Gysin-homology}
H_{i+n}(Y;G) &\to H_{i+n}(Y,Y-X;G) \cong H_i(X;G)\\
\intertext{denoted $h_!$ and $h^!$ respectively, where $h:X\to Y$
denotes the inclusion.  If $U\subset Y$ is open then taking 
$A = X - X\cap U$ gives Gysin homomorphisms}
H^i(X,X\cap U;G) &\cong H^{i+n}(Y,Y-A;G) \to H^{i+n}(Y,U;G)\label{rel-Gysin-cohomology} \\
H_{i+n}(Y,U;G) &\to H_{i+n}(Y,Y-A;G) \cong H_i(X,X \cap U;G)
\label{eqn-relative-Gysin} \end{align}
 \end{prop}

\proof  We will describe the argument for
(\ref{eqn-relative-Thom-cohomology}); the argument
for (\ref{eqn-relative-Thom-homology}) is the same, with the arrows reversed.
We suppress the coefficient group $G$ in order to simplify the notation
in the following argument.
Let $E^A = \pi^{-1}(X-A) \subset E$ and let $E^0 = E-X.$
The sets $E^0,E^A$ are open in $X$ so they form an
excisive pair\footnote{A pair $A,B\subset X$ is excisive if
$A \cup B = A^o \cup B^o$ where $A^o$ denotes the relative interior of $A$ in
$A\cup B$, cf.~\cite{Spanier} p. 188} giving the excision isomorphism
\[ \begin{CD}
H^i(E^0 \cup E^A, E^0) @>{\cong}>> H^i(E^A, E^0 \cap E^A).\\
\end{CD}\]
The cup product with the Thom class $\mu_E \in H^n(E, E^0; \mathbb Z)$ gives a
mapping
\[H^i(X,X-A) \cong H^i(E,E^A) \to H^{i+n}(E, E^0\cup E^A)\]
which we claim is an isomorphism.  This follows from the five lemma and the
exact sequence of the triple
\[ E^0 \subset (E^0 \cup E^A) \subset E.\]  In fact, the following diagram
commutes:

\begin{figure}[!h]
{\footnotesize\[ \begin{CD}
H^{i+n}(E, E^0 \cup E^A) @>>> H^{i+n}(E,E^0) @>>>
\begin{matrix} H^{i+n}(E^0 \cup E^A, E^0)&=\\
H^{i+n}(E^A, E^0 \cap E^A)\end{matrix}\\
@AAA @A{\cong}AA @A{\cong}AA \\
H^{i}(E, E^A) @>>> H^i(E) @>>> H^i(E^A) \\
@| @| @| \\
H^{i}(X,X-A) @>>> H^i(X) @>>> H^i(X-A)
\end{CD}\]} 
\caption{}
\end{figure}
\noindent
so the left hand vertical mapping is an isomorphism.  Since $A$ is closed in $Y$
we may excise $Y-N$ from $Y-A$ to obtain an isomorphism
\[ H^{i+n}(Y,Y-A)\cong H^{i+n}(N,N-A) \cong H^{i+n}(E, E^0 \cup E^A).\qedhere\]

\subsection{}
We will also need in \S \ref{sec-nilpotence} the following standard facts
concerning the Thom isomorphism.  Suppose $E_1\to A$ and $E_2 \to A$ are oriented
vector bundles of ranks $d_1$ and $d_2.$  If both are oriented let $G= \mathbb Z$
(and choose orientations of each), otherwise let $G = \mathbb Z/(2)$ be the
coefficient group for homology.  Let $E = E_1 \oplus E_2.$ The diagram of projections
\[
\begin{CD}
  E @>>> E_1 \\
@VVV @VV{\pi_1}V  \\
E_2 @>{\pi_2}>> A
\end{CD} \]
gives identifications $E \cong \pi_1^*(E_2) \cong \pi_2^*(E_1)$
of the total space $E$ as a vector bundle $\pi_1^*(E_2)$
over $E_1$ (resp. as a vector bundle $\pi_2^*(E_1)$ over $E_2$).
The Thom class $\mu_2 \in H^{d_2}(E_2, E_2 - A)$ pulls up to a class
\[
\pi_1^*(\mu_2) \in H^{d_2}(E, E - E_1)\]
and similarly with the indices reversed.  Then the relative cup product
\[
H^{d_1}(E, E-E_1) \times H^{d_2}(E, E-E_2) \to H^{d_1+d_2}(
E, (E-E_1) \cup (E-E_2)) = H^{d_1 + d_2}(E, E-A)\]
takes $(\pi_1^*(\mu_2), \pi_2^*(\mu_1))$ to the Thom class
\[
\mu_E = \pi_1^*(\mu_2) \smile \pi_2^*(\mu_1).\]
Consequently the Thom isomorphism for $E$ is the composition of the Thom
isomorphisms
\begin{diagram}[size=2em]
H^i(A) &\rTo_{\smile \mu_1}& H^{i+d_1}(E_1,E_1-A) & \rTo_{\smile \mu_2}
H^{i+d_1+d_2}(E,E-A).
\end{diagram}

\begin{cor}\label{cor-relative-coGysin}
In the situation of Proposition \ref{prop-Thom}, suppose that $A\subset X\subset Y$
are closed sets, that $A$ has a tubular neighborhood in $X$ with oriented normal
bundle $\Upsilon$ of rank $m,$ and suppose $X$ has a tubular neighborhood in $Y$ with
oriented normal bundle $E$ of rank $n.$ Then
$\mu_{E\oplus\Upsilon}=\pi_{\Upsilon}^*(\mu_E) \smile
\pi_E^*(\mu_{\Upsilon})$ is a Thom class in $H^{n+m}(E\oplus\Upsilon,
E\oplus\Upsilon - A)$ and the composition of Thom isomorphisms across the bottom,
in the following diagram
\begin{diagram}[size=2em]
H^i(X) && \rTo^{\cong} && H^{i+n}(Y,Y-X)\\
\dTo && && \dTo_{\psi}\\
H^i(A) & \rTo^{(\ref{eqn-Thom-cohomology})}_{\smile\mu_{\Upsilon}} &
H^{i+m}(X,X-A) & \rTo^{(\ref{eqn-relative-Thom-cohomology})}_{\smile \mu_E} &
H^{i+n+m}(Y,Y-A)
\end{diagram}
is the Thom isomorphism $\smile (\mu_{E\oplus\Upsilon}).$  The diagram gives
rise to a Gysin homomorphism $\psi:H^r(Y,Y-X) \to H^{r+m}(Y,Y-A)$  which
may be interpreted as the cup product with the Thom class
\[ \mu_{\Upsilon} \in H^m(\Upsilon,\Upsilon-A) \cong
H^m(X,X-A) \cong H^m(E, E- \pi_E^{-1}(A))\]
in the following sequence of homomorphisms
\begin{equation*}
{\small\begin{diagram}[size=2em]
H^{r}(Y,Y-X) && && H^{r+m}(Y,Y-A)\\
\dTo_{\cong} && &&  \uTo^{\cong}\\
H^{r}(E,E-X) &\rTo_{\cup \mu_{\Upsilon}}&
H^{r+m}\left(E, (E-\pi_E^{-1}(A))\cup(E-X)\right) & = &
H^{r+m}(E,E-A)\qedhere
\end{diagram}}\end{equation*}
\end{cor}

\section{Theorems of Morse and Bott} \label{sec-Bott-Morse}
\subsection{} We have the following theorems of Morse and
Bott \cite{Morse}, \cite{Bott2}, \cite{Milnor}, \cite{Bott1}, \cite{Long},
\cite{Klingenberg} Cor. 2.4.11, and \S 3.2, \cite{Rademacher}.
(By Lemma \ref{lem-cech} the homology groups that appear in the following
theorem may be taken to be either \v{C}ech\ or singular.)

\begin{thm}
\label{prop-Morse-lemma}\label{thm-Morse-Bott} Let $X$ be a Riemannian Hilbert
manifold and let $f:X \to\mathbb{R}$ be a smooth function that satisfies
condition C. Let $\Sigma$ be a finite dimensional connected nondegenerate 
critical submanifold in
the strong sense that the eigenvalues of $d^{2}f$ on the normal bundle of
$\Sigma$ are bounded away from 0. Let $\lambda< \infty$ be the index of
$\Sigma$ and let $d < \infty$ be the dimension of $\Sigma.$ Let $c =
f(\Sigma)$ be the critical value. Suppose there is a smooth connected 
manifold $V$ with
\[ \dim(V) = \dim(\Sigma) + \lambda\]
and smooth embeddings \begin{diagram}[size=2em]
\Sigma & \rTo_{\sigma}& V & \rTo_{\rho} X
\end{diagram}
so that $\rho\circ\sigma:\Sigma\to\Sigma$ is the identity and $f \circ\rho(x)
<c$ whenever $x\in V-\Sigma.$ Then $\rho$ induces an isomorphism
\begin{diagram}[height=1.5em]
H_i(V, V-\Sigma;G) &\rTo^{\cong}& H_i(X^{<c}\cup\Sigma,X^{<c}; G)
\end{diagram}
for any coefficient group $G$ and for all $i \ge0.$ In fact, $\rho$ induces
a local diffeomorphism of pairs $(V, V-\Sigma) \cong 
(\Sigma^-, \Sigma^- -\Sigma),$ where $\Sigma^-$ is defined below.

Composing with the Thom isomorphism (\ref{eqn-Thom-homology}) gives a 
further isomorphism 
\[H_i(V,V-\Sigma;G) \cong H_{i-\lambda}(\Sigma;G) \] 
where $G= \mathbb{Z}$ if the normal bundle of $\Sigma$ in $V$ is
orientable, and $G = \mathbb{Z}/(2)$ otherwise. \quash{In particular, the
fundamental class $[V,V-\Sigma] \in H_{d+\lambda}(V,V-\Sigma)$ (with $\mathbb
Z$ coefficients if $V$ is orientable, and with $\mathbb Z/(2)$ coefficients
otherwise) maps to the generator of the top dimensional level homology group
$H_{i+\lambda} (X,X-\Sigma)$ (with $\mathbb Z$ or $\mathbb Z/(2)$ coefficients
accordingly.}
\end{thm}

\subsection{Proof}
This essentially follows from Theorem 7.3 (p.~72) of \cite{Chang} or Corollary
2.4.8 and Proposition 2.4.9 of \cite{Klingenberg}. The tangent bundle
$TX|\Sigma$ decomposes into an orthogonal sum of vector bundles $\Gamma^{+}
\oplus\Gamma^{0} \oplus\Gamma^{-}$ spanned by the positive, null, and negative
eigenvectors (respectively) of the self adjoint operator associated to
$d^{2}f.$ The inclusion $\Sigma\to X$ induces an isomorphism $T\Sigma
\cong\Gamma^{0}$ so we may identify the normal bundle of $\Sigma$ in $X$ with
$\Gamma^{+} \oplus\Gamma^{-}.$

For $\epsilon$ sufficiently small the restriction of the exponential map
\[
\exp: (\Gamma^{+} \oplus\Gamma^{-})_{\epsilon} \to X
\]
is a homeomorphism onto some neighborhood $U\subset X.$ Let $\Sigma^{-} =
\exp(\Gamma^{-}_{\epsilon})\subset X.$ This submanifold is often described as
``the unstable manifold that hangs down from $\Sigma,$'' for if $\epsilon$ is
sufficiently small and if $0 \ne a \in\Gamma^{-}_{\epsilon}$ then
$f(\exp(a))<c.$ Its tangent bundle, when restricted to $\Sigma,$ is
\[
T\Sigma^{-}|\Sigma= \Gamma^{0} \oplus\Gamma^{-}.
\]

The projection $\Gamma^{+} \oplus\Gamma^{-} \to\Gamma^{-}$ induces a
projection $\pi:U \to\Sigma^{-}$ which is homotopic to the identity by the
homotopy
\[
\pi_{t}(\exp(a \oplus b)) = \exp(ta \oplus b)
\]
where $t\in[0,1]$, $a \in\Gamma^{+}_{\epsilon}$, $b \in\Gamma^{-}_{\epsilon},$
and where $\pi_{1}$ is the identity and $\pi_{0} = \pi.$ Moreover the kernel
of the differential $d\pi(x):T_{x}X \to T_{x}\Sigma^{-}$ at any point $x
\in\Sigma$ is precisely the positive eigenspace, $\Gamma^{+}_{x}\subset
T_{x}X.$ Let us identify the manifold $V$ with its image $\rho(V)\subset X$ so
that $TV|\Sigma\subset\Gamma^{0} \oplus\Gamma^{-}.$ It follows that the
restriction of $\pi$
\begin{equation}
\label{eqn-pi-V}\pi: V \cap U \to\Sigma^{-}%
\end{equation}
has nonvanishing differential at every point $x \in\Sigma\subset V$ and
consequently the mapping (\ref{eqn-pi-V}) is a diffeomorphism in some
neighborhood of $\Sigma.$ It follows that $\pi$ induces an isomorphism
\begin{equation}
\label{eqn-pistar}\pi_{*}:H_{i}(V,V-\Sigma) \to H_{i}(\Sigma^{-}, \Sigma^{-} -
\Sigma) \cong H_{i}( D\Gamma^{-},\partial D\Gamma^{-})
\end{equation}
where $D\Gamma^{-}$ denotes a sufficiently small disk bundle in $\Gamma^{-}$
and $\partial D\Gamma^{-}$ is its boundary. On the other hand, by Morse theory
(the above mentioned Theorem 7.3 of \cite{Chang} or Proposition 2.4.9 of
\cite{Klingenberg}), the space $U^{\le c +\delta}$ has the homotopy type of
the adjunction space $U^{\le c-\delta} \cup_{\partial D\Gamma^{-}}D\Gamma
^{-}.$ This gives the standard isomorphism of Morse theory,
\begin{equation}
\label{eqn-Morse-iso}H_{i}(D\Gamma^{-}, \partial D\Gamma^{-})\cong
H_{i}(U^{\le c+\delta}, U^{\le c-\delta}).
\end{equation}
All these isomorphisms fit together in a commutative diagram:
{\footnotesize \begin{diagram}
H_i(U^{\le c},U^{<c}) & \rTo^{\pi_*} & H_i(\Sigma^-, \Sigma^- -\Sigma) & \rTo^{i_*} &
H_i(U^{\le c}, U^{<c}) & \rTo^{(\ref{prop-level})}\ & H_i(U^{<c} \cup \Sigma, U^{<c}) \\
\uTo^{\rho_*} &\ruTo^{(\ref{eqn-pistar})}& \uTo &\ruTo& \dTo_{(\ref{prop-level})}\\
H_i(V,V-\Sigma) & \rTo_{(\ref{eqn-pistar})}& H_i(D\Gamma^-,\partial D\Gamma^-)
& \rTo_{(\ref{eqn-Morse-iso})} & H_i(U^{\le c+\delta}, U^{\le c-\delta})
\end{diagram}}
Each of the arrows labeled by an equation number is an isomorphism, so $i_{*}$
is an isomorphism. But $i_{*}\pi_{*}$ is the identity, so $\pi_{*}$ is also an
isomorphism, hence also $\rho_{*}.$ \qed

\section{Proof of Theorem \ref{thm-X}}\label{sec-proof-of-theoremX}
\subsection{}  The proof of Theorem \ref{thm-X} involves a second construction of the
cohomology bracket, along the same lines as the definition of the $\circledast$
product.  As in \S \ref{sec-freeloop} let $\L$ be the free loop space of 
mappings $x:\mathbb R/\mathbb Z \to M$ (or $x:[0,1] \to M$).
For the purposes of this appendix only, let $\widehat{\L}$ be the free
loop space of $H^1$ mappings $\mathbb R/2\mathbb Z \to M$ (or $[0,2] \to M$).  
If $x,y\in \L,$ if $s \in \mathbb R/\mathbb Z$ (or $s \in [0,1]$) and if 
$x(0)=y(s),$ define $x\cdot_sy \in \widehat{\L}$ (see Figure \ref{fig-join}) by
\begin{equation*}
x\cdot_sy(t) = \begin{cases}
y(t) &\text{if }\ 0 \le t \le s \\
x(t-s) &\text{if }\ s \le t \le 1+s\\
y(t)=y(t-1)&\text{if }\ 1+s \le t \le 2
\end{cases}. \end{equation*}

\begin{center}\begin{figure}[!h]
\setlength{\unitlength}{1.7pt}
\begin{picture}(100,90)
\put(40,40){\circle{56.568}}
\put(70,70){\circle{28.284}}
\put(40,11.716){\circle*{4}}
\put(60,60){\circle*{4}}
\put(40,3){\makebox(0,0){$y(0)$}}
\put(8,40){\makebox(0,0){$y$}} 
\put(70,87){\makebox(0,0){$x$}}
\put(53,53){\makebox(0,0){$y(s)$}}
\put(65,65){\makebox(0,0){$x(0)$}}
\put(20,20){\vector(-1,1){0}}\put(20,60){{\vector(1,1){0}}}
\put(60,20){\vector(-1,-1){0}}
\put(60,80){\vector(1,1){0}}
\put(80,80){\vector(1,-1){0}} \put(80,60){\vector(-1,-1){0}}
\end{picture}\caption{Joining two loops at time $s$}
\label{fig-join}
\end{figure}
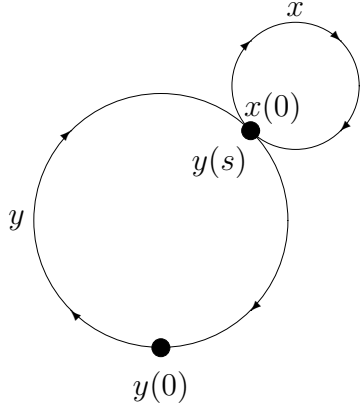\end{center}
\setlength{\unitlength}{1pt}

Define $\LL$ to be the set of triples $(x,y,s) \in
\L \times\L \times \RZ$ such that
\[
\begin{cases}
x(0) = y(s) &\text{if }\ 0 \le s \le 1 \\
y(0)= x(s) &\text{if }\ 1 \le s \le 2 \end{cases}\]
Define $\Phi_1:\LL\to\widehat{\L}$ by
$\Phi_1(x,y,s) =
\begin{cases} x\cdot_s y &\text{if }\ 0 \le s \le 1 \\
y\cdot_{(s-1)} x &\text{if } 1 \le s \le 2. \end{cases}$

\noindent
We have embeddings
\begin{equation}\label{eqn-boxed-bracket}\boxed{
\begin{diagram}[size=2em]
\L \times \L \times \mathbb R/(2\mathbb Z) & \lTo^{h} & \LL & \rTo^{\Phi} &
\widehat{\L} \times \RZ
\end{diagram}}\end{equation}
\noindent
where $\Phi(x,y,s) = \left(\Phi_1(x,y,s),s\right).$  The images $\Phi(\LL)$
and $h(\LL)$ have normal bundles and tubular neighborhoods and in fact
they are given by the pull-back of the diagonal $\Delta$ under the mappings
\begin{diagram}[size=2em]
\L \times \L \times \RZ & \rTo^{\alpha} & M \times M & \lTo^{\beta} & \widehat{\L} \times \RZ\\
\uTo && \uTo && \uTo\\
h\left(\LL\right) & \rTo & \Delta & \lTo & \Phi\left( \LL \right)\\
\end{diagram}
where 
\[\alpha(x,y,s) = \begin{cases} (x(0),y(s)) &\text{if } 0 \le s \le 1 \\
(x(s),y(0)) &\text{if } 1 \le s \le 2 \end{cases}\]
and $\beta(w,s)=\left(w(s),w(s+1)\right).$  
(Each half of $\LL$ has a smooth tubular
neighborhood and normal bundle in $\L\times\L\times\RZ$ but there is a ``kink'' 
where the two halves are joined so we only obtain a topological tubular
neighborhood and normal bundle of $h(\LL).$)  In particular,
\begin{equation}\label{eqn-phi-LL}
\Phi\left(\LL\right) = 
\left\{(w,s)\in \widehat{\L} \times \RZ:\ w(s) = w(s\pm1)\right\} \end{equation}

\subsection{}\label{subsec-geometric-bracket}
We claim that the bracket $\left\{ x, y \right\}
\in H_{\deg(x)+\deg(y)-n+1}(\L)$ is obtained by passing from left to right
in (\ref{eqn-boxed-bracket}), i.~e., it is the image of 
$x\times y\times [\RZ]$ under the composition
\begin{diagram}[size=2em]
H_i(\L) \times H_j(\L)\times H_1(\RZ) & \rTo^{\epsilon\times}& 
H_{i+j+1}(\L \times \L \times \RZ)\\
&& \dTo^{h^!}_{(\ref{eqn-Gysin-homology})}\\
H_{i+j-n+1}(\widehat{\L})&\lTo^{\Phi_1}&H_{i+j+1-n}(\LL)\\
\end{diagram}
where $\deg(x)=i$, $\deg(y)=j$,  $\epsilon = (-1)^{n(n-j-1)}$, and
$[\RZ]\in H_1(\RZ)$ denotes the orientation class. First we show this
agrees with the definition of $\left\{x,y\right\}$ in \cite{CS}.

The projection $\pi: \LL \to \RZ$ 
is locally trivial and $\pi^{-1}(0) \cong \pi^{-1}(1) \cong \F$ is the figure eight
space.  Let $\LL_{[0,1]}=\pi^{-1}([0,1])$ and $\partial 
\LL_{[0,1]} = \pi^{-1}\left(\left\{0\right\} \cup \left\{1\right\} \right).$
Then the bracket product in \cite{CS} is a sum of two terms,
\[\left\{x,y\right\} = x \scs y - (-1)^{(i-n+1)(j-n+1)}y\scs x\] 
(but \cite{CS} use a $*$ rather than a $\scs$) which may be identified as the 
two images of 
\[h^!(\epsilon x\times y\times[\RZ])\in H_{i+j-n+1}(\LL)\] in
\[H_{i+j-n+1}\left(\Phi(\LL_{[0,1]}),\partial\Phi(\LL_{[0,1]})\right)\ \text{and}\ 
H_{i+j-n+1}\left(\Phi(\LL_{[1,2]}),\partial\Phi(\LL_{[1,2]})\right)\] 
respectively.  The projection to $\widehat{\L}$ adds these together (with the
appropriate sign).

The proof that the construction of \S \ref{subsec-geometric-bracket}
agrees with (\ref{eqn-CS-bracket}) is essentially
the same as the proof of Corollary 5.3 in \cite{CS}.  Using this fact, the proof of
(1),(2),(3) in \S \ref{subsec-CS-bracket} is then the same as in \cite{CS}
\S 4.

\subsection{} In this paragraph we define the reparametrization function 
\[\widehat{J}:\widehat{\L} \times \RZ \times [0,2] \to \widehat{\L} \times \RZ\] that is analogous
to the function $J$ of \S  \ref{subsec-first-cohomology-product}.  First some
notation.  For $r\in [0,2]$ let $\widehat\theta_{1\to r}:[0,2] \to [0,2]$ be
the piecewise linear function taking $0 \mapsto 0$, $1\mapsto r$, and $2 \mapsto 2.$
It is just the function $\theta$ of \S \ref{subsec-first-cohomology-product},
but the domain and range have been stretched to $[0,2].$ For any real number $s$
define the translation $\chi_s:\RZ \to RZ$ by $\chi_s(t) = t+s.$  Define
\begin{equation}\label{eqn-hat-J}
\widehat{J}(w,s,r) = \left(w\circ \chi_s \circ \theta_{1\to r} \circ \chi_{-s},
s\right)\end{equation}
and set $\widehat J_r(w,s) = \widehat{J}(w,s,r).$
In analogy with our notation for $\F$ in \S 
\ref{subsec-first-cohomology-product} let $\LL^{>0,>0}$ be the set of 
$(x,y,s)\in \LL$ such that $F(x)>0$ and $F(y)>0.$  Let $\widehat{\L}_0 =
\widehat{\L}^{=0}$ denote the constant loops in $\widehat{\L}.$  We claim that
$\widehat{J}$ takes both of the following sets
\begin{equation}\label{eqn-hatJ-set}
  \widehat{\L}_0 \times \RZ \times [0,2]\ \text{ and }\ \widehat{\L} \times \RZ \times \partial[0,2]
\end{equation}
into the set $\widehat{\L} \times \RZ - \Phi\left(\LL^{>0,>0}\right).$

This is obvious for the first of these sets, while the verification for the second set
involves four cases:  $s \le 1$ or $s\ge 1$ vs. $r=0$ or $r=2.$  In each case the
function $\chi_s\circ\widehat{\theta}_{1\to r} \circ \chi_{-s}$ is constant, with value
$s$, on the interval $[s,s+1] \mod 2.$ Therefore $J(w,s,r) = (\gamma,s)$ where
$\gamma = x\cdot_sy$ ($s \le 1$) or $\gamma =x \cdot_{(s-1)}y$ ($s \ge 1$)
and either $x$ or $y$ is a constant loop.

\subsection{}  \label{subsec-geometric-cobracket}
The geometric construction of the cohomology bracket is  the following composition,
\begin{center}
{\footnotesize\begin{diagram}[size=2em]
&H^i(\L,\L_0)\times H^j(\L,\L_0)\times H^0(\RZ) &\rTo^{\epsilon\times}&
H^{i+j}\left((\L,\L_0)\times(\L,\L_0)\times \RZ\right)\\
&&&\dTo^{h^*}&\\
&&&H^{i+j}\left( \LL, \LL - \LL^{>0,>0}\right) \\
&&& \dTo^{\Phi_{!}}_{(\ref{eqn-relative-Thom-cohomology})}& \\
&&& H^{i+j+n}\left(\widehat{\L}\times \RZ, \widehat{\L}\times\RZ - \Phi\left(\LL^{>0,>0}\right)\right) \\
&&&\dTo^{\widehat{J}^*}&\\
&H^{i+j+n-2}\left(\widehat{\L},\widehat{\L}_0)\right) &\lTo^{\pi_{!}}&
H^{i+j+n}\left((\widehat{\L},\widehat{\L}_0) \times \RZ \times \left([0,2], \partial[0,2]\right) \right)
\end{diagram}}
\end{center}
where $\pi$ denotes the projection to $\widehat{\L}.$  So the bracket is obtained by passing
from left to right in the following diagram,
\begin{equation*}\boxed{
\begin{diagram}[size=2em]
\L\times\L \times \RZ & \lTo^h& \LL &\rTo^{\Phi}& \widehat{\L}\times \RZ 
&\lTo^{\widehat{J}}& \widehat{\L}\times\RZ\times [0,2] &\rTo& \widehat{\L}
\end{diagram}
}\end{equation*}

\quash{
\begin{equation*}\boxed{
\begin{diagram}[size=2em]
(\L,\L_0)\times(\L,\L_0)\times \RZ &\lTo_h&
\left( \LL, \LL- \LL^{>0,>0}\right) \\
&&\dTo^{\Phi}&\\
\left(\widehat{\L},\widehat{\L}_0\right)\times \RZ \times \left([0,2], \partial[0,2]\right)
 &\rTo^{\widehat{J}}&
(\widehat{\L}, \widehat{\L}- \LL^{>0,>0})\times \RZ 
\end{diagram}
}\end{equation*}
}

\subsection{}  For $r \in [0,2]$ set $\widehat{J}_r(w,s) = \widehat{J}(w,s,r).$
Let $\mathcal T: \LL \to \LL$ by $T(x,y,s) = (y,x,s+1)$ and (by
abuse of notation) set $\mathcal T:\widehat{\L}\times\RZ 
\to \widehat{\L}\times \RZ$ by 
$\mathcal T(z,s) = (z\circ\chi_1,s+1).$ (So
$\mathcal T$ moves the basepoint half way around the loop.) Then the following
diagram commutes:
\begin{diagram}[size=2em]
\LL &\rTo^{\Phi}& \widehat{\L} \times \RZ &\rTo^{\widehat{J_r}}& \widehat{\L} \times \RZ \\
\dTo_{\mathcal T} && \dTo_{\mathcal T} && \dTo_{\mathcal T} \\
\LL &\rTo^{\Phi} & \widehat{\L} \times \RZ &\rTo^{\widehat{J_r}}& \widehat{\L} \times \RZ.
\end{diagram}
 As in \S \ref{subsec-geometric-bracket} the bracket is a sum of two terms,
\[ \left\{x,y\right\} = x \clubsuit y -(-1)^{(|x|+1)(|y|+1)} y \clubsuit x\]
which are interchanged by the involution $\mathcal T.$  The proof that the 
construction in \S \ref{subsec-geometric-cobracket} agrees
with the formula (\ref{eqn-cohomology-bracket}) is essentially the same as the proof of
Corollary 5.3 in \cite{CS}.  The proof of (A), (B), (C) in Theorem \ref{thm-X}
\ref{subsec-CS-bracket} is then similar to the argument in \cite{CS} \S 4.

\section{Associativity of $\circledast$}\label{sec-associative}
\subsection{}
In this section we complete the proof of Proposition
\ref{prop-cohomology-commutative}.
The following statements refer to the diagram ``Associativity of $\circledast$".  
We omit the parallel diagram that is obtained by taking cohomology of each
of the spaces and pairs that appear in this diagram.
Each mapping denoted $\tau$ denotes an inclusion with normal bundle.
The corresponding homomorphism in the cohomology diagram is the Thom
isomorphism.  In the cohomology diagram, the squares involving arrows
denoted $\tau$ commute because the relevant normal bundles pull back.

\newcommand{\third}{\frac{1}{3}}
\newcommand{\twoth}{\frac{2}{3}}
\newcommand{\id}{{\sf id}}
\newcommand{\gogo}{>0,>0} \newcommand{\gogogo}{>0,>0,\ge 0}
\newcommand{\wow}[2]{{\begin{subarray}{l} \frac{1}{3}\to {#1}\\
\frac{2}{3} \to {#2}\end{subarray}}}

\begin{sidewaysfigure}[!p]
{\scriptsize
\begin{diagram}[size=2em,PostScript=dvips]
(\L,\L_0)\times(\L,\L_0)\times(\L,\L_0) & \lTo_{\id \times \Delta} & (\L,\L_0)\times
(\F_{\twoth},\F_{\twoth}-\F_{\twoth}^{\gogo}) & \rTo_{\id \times\tau} &
(\L,\L_0)\times(\L,\L-\F_{\twoth}^{\gogo}) &\lTo_{\id \times J_2}& (\L,\L_0)\times(\L,\L_0)
\times (I_2,\partial I_2) \\
\uTo_{\Delta\times\id} &\boxed{1}& \uTo && \uTo && \uTo_{\Delta\times\id} \\
(\F_{\third},\F_{\third}-\F_{\third}^{\gogo})\times(\L,\L_0) &\lTo&
(\mathcal C,\mathcal C - \mathcal C^{\gogogo})
&\rTo_{\tau}& (\F_{\third},\F_{\third}-\mathcal C^{\gogogo}) & \lTo_{J_2}&
(\F_{\third}, \F_{\third} - \F_{\third}^{\gogo}) \times (I_2, \partial I_2)\\
\dTo_{\tau\times\id} && \dTo_{\tau} && \dTo_{\tau} && \dTo_{\tau\times\id}\\
(\L,\L-\F_{\third}^{\gogo})\times (\L,\L_0) &\lTo& (\F_{\twoth},\F_{\twoth}-
\mathcal C^{\gogogo})
&\rTo_{\tau}& (\L,\L-\mathcal C^{\gogogo}) &\lTo_{J_2}& (\L,\L-\F_{\third}^{\gogo}) \times (I_2,
\partial I_2) \\
\uTo_{J_1\times\id} && \uTo_{J_1} && \uTo_{J_1} &\boxed{2}&
\uTo_{\widehat{J}_1\times\id} \\
 (I_1, \partial I_1)\times(\L,\L_0) \times (\L,\L_0)  & \lTo &
(I_1, \partial I_1)\times (\F_{\twoth},\F_{\twoth}-\F^{\gogo})  &\rTo_{\tau}&
(I_1,\partial I_1)\times (\L,\L-\F_{\twoth}^{\gogo}) &\lTo_{\id\times\widehat{J}_2}&
(I_1, \partial I_1)\times(\L,\L_0)   \times (I_2, \partial I_2)
\end{diagram}}
\caption{Associativity of $\circledast$ product}
\end{sidewaysfigure}

Starting with $x\times y \times z$ in the upper left corner, the product
$x\circledast(y\circledast z)$ is obtained by going across the top row then
down the right side of the diagram, while the product $(x\circledast y)
\circledast z$ is obtained by going down the left side of the diagram
and then along the bottom row.  Here, the symbol $\F_{\third}$ denotes
$\phi_{\third}(\F))$ and the space $\mathcal C$ denotes the space of
(three-leaf) clovers, that is, the pre-image of the (small) diagonal
under the mapping
\[(\ev_0,\ev_{\third},\ev_{\twoth}):\L \to M\times M \times M.\]
It has a normal bundle in $\L$ that is isomorphic
to $TM \oplus TM.$  The symbol $\mathcal C^{\gogogo}$ denotes those loops
consisting of three composable loops $\alpha\cdot\beta\cdot\gamma,$
with $\alpha,\beta$ glued at time $1/3$ and with $\beta,\gamma$ glued at
time $2/3,$ such that two (or more) of these ``leaves'' have positive energy
(that is, one or fewer of these loops is constant). Hence $\mathcal C -
\mathcal C^{\gogogo}$ consists
of clovers such that two or more of the leaves are constant.  The square marked
$\boxed{1}$ is Cartesian:  the lower right corner is the intersection of
the upper right and lower left corners.  The symbol $\Delta$ denotes a
diagonal mapping and $\id$ denotes an identity mapping.

\subsection{}
Using the obvious extension of the notation for $\theta_{\half\to s}:I \to I$
(with $\theta(0)=0$ and $\theta(1)=1$), the mappings $J_i$ and $\widehat{j}_i$ are
(re)defined by
\begin{align*}
J_1(s,\gamma) &= \gamma \circ \theta_{\wow{\frac{2}{3}s}{\frac{2}{3}}}
&J_2(\gamma,t) &= \gamma\circ \theta_{\wow{\third}{\third+\twoth t}} \\
\widehat{J}_1(s,\gamma) &= \gamma \circ \theta_{\third \to s}
&\widehat{J}_2(\gamma,t) &= \gamma \circ \theta_{\twoth \to t}
\end{align*}
so that
\begin{align}
J_1 \circ (\id\times\widehat{J}_2)(s,\gamma,t) &= \gamma\circ \theta_{\twoth \to t}
\circ\theta_{\wow{\twoth s}{\twoth\phantom{\third+\twoth t}}} =
\gamma \circ \theta_{\wow{st}{t}}\label{eqn-theta1}\\
J_2 \circ (\widehat{J}_1\times\id)(\gamma,s,t) &= \gamma\circ\theta_{\third\to s}
\circ \theta_{\wow{\third}{\third+\twoth t}} =
\gamma \circ\theta_{\wow{s}{s+(1-s)t}}\label{eqn-theta2}
\end{align}

We need to prove that the corresponding cohomology diagram commutes.  The only
part that is not obvious is the square designated $\boxed{2}$ in the diagram.
This square commutes up to (relative) homotopy for the following reason.  Let
$\mathfrak m$ denote the set of continuous non-decreasing mappings
$\theta:[0,1] \to [0,1]$ such that $\theta(0)=0$, $\theta(1)=1$, and
$\theta$ is linear on $[0,{\third}]$, on $[\third,\twoth]$, and on $[\twoth, 1].$
For $i = 1,2,3$ let $\mathfrak m_i$ denote the collection of all $\theta\in
\mathfrak m$ such that $\theta$ is constant on $\left[\frac{i-1}{3}, \frac{i}{3}
\right].$  The functions $\theta_{etc.}$ appearing on the right side of
(\ref{eqn-theta1}) and (\ref{eqn-theta2}) may be interpreted as continuous
mappings $(s,t)\in I_1 \times I_2 \to \mathfrak m$  with the following boundary behavior:
$\{0\} \times I_2 \to \mathfrak m_1,$ $\{1\} \times I_2 \to \mathfrak m_2;$
$I_1 \times \{0\} \to \mathfrak m_2;$ and $ I_1 \times \{1\} \to
\mathfrak m_3.$  This boundary behavior is indicated in Figure \ref{fig-boundary}.

\begin{figure}[!h]\begin{center}
\begin{picture}(120,115)(0,0)
\put(0,20){\vector(1,0){100}} \put(105,15){$s$}
\put(20,0){\vector(0,1){100}} \put(15,105){$t$}
\put(20,70){\line(1,0){50}}
\put(70,20){\line(0,1){50}}
\put(40,10){$\mathfrak m_2$}
\put(0,40){$\mathfrak m_1$}
\put(75,40){$\mathfrak m_2$}
\put(40,75){$\mathfrak m_3$}
\end{picture}\end{center}\caption{Boundary behavior}
\label{fig-boundary}
\end{figure}
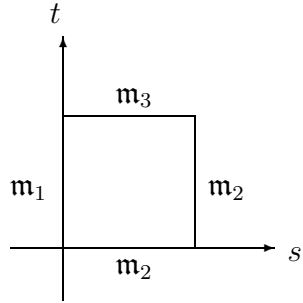

But the collection of such maps $I_1 \times I_2 \to \mathfrak m$ is
convex, so the mappings (\ref{eqn-theta1}) and (\ref{eqn-theta2}) are homotopic.
This completes the proof that the $\circledast$ product is associative.\qedhere
\end{proof}


\begin{thebibliography}{[BRS]}

\bibitem[AS1]{AS1} B. Abbondandolo and M. Schwarz, Notes on Floer homology and loop space homology, in
{\bf Morse theoretic methods in nonlinear analysis and in symplectic topology}, 
Proceedings of the NATO Advanced Study Institute, Montral,
Canada, July 2004, Springer, Dordrecht, 2006.

\bibitem[AS2]{AS2} B.~Abbondandolo and M.~Schwarz, On the Floer homology of cotangent bundles,
Comm. Pure. Appl. Math. {\bf 59} (2005), 254-316.

\bibitem[A]{Anosov} D.~Anosov, Certain homotopies in the space of closed curves.
Math. USSR Izvestiya {\bf 17} (1981), 423-453.


\bibitem[Ba]{Baas} N.~Baas, On bordism theories of manifolds with singularities,
Math. Scand.,33 (1973), 279-302.

\bibitem[BC]{Bahri} A.~Bahri and F.~R.~Cohen, On ``small geodesics'' and free loop spaces,
preprint/work in progress

\bibitem[BO]{Baker} A.~Baker and C.~\"Ozel, Complex cobordism of Hilbert
manifolds with some applications to flag varieties and loop groups.  in {\bf
Geometry and Topology:  Aarhus 1998},1-19, Contemp. Math. {\bf 258}, Amer. Math.
Soc., Providence RI, 2000.

\bibitem[Bal]{Ballman} W.~Ballman, G.~Thorbergsson and W.~Ziller, Closed
geodesics on positively curved manifolds, Ann. Math. {\bf 116} (1982), 213-247.

\bibitem[Ban]{Bangert} V.~Bangert, On the existence of closed geodesics on
two-spheres, Int. J. Math. {\bf 4} (1993), 1-10.

\bibitem[Be]{Besse} A.~Besse, {\bf Manifolds all of whose Geodesics are Closed},
Ergebnisse der Mathematik {\bf 93}, Springer-Verlag, Berlin, 1978.

\bibitem[BH1]{Birkhoff1} G.~Birkhoff and M.~Hestenes, Generalized minimax
principle in the calculus of variations,
Proc Natl Acad Sci U.S.A. 1935 February; 21(2): 96-99.


\bibitem[BH2]{Birkhoff2} G.~Birkhoff and M. Hestenes, Generalized minimax
principle in the calculus of variations,  Duke Math. J.  1, no. 4 (1935), 413-432

\bibitem[BM]{BorelMoore} A.~Borel and J.~C.~Moore, Homology theory for
locally compact spaces, Michigan Math. J. {\bf 7}, (1960), 137-159.

\bibitem[Bo1]{Bott1} R.~Bott, On the iteration of closed geodesics and the
Sturm intersection theory, Comm. Pure Appl. Math. {\bf 9} (1956), 171-206.

\bibitem[Bo2]{Bott2} R.~Bott, {\bf Morse theory and its application to
homotopy theory, Lectures delivered in Bonn, 1958}, notes by A. van de Ven.
Harvard lecture notes;  Reprinted in
{\bf Collected Papers}, Vol. 1-4, Birkh\"auser, Boston, 1994, 1995.


\bibitem[Br]{Bredon} G.~Bredon, {\bf Sheaf Theory} (second edition), Graduate
Texts in Mathematics {\bf 170}, Springer Verlag, N.~Y., 1997.

\bibitem[Br2]{Bredon2} G.~Bredon, {\bf Topology and Geometry}, Springer Verlag,
New York, 1993.

\bibitem[BRS]{BRS} S.~Buoncristiano, C.~Rourke and B.~Sanderson, {\bf A
geometric
approach to homology theory}, L. M. S. Lecture Notes 18, Cambridge Univ. Pr.,
Cambridge, 1978.

\bibitem[C]{Chang} K.~-C.~Chang, {\bf Infinite Dimensional Morse Theory and
Multiple Solution Problems}, PNLDE {\bf 6}, Birkh\"auser Boston, 1993.

\bibitem[Chs]{Chas} M.~Chas, Combinatorial Lie bialgebras of curves on surfaces,
Topology {\bf 43} (2004), 543-568.

\bibitem[CS]{CS} M.~Chas and D.~Sullivan, String topology,
preprint, math.GT/9911159 (1999).

\bibitem[Cht]{Chataur} D.~Chataur, A bordism approach to string topology,
preprint, math.AT/0306080.

\bibitem[Co]{Cohen2} R.~Cohen, {\bf Homotopy and geometric perspectives on
string topology}, Lecture notes, Stanford University, 2005.


\bibitem[CHV]{CHV} R.~Cohen, K.~Hess, and A.~Voronov, {\bf String Topology
and Cyclic Homology}, Birkhauser, Basel, 2006.

\bibitem[CJ]{Cohen1} R.~Cohen and J.~Jones, A homotopy theoretic realization of
string topology, Math. Ann. {\bf 324} (2002), 773-798.


\bibitem[CJY]{CJY} R.~Cohen, J.~Jones, and J.~Yan, The loop homology algebra of
spheres and projective spaces, Progr. Math. {\bf 215}, Birkhauser, Basel 2003,
77-92.

\bibitem[CKS]{Cohen} R.~Cohen, J.~Klein, and D.~Sullivan, The homotopy
invariance of the
string topology loop product and string bracket, math.GT/0509667.

\bibitem[D]{Dold} A.~Dold, {\bf Lectures on Algebraic Topology}, Grundlehren
Math.
{\bf 200}, Springer Verlag, Berlin, 1972.

\bibitem[ES]{Eilenberg} S.~Eilenberg and N.~Steenrod, {\bf Foundations of
Algebraic Topology}, Princeton University Press, 1952, Princeton, NJ.

\bibitem[F]{Franks} J.~Franks, Geodesics on $S^2$ and periodic points of
annulus homeomorphisms, Invent. Math. {\bf 108} (1992), 403-418.

\bibitem[GeM]{GelfandManin} S.~I.~Gelfand and Yu.~I.~Manin, {\bf Methods
of Homological Algebra}, second edition, Springer Verlag, Berlin, 2002.

\bibitem[GeM2]{GelfandManin2} S.~I.~Gelfand and Yu.~I.~Manin, {\bf
Homological Algebra}, Encyclopedia of Mathematics {\bf 38}, {\bf Algebra V},
Springer Verlag, Berlin, 1994.

\bibitem[Go]{Goldman} W.~Goldman, Invariant functions on Lie groups and Hamiltonian
flows of surface group representations, Inv. Math. {\bf 85} (1986), 263-302.

\bibitem[Gor]{goresky} M.~Goresky, Triangulation of stratified objects, Proc.
Amer. Math. Soc. {\bf 72} (1978), 193-200.

\bibitem[GoM]{IH2} M.~Goresky and R.~MacPherson, Intersection Homology II,
Inv. Math. {\bf 71} (1983), 77-129.

\bibitem[GrM]{Gromoll} D.~Gromoll and W. Meyer, Periodic geodesics on compact
Riemannian manifolds, J. Diff. Geom. {\bf 3} (1969), 493-510.

\bibitem[Har]{Hardt1} R.~Hardt, Stratification of real analytic mappings and
images. Invent. Math. {\bf 28} (1975), 193--208.
\bibitem[Ha]{Hatcher} A.~Hatcher, {\bf Algebraic Topology}, Cambridge University
Press, Cambridge, UK, 2002.

\bibitem[Hi1]{Nancy1} N.~Hingston, On the lengths of closed geodesics on a
two-sphere, Proc. Amer. Math. Soc. {\bf 125} (1997), 3099-3106.

\bibitem[Hi2]{Nancy2} N.~Hingston, On the growth of the number of closed
geodesics on the two-sphere, Int. Math. Res. Not. {\bf 9} (1993), 253-262.


\bibitem[Hir]{Hironaka} H.~Hironaka, Subanalytic sets. Number theory, algebraic
geometry and commutative algebra, in honor of Yasuo Akizuki, pp. 453--493.
Kinokuniya, Tokyo, 1973.

\bibitem[Hs]{Hirsch} M.~Hirsch, On normal microbundles, Topology {\bf 5}
(1966), 229-240.

\bibitem[Ho]{Holm} P.~Holm, Microbundles and Thom classes, Bull. Amer.
Math. Soc. {\bf 72} (1966), 549-554.


\bibitem[I]{Iverson} B.~Iverson, {\bf Cohomology of Sheaves}, Universitext,
Springer Verlag, Berlin, 1986.


\bibitem[J]{Johnson} F.~E.~A.~Johnson, On the triangulation of stratified sets and singular
varieties, Trans. Amer. Math. Soc. {\bf 275} (1983), 333-343.

\bibitem[K]{Klingenberg} W.~Klingenberg, {\bf Lectures on Closed Geodesics},
Grundlehren der mathematischen Wissenschaften {\bf 230}, Springer Verlag,
Berlin, 1978.

\bibitem[L]{Long} Y.~Long, {\bf Index Theory
for Symplectic Paths with Applications}  Progress in Mathematics {\bf 207},
Birkh\"{a}user, Basel,  2002.

\bibitem[Mi]{Milnor} J.~Milnor, {\bf Morse Theory}, Annals of Mathematics
Studies {\bf 51}, Princeton University Press, Princeton N.J., 1963.

\bibitem[Mi2]{Milnor2} J.~Milnor, Microbundles, Part I, Topology
{\bf 3}, supp. 1 (1964), 53-80.

\bibitem[Mo1]{Morse} M.~Morse, {\bf Calculus of Variations in the Large},
Amer. Math. Soc. Colloquium Publications, XVIII, Providence, R.I., 1934.

\bibitem[Mo2]{Morse2} M.~Morse, {\bf Functional Topology and Abstract
Variational
Theory}, Memoriale des Sciences Math. {\bf 92}, Gauthier-Villars, Paris, 1939.

\bibitem[P]{Pope} D.~Pope, On the approximation of function spaces in the
calculus of variations, Pacific J. Math. {\bf ?} (1962), 1029-1045.

\bibitem[R]{Rademacher} H.~-B.~Rademacher, On the average indices of closed
geodesics, J. Diff. Geom. {\bf 29} (1989), 65-83.

\bibitem[SaW]{SW} D.~Salamon and J.~Weber, Floer homology and the heat flow,
Geometric and Functional Analysis (GAFA) {\bf 16} (2005), 1050-1138.


\bibitem[Sp]{Spanier} E.~Spanier, {\bf Algebraic Topology}, McGraw-Hill,
New York, 1966.

\bibitem[Su]{Sullivan} D.~Sullivan, Open and closed string field theory
interpreted in classical algebraic topology, in {\bf Topology, Geometry
and Quantum Field Theory, Procdeedings of the 2002 Oxford Symposium}, London
Mathematical Society Lecture Note Series {\bf 308}, Cambridge University
Press, Cambridge, 2004; pages 344-357.

\bibitem[T]{Turaev} V.~Turaev, Skein quantization of Poisson algebras
of loops on surfaces, Ann. Sci. \'Ecole Norm. Sup. {\bf 24} (1991), 
635-704.

\bibitem[V]{Verdier} J.~L.~Verdier, {\bf Des Cat\'egories D\'eriv\'ees
des Cat\'egories Ab\'eliennes}, 1963, reprinted in Ast\'erisque {\bf 239},
Soc. Math. de France, Paris, 1996.

\bibitem[Vi]{Viterbo} C. Viterbo, Functors and computations in Floer
homology with applications II, preprint, 1996, revised 2003.

\bibitem[VS]{Vigue} M. Vigu\'e-Poirrier and D. Sullivan, The homology theory
of the closed geodesic problem, J. Diff. Geom. {\bf 11} (1976), 633-644.

\bibitem[Wh]{Whitehead} G.~Whitehead, Generalized Homology Theories,
Trans.  Amer. Math. Soc. {\bf 102} (Feb., 1962) , pp. 227-283.

\bibitem[Wi]{Wilking} B.~Wilking, Index parity of closed geodesics and rigidity
of Hopf fibrations, Inv. Math. {\bf 144} (2001), 281-295.

\bibitem[Wo]{Wolpert} S.~Wolpert, On the symplectic geometry of
deformations of hyperbolic surfaces, Ann. Math. {\bf 117} (1983),
207-234.

\bibitem[Z]{Ziller} W.~Ziller, Geometry of the Katok examples.
Ergod.Th. \& Dyn.Syst. 3 (1982)

\end{thebibliography}
\end{document}